\documentclass[11pt]{article}

\usepackage[margin=1in]{geometry} 
\usepackage{amsmath, amsthm, amssymb, scrextend, mathtools, tikz,xcolor,esint, fancyhdr, relsize, graphicx, subcaption, float}
\usepackage{hyperref}
\usepackage{chngcntr}
\numberwithin{equation}{section}
\numberwithin{figure}{section}
\numberwithin{table}{section}

\usepackage{color}

\definecolor{armygreen}{rgb}{0.29, 0.33, 0.13}

\tikzset{fontscale/.style = {font=\relsize{#1}}}
\allowdisplaybreaks
\theoremstyle{definition}

\newcommand{\Natural}{\mbox{$\mathrm{I\!N}$}}
\newcommand{\real}{\mbox{$\mathbb{R}$}}
\newcommand{\alpot}{\ensuremath{\frac{\alpha}{2}}}

\newcommand{\wt}{\widetilde}
\renewcommand{\epsilon}{\varepsilon}
\renewcommand{\phi}{\varphi}
\renewcommand{\tilde}{\widetilde}
\newcommand{\R}{\mathbb{R}}
\newcommand{\N}{\mathbb{N}}
\newtheorem{theorem}{Theorem}[section]
\newtheorem{corollary}{Corollary}[section]
\newtheorem{lemma}{Lemma}[section]
\newtheorem{example}{Example}[section]

\newtheorem{defn}{Definition}[section]
\newcommand{\grad}{\nabla}
\newcommand{\st}{\ |\ }
\newcommand{\pf}{\indent\begin{proof}}
\newcommand{\epf}{\end{proof}}
\newcommand{\eps}{\varepsilon}

\newcommand{\alt}{\ensuremath{\frac{\alpha}{2}}}
\newcommand{\dt}{\Delta t}
\newcommand{\ra}{\rightarrow}
\newcommand{\vertiii}[1]{{\left\vert\kern-0.25ex\left\vert\kern-0.25ex\left\vert #1 
    \right\vert\kern-0.25ex\right\vert\kern-0.25ex\right\vert}}
\newcommand{\U}[1]{\mathcal{U}^{(#1)}}
\renewcommand{\P}{\mathcal{P}}
\renewcommand{\bf}[1]{\mathbf{#1}}
\newcommand{\mc}[1]{\mathcal{#1}}
\newcommand{\mcV}{\ensuremath{\mathcal{V}}}
\newcommand{\mcF}{\ensuremath{\mathcal{F}}}

\def\qed{\hbox{\vrule width 6pt height 6pt depth 0pt}}

\begin{document}
\title{A time dependent fractional order diffusion equation with constant diffusivity matrix}
\author{	
	  T.~Catoe\thanks{School of Mathematical and Statistical Sciences,
	  Clemson University, Clemson, South Carolina 29634-0975, USA.
	  email: {\tt tcatoe@clemson.edu}. } 
	  \and
	V.J.~Ervin\thanks{School of Mathematical and Statistical Sciences,
	  Clemson University, Clemson, South Carolina 29634-0975, USA.
	  email: {\tt vjervin@clemson.edu}. } }
\date{\today}
\maketitle

\begin{abstract}
Of primary interest in this paper is the numerical approximation of a time dependent fractional, in space, diffusion equation
where the domain is assumed to be nonhomogeneous, having different axial diffusion coefficients. This work is motivated
from the consideration of composite material which can exhibit different material properties along, and perpendicular to,
internal planar structures. Careful attention is paid to accurately capture the boundary behavior of the solution.
A spectral approximation scheme is used for the spatial discretization and a backward Euler approximation used for the
temporal discretization. Following an error analysis for the approximation scheme, numerical experiments are given
to demonstrate the effects of the nonhomogeneous domain and to support the theoretical analysis.
\end{abstract}



\section{Introduction}
\label{SPFD_intro}
In recent years, fractional differential equations have become increasingly popular in modeling physical phenomena which are not well captured by classical integer-order differential equations. A vast variety of application areas have already begun to explore the benefits that the fractional setting brings, including turbulence modeling \cite{del041, epp181, shl871}, biomedicine \cite{bue141, jav131, kla051}, geophysics \cite{bae102, bue141, zha121}, image processing \cite{bua101, gil081}, quantum mechanics \cite{las001, duo181}, and AI and machine learning \cite{ant202, ant223}. In traditional integer-order partial differential equation models (PDEs), diffusion of a quantity is typically captured by the Laplacian operator, $\Delta$. But, in many applications, the spatial decay of the quantity has been observed to be algebraic (``heavy tailed"), leading to these processed being categorized as exhibiting ``anomalous diffusion". The most commonly used approach to modeling such phenomena is a nonlocal version of the standard diffusion operator, namely the fractional Laplacian, $(-\Delta)^\alt.$

There are a number of definitions for the fractional Laplacian. In the case of  $\Omega \, = \, \real^n$ most of the definitions are equivalent \cite{kwa171}. In the case of a bounded domain, the non-locality of the fractional Laplacian causes certain definitions to no longer be equivalent. The three most common definitions used in the bounded domain setting are: (1) the \textit{spectral} fractional Laplacian, (2) the \textit{regional} fractional Laplacian, and (3) the \textit{integral} fractional Laplacian \cite{aco171}. The spectral fractional Laplacian is defined in terms of the eigenfunctions and eigenvalues of the homogeneous Laplace operator on the given domain. The second definition, the regional fractional Laplacian, has been shown to be the infinitesimal generator of the censored stable Lévy processes \cite{bog031}. The regional definition can also be viewed as a restriction of the integral fractional Laplacian. The third approach to defining this operator is via the \textit{integral} fractional Laplacian $(-\Delta)^\alt$. For a function $u$ with support in the domain and $1 < \alpha < 2$, 
\[
(-\Delta)^\alt u(x) \coloneqq \frac{1}{|\gamma_d(-\alpha)|}\lim_{\eps\ra 0}\int_{\R^d\setminus B(x,\eps)}
\frac{u(x) - u(y)}{| x - y |^{d+\alpha}}dy,
\]
where $\gamma_d(\alpha)$ is a normalization factor \cite{aco171}.

The authors in \cite{zhe251} showed that for $k\in \R$ and a sufficiently nice function $f$ on $\real^d$, the fractional Laplacian factors as:
\[
k(-\Delta)^\alt f = - \grad \cdot (- \Delta)^\frac{\alpha -2}{2} k\mathbf I \grad f,
\]
where $\mathbf I$ is the identity matrix in $\R^{d\times d}$ and $(- \Delta)^\frac{\alpha -2}{2}$ the Riesz potential operator. 
Motivated by the consideration of a  non-homogeneous domain (in particular, a composite material \cite{fai201}), 
the authors in \cite{zhe251} then considered the following model for the fractional Laplacian: \textit{For the bounded domain $\Omega$ 
denoting the unit disk in $\R^2$, given $f(x)$, $K(x) = \begin{bmatrix} k_1 & 0 \\ 0 & k_2\end{bmatrix},$ with $k_1, k_2 >0$, determine $\wt{u}(x)$ satisfying }
\begin{align}
\label{model1}\grad \cdot (-\Delta)^{\frac{\alpha-2}{2}}K\grad \wt{u}(x) & = f(x), \ \ x\in \Omega, \\
\notag \wt{u}(x) & = 0, \ \ x\in \R^2\setminus \Omega. 
\end{align}

Using a pseudo-eigenfunction property of the integral fractional Laplacian on the unit disk 
\cite{dyd171, hao211}, the authors in \cite{zhe251} utilized a power series argument to establish well posedness of the model \eqref{model1}. 
(Note that for $K(x)$ a constant, symmetric, positive definite matrix, a rotation of the coordinate axes such that the
axes align with the eigenvectors of $K(x)$ transforms $K(x)$ to a positive diagonal matrix.)

In \cite{erv251} the work in \cite{zhe251} was extended to the case of a positive definite matrix $K(x)$, subject to the restriction that the smallest and largest eigenvalues of $K$ satisfy $\lambda_m \mathbf v^T \mathbf v \leq \mathbf v^T K(x) \mathbf v \leq \lambda_M\mathbf v^T\mathbf v$ for all $\mathbf v\in \R^2$, $x\in \Omega$, with $\lambda_M < \frac{\sqrt{\alpha(2+\alpha)}}{2-\alpha}\lambda_m$. Whether this condition is also necessary remains an open question. The analysis in \cite{erv251} used a weak formulation of the problem and a different function space setting than in \cite{zhe251}. 

The focus of this paper is the extension of the fractional Laplacian model with constant $K$ to the evolution problem:
\begin{equation} \label{evModel}
\begin{cases}
\frac{\partial \wt{u}}{\partial t}(x,t) - \grad \cdot (-\Delta)^{\frac{\alpha - 2}{2}}K\grad \wt{u}(x,t) &= f(x,t),\ \ x \in \Omega,\ t>0,\\
\hfill \wt{u}(x,t)  &= 0, x\in \R^2\setminus \Omega,\\
\hfill \wt{u}(x,0) &= g(x), x\in \Omega. 
\end{cases}
\end{equation}

In order to accurately account for the singular behavior of $\wt{u}$ at the boundary of $\Omega$, we assume
$\wt{u}(x, t) \, = \, \omega^{\alt} u(x, t)$, where $\omega \, = \, (1 - |x|^{2})$, $x \in \Omega$.
The existence and uniqueness of $u$ is studied via a weak formulation in Sections \ref{steady_Ex_Uq} and \ref{evExEq}. 

The numerical approximation of problems involving the fractional Laplacian on a bounded domain has received considerable
attention over the past fifteen years. The three most commonly used approaches have been the 
finite difference \cite{duo182, duo191, hao212, hua141, min201, wan111},
finite element \cite{aco171, ain171, bon191, del131, tia161},
and spectral \cite{hao211, xu181, zen141, zho241} approximation schemes. 
Approximation schemes using radial basis functions have also been investigated \cite{bur211, leh161, hao251}.

As a prelude to discussing the numerical approximation of the solution to \eqref{evModel}, in Sections 
\ref{steady_Ex_Uq} and \ref{steady_approximation}
an analysis of the regularity of the solution and
an error analysis of the spectral approximation to the steady-state solution of \eqref{model1} is given. 
This analysis differs from that in \cite{zhe251} with regard to the function space setting, and is used in the error analysis for the fully discrete approximation to \eqref{evModel}. 
Additionally, the analysis highlights the block diagonal structure of the coefficient matrix that arises in the spectral approximation.

Using a backward Euler time discretization and a spectral spatial approximation, an error analysis for the approximation of \eqref{evModel} is given in Section \ref{time_approx}. Of particular note is the relationship between the temporal discretization parameter, $\Delta t$, and the spectral approximation parameter $R$ (describing the highest polynomial degree with respect to the radial variable). For the optimal convergence rate for the approximation for $u(x,t) \in C^2(0,T; H^s_\alt(\Omega))$, $\Delta t$ and $R$ should be chosen to satisfy
\[
 \Delta t \, \sim \, R^{- (s + \alpha/2) / 4} \, .
\]

In Section \ref{numerical_results}, two numerical examples are presented for the approximation of the (steady-state) fractional Laplacian with the coefficient matrix $K$ given by $K = \begin{bmatrix} 3& 0\\ 0 & 9 \end{bmatrix}.$ Example \ref{ex1} highlights the influence of $K$ on modeling a material with different axial conductivities, and the influence of different values for the diffusivity parameter $\alpha$. Example \ref{ex2} then computes an experimental convergence rate for the approximation of the solution and compares it to the theoretically predicted value. Additionally, two numerical examples are presented for the approximation of the evolution problem \eqref{evModel}, again using $K  = \begin{bmatrix} 3& 0\\ 0 & 9 \end{bmatrix}.$ For the first example, Example \ref{ex3}, the RHS is chosen to demonstrate the evolution of the solution as it evolves to the solution of the steady-state problem in Example \ref{ex1}. The second example, Example \ref{ex4}, computes an experimental convergence rate for the approximation of the solution and compares it to the theoretically predicted value. A brief summary of the results in this paper and concluding remarks are given in Section \ref{conc}.

\section{Preliminaries}
\label{preliminaries}
In this section we present definitions and notation used in this article. Much of the notation is adopted from that 
used in \cite{zhe251} and  \cite{erv251}. Throughout, we assume $1 < \alpha < 2$.

\textbf{The Fractional Laplacian}\\
In $\R^d$, the integral fractional Laplacian of a function $u(x)$, $(-\Delta)^{\alpha/2}u(x)$, is defined as
\begin{equation}\label{IFL}
(-\Delta)^\alt u(x) \coloneqq \frac{1}{|\gamma_d(-\alpha)|}\lim_{\eps\ra 0}\int_{\R^d\setminus B(x,\eps)}
\frac{u(x) - u(y)}{| x - y |^{d+\alpha}}dy,\  x \in \Omega,
\end{equation}
where $\gamma_d(\alpha) = \frac{2^\alpha \pi^{d/2} \Gamma(\alt)}{\Gamma\big(\frac{d-\alpha}{2}\big)}$ 
and $B(x,\eps)\subseteq \R^d$ is the ball centered at $x$ of radius $\eps$. \\

\textbf{Riesz Potential Operator}\\
For $0 < \alpha < d$, the Riesz potential operator of $u(x)$, denoted $(-\Delta)^{-\alpha/2}u(x)$, is defined as \cite{dyd171}  
\[
(-\Delta)^{-\alpha/2}u(x) \coloneqq \frac{1}{\gamma_d(\alpha)}\int_{\R^d} \frac{u(x - y)}{| y|^{d-\alpha}} \, dy .
\]
For suitably nice functions (\cite{dyd171}, Proposition 1), the integral fractional Laplacian and the Riesz potential operator are inverses. 
In \cite{erv251}, it is shown that for a bounded domain $\Omega\subset \R^d$, $(-\Delta)^{-\alpha/2}(\cdot)$ is self-adjoint with respect to the $L^2(\Omega)$ inner product.\\

\textbf{Jacobi Polynomials}\\
The Jacobi polynomials are defined in terms of the regularized hypergeometric function $_2\textbf{F}_1(\cdot)$ (see \cite{dyd171}) and can be written as
\begin{align*}
P_n^{(a,b)}(t) &= \frac{(-1)^n\Gamma(b+1+n)}{n!}{_2}{\bf F}_1\Big(\genfrac{}{}{0pt}{}{-n, 1+a+b+n}{b+1}\Big|\frac{1+t}{2}\Big)\\
&= (-1)^n\frac{\Gamma(n+1+b)}{n!\Gamma(n+1+a+b)}\sum_{j=0}^n (-1)^j 2^{-j}\binom{n}{j}\frac{\Gamma(n+j+1+a+b)}{\Gamma(j+1+b)}(1+t)^j.
\end{align*}
Importantly, when $a,b > -1$, the Jacobi polynomials satisfy the following weighted orthogonality property:
\[
\int_{-1}^1 (1-t)^a(1+t)^b P_j^{(a,b)}(t)P_k^{(a,b)}(t)dt = 
\begin{cases} 0, & k\neq j,\\ 
\vertiii{ P_j^{(a,b)}}^2, & k = j \end{cases}, 
\]
\[\text{where  } \vertiii{P_j^{(a,b)}}= \Big(\frac{2^{(a+b+1)}}{2j+a+b+1}\frac{\Gamma(j+a+1)\Gamma(j+b+1)}{\Gamma(j+1)\Gamma(j+a+b+1)}\Big)^{1/2}.
\]

\centerline{}\textbf{Solid Harmonic Polynomials}\\
The solid harmonic polynomials are defined to be solutions to Laplace's equation. In $\mathbb{R}^2$ the first few linearly independent solid harmonic polynomials in increasing degree are (up to linearly dependence):
\[1,\ x_1,\ x_2,\ x_1x_2,\ x_1^2 - x_2^2,\ x_1^3 - 3x_1x_2^2,\ 3x_1^2x_2-x_2^3,\ \hdots\]
Written in polar coordinates, the solid harmonics can be expressed as $r^l\cos(l\phi)$ or $r^l\sin(l\phi)$, $l\in \mathbb{N}$. \\

\textbf{The Function Space $L^2_\gamma(\Omega)$}\\
Throughout this paper, the domain $\Omega$ will denote the unit disk in $\R^2$, and we use the weight function
\[
\omega(x) = (1-|x|^2)_+ \coloneqq \begin{cases} (1 - |x|^2), & |x| \leq 1;\\ 0, & |x| >1,\end{cases}
\]
The analysis presented resides in weighted Lebesgue and Sobolev spaces. With $\gamma \in \real$, we use the weight $\omega(x)^\gamma$ to define the weighted $L^2(\Omega)$ space to be $L^2_\gamma(\Omega) \coloneqq \{f\colon \Omega\ra \R\st f \text{ is measurable, } \|f\|_{L^2_\gamma(\Omega)} <\infty\}$, with inner product and norm given by 
\[
(f,g)_\gamma \coloneqq \int_\Omega\omega(x)^\gamma f(x)g(x)d\Omega,\text{  and  } \|f\|_{L^2_\gamma(\Omega)} \coloneqq (f,f)_\gamma^{1/2}.
\]

\textbf{Note}: We denote the $L^{2}(\Omega)$ inner product of functions $f$ and $g$ over $\Omega$ simply by  $(f \, , \, g)$.

\centerline{}\textbf{Basis for $L^2_\gamma(\Omega)$}\\
In \cite{li141} the authors showed that products of solid harmonic polynomials and Jacobi polynomials form an orthogonal basis for $L^2_\gamma(\Omega)$, and further, these basis functions are pseudo-eigenfunctions for the  fractional Laplacian operator on $\Omega$ \cite{dyd171}. 

Using the Cartesian coordinate system, we denote a point $x \in \R^2$ as $x = (x_{1} , x_{2})$, and in polar coordinates, we denote $x = (r,\phi)$. 

Define $\mc V_{0,1}(x) \coloneqq \frac{1}{2},\ \mc V_{l,1}(x)\coloneqq r^l\cos(l\phi),$ and $\mc V_{l,-1}(x)\coloneqq r^l\sin(l\phi)$ for $l = 1, 2, \hdots$. 

We use the following notation
\[
    \mcV_{l , \mu^{*}}(x) \ = \ \left\{ \begin{array}{rl}
    \mcV_{l , -1}(x) & \mbox{  if } \mu = 1 \, ,   \\
    \mcV_{l , 1}(x) & \mbox{  if } \mu = -1 \, .  \end{array} \right.
\]

Additionally, for a linear operator $\mcF( \cdot )$,  we use $\mcF(\mcV_{l , \mu}(x)) \ = \ (\pm) \mcV_{l , \sigma}(x)$ to denote
\[
\mcF(\mcV_{l , \mu}(x)) = \ (\pm) \mcV_{l , \sigma}(x) \ =  \ \left\{ \begin{array}{rl}
   + \mcV_{l , \sigma}(x) & \mbox{  if } \mu = 1 \, , \\
   - \mcV_{l , \sigma}(x) & \mbox{  if } \mu = -1 \, .  \end{array} \right.
\]
For example,
\[
 \frac{\partial}{\partial \varphi} \mcV_{l , \mu}(x) \ = \ (\mp) \, l \, \mcV_{l , \mu^{*}}(x) \, .
\]

For $\mu = \pm 1$, $n\geq 0$, let 
$\P_{l,n,\mu}^{(\gamma)}(x) = \mc V_{l,\mu}(x)P_n^{(\gamma, l)}(2r^2 -1)$. An orthogonal basis for $L^2_\gamma(\Omega)$ is then given by \cite{li141}
\begin{equation}
\Big(\bigcup_{l=0}^\infty\bigcup_{n=0}^\infty \big\{\P_{l,n,1}^{(\gamma)}(x)\big\}\Big) \cup \Big(\bigcup_{l=1}^\infty\bigcup_{n=0}^\infty \big\{\P_{l,n,-1}^{(\gamma)}(x)\big\}\Big).
\label{basisL2w}
\end{equation}
For brevity, we write the basis as 
\[
\bigcup_{\mu = \pm1}\bigcup_{l = 0}^\infty \bigcup_{n=0}^\infty \big\{\P_{l,n,\mu}^{(\gamma)}(x)\big\},
\]
where we implicitly assume that the terms $\P_{0,n,-1}^{(\gamma)}(x)$, with $n \geq 0$, are excluded. We trivially extend the definition of $\P_{l,n,\mu}^{(\gamma)}(x)$ to negative integer values of $n$ and $l$ by defining $\P_{l,n,\mu}^{(\gamma)}(x) \coloneqq 0$ when $l < 0$ or $n < 0$. \\

\textbf{The Function Space $H^s_\gamma(\Omega)$}\\
Weighted Sobolev spaces are commonly employed to succinctly characterize the regularity of solutions which are known to have singular behaviors (such as solutions to the fractional diffusion equation, specifically at the boundary of $\Omega$ \cite{aco171}). Following Babuška and Guo \cite{bab011}, define the weighted Sobolev space $\mathcal{H}^s_\gamma(\Omega)$, where $s\in \N$ as 
\begin{equation}
\mathcal{H}^s_\gamma(\Omega)\coloneqq \big\{f\in L^2_\gamma(\Omega) \st |f|_{\mathcal{H}^s_\gamma} \coloneqq \Big(\sum_{j=0}^s {s\choose j}\Big\|\frac{\partial ^s f(x)}{\partial x_{2}^j \, \partial x_{1}^{s-j}}\Big\|^2_{L^2_{\gamma+s}(\Omega)}\Big)^{1/2} < \infty\big\},
\label{defHsp}
\end{equation}
with norm given by 
\[
\|f\|_{\mathcal{H}^s_\gamma(\Omega)} \coloneqq \big(\|f\|^2_{L^2_\gamma(\Omega)} + |f|^2_{\mathcal{H}^s_\gamma(\Omega)}\big)^{1/2}.
\]
For $s > 0$, $s \not\in \Natural$,  $\mathcal{H}_{\gamma}^{s}(\Omega)$ is defined using the K-method of interpolation. For
$s < 0$ the spaces are defined by  $L^{2}_{\gamma}(\Omega)$ duality.

In \cite{dyd171}, Dyda, Kuznetsov, and Kwa\'{s}nicki showed that the basis functions for $L^{2}_{\alpot}(\Omega)$
are pseudo-eigenfunctions for the fractional Laplacian of order $\alpha$ on the unit disk.
In view of this property, it is natural to seek a solution of the fractional diffusion equation 
 on the unit disk
expressed as an infinite series of these basis functions. The 
appropriate space to study such functions is defined by the decay rate of their coefficients.

For $s \in \real$, the space $H_{\gamma}^{s}(\Omega)$ is defined as
\begin{equation}
H_{\gamma}^{s}(\Omega) \ := \ \Big\{ u \in L^{2}_{\gamma}(\Omega) \, : \, 
\sum_{l, n, \mu} (n + 1)^{s} (n + l + 1)^{s} \, a_{l, n, \mu}^{2} \, 
\| \P_{l,n,\mu}^{(\gamma)}(x) \|_{L^{2}_{\gamma}}^{2}  < \infty \Big\} \, ,
\label{defXs}
\end{equation}
where
\[
u(x) \ = \ \sum_{l, n, \mu} a_{l, n, \mu} \, \P_{l,n,\mu}^{(\gamma)}(x) \, , 
\]
with 
\[
\| u \|_{H_{\gamma}^{s}}  := \ \Big( \sum_{l, n, \mu} (n + 1)^{s} (n + l + 1)^{s} \, a_{l, n, \mu}^{2} \, 
\|\P_{l,n,\mu}^{(\gamma)}(x) \|_{L^{2}_{\gamma}}^{2} \Big)^{1/2}.
\]

In \cite{erv241}, it was shown that the spaces $\mathcal{H}_{\gamma}^{s}(\Omega)$ and $H_{\gamma}^{s}(\Omega)$
are equivalent, and that their dual spaces, with respect to the $L^{2}_{\gamma}(\Omega)$ pivot space, are characterized by
$H_{\gamma}^{-s}(\Omega)$.

When convenient, we instead us the notation $L^2_\gamma \coloneqq L^2_\gamma(\Omega)$ and $H^s_\gamma \coloneqq H^s_\gamma(\Omega)$ (since the domain $\Omega$ is fixed throughout), and $Q := L^{2}_{\alt}(\Omega)$, 
$V := H^{\alt}_{\alt}(\Omega)$, and  $V' := H^{- \alt}_{\alt}(\Omega)$.

\textbf{Bochner Function Spaces}
When considering the parabolic problem, function spaces involving time are central to our analysis. For a Banach space $(X, \|\cdot\|_X)$ and final time $T>0$, the Bochner spaces $L^2(0, T; X)$ are defined as the set of all measurable functions $v\colon (0,T) \ra X$ such that the corresponding norm is finite:
\[
\|v\|_{L^2(0,T;X)} \coloneqq \Big (\int_0^T \|v(t)\|_X^2\, dt \Big)^{1/2}.
\]
As $\{ H^s_\gamma(\Omega) \}_{s \geq 0}$ form a family of interpolation spaces, the authors in \cite[Proposition 5.1]{lin021} showed that the spaces $\big\{L^2(0,T; H^s_\gamma(\Omega))\big\}_{s \ge 0}$ also 
form a family of interpolation spaces.


\textbf{Useful Derivative, Norm, and Asymptotic Properties} \\
Following are several results used in the existence and uniqueness analysis.
\begin{lemma} \label{Basis Riesz} \cite{zhe251} For $f(x) = \omega^\alt \P_{l, n, \mu}^{(\alpot )}(x)$, \begin{align*}
(-\Delta)^\frac{\alpha -2}{2} \frac{\partial f(x)}{\partial x_1} & = C_2 (n+\alt + l) \P_{l+1, n, \mu}^{(\alpot-1 )}(x) + C_2(n+l+1)\P_{l-1, n+1, \mu}^{(\alpot-1 )}(x), \text{ and}\\
(-\Delta)^\frac{\alpha -2}{2}\frac{\partial f(x)}{\partial x_2} & = C_2(\pm)(n+\alt+l)\P_{l+1, n, \mu^\star}^{(\alpot-1 )}(x) + C_2(\mp)(n+l+1)\P_{l-1, n+1, \mu^\star}^{(\alpot-1 )}(x),
\end{align*}
where $C_2 = -2^{\alpha -2}\frac{\Gamma(n+1+\alt)\Gamma(n+\alt +l)}{\Gamma(n+1)\Gamma(n+2+l)}$. \end{lemma}

\begin{lemma} \label{Basis Partials} \cite{zhe251} For $f(x) = \omega^\alt \P_{l, n, \mu}^{(\alpot )}(x)$, 
\begin{align*}
\frac{\partial f(x)}{\partial x_1} & = -(n+\alt)\omega^{\alt-1}\P_{l+1, n, \mu}^{(\alpot-1 )}(x) - (n+1)\omega^{\alt-1}\P_{l-1, n+1, \mu}^{(\alpot-1 )}(x), \text{ and}\\
\frac{\partial f(x)}{\partial x_2} & = -(\pm)(n+\alt)\omega^{\alt-1}\P_{l+1, n, \mu^\star}^{(\alpot-1 )}(x) - (\mp)(n+1)\omega^{\alt-1}\P_{l-1, n+1, \mu^\star}^{(\alpot-1 )}(x).\end{align*}
\end{lemma}

\begin{lemma} \label{NormWeightShift} \cite{erv241} For $l,n,k\geq 0$ and $(l,\mu) \neq (0,-1)$,
\begin{equation*}
\|\P_{l, n, \mu}^{(\gamma + k )} \|^2_{L^2_{\gamma+k}} = \frac{2n+\gamma + l+1}{2n + \gamma + l+1+k}\frac{\prod_{s=1}^k (n+\gamma + s)}{\prod_{s=1}^k(n+\gamma+l+s)}\|\P_{l, n, \mu}^{(\gamma)} \|^2_{L^2_\gamma}.
\end{equation*}
\end{lemma} 

\begin{lemma} \label{NormShiftAdd} \cite{erv241} For $l,n,j,m\geq 0$ and $(l,\mu)\neq (0,-1)$, 
\begin{align*}
\|\P_{l+j, n+m, \mu}^{(\gamma)} \|_{L^2_\gamma}^2 = \frac{C_{l+j,\mu}}{C_{l,\mu}}\frac{2n + \gamma + l + 1}{2n + 2m + \gamma + l + j + 1} \frac{\prod_{s=1}^m (n+\gamma+s)}{\prod_{s=1}^m (n+s)}&\frac{\prod_{s=1}^{m+j}(n+l+s)}{\prod_{s=1}^{m+j}(n+\gamma+l+s)}\\
&\cdot\|\P_{l, n, \mu}^{(\gamma)} \|^2_{L^2_\gamma}.
\end{align*}
where $C_{l,\mu} = \begin{cases} \pi/2, & \text{if } l \leq 0;\\ \pi, &\text{if } l \geq 1\end{cases}$.
\end{lemma}

\begin{lemma}\label{NormShiftSub}\cite{erv241} For $l,n\geq 0$ and $m\geq j\geq 0$ with $(l-j,\mu)\neq (0,-1)$, 
\begin{align*}
\|\P_{l-j, n+m, \mu}^{(\gamma)} \|^2_{L^2_\gamma} = \frac{C_{l-j,\mu}}{C_{l,\mu}}\frac{2n + \gamma + l+1}{2n + 2m + \gamma + l - j + 1}\frac{\prod_{s=1}^m (n+\gamma+s)}{\prod_{s=1}^m (n+s)}&\frac{\prod_{s=1}^{m-j} (n+l+s)}{\prod_{s=1}^{m-j}(n+\gamma+l+s)}\\
&\cdot\|\P_{l, n, \mu}^{(\gamma)} \|^2_{L^2_\gamma},
\end{align*}
with $C_{l,\mu}$ defined as in Lemma \ref{NormShiftAdd}. 
\end{lemma}
From Stirling's formula we have \cite{qi121},  
\[
\lim_{n\rightarrow\infty}\frac{\Gamma(n+\gamma)}{\Gamma(n)n^\gamma} = 1.
\]
In particular, for $n \geq 1$ and $0 < \alt < 1$,
\begin{equation}
   \frac{1}{2} \ <  \ \Big( \frac{ n }{n + \alt} \Big)^{1 - \alt}  \ \leq \  \frac{\Gamma(n + \alt)}{\Gamma(n) \, n^\alt} \ \leq \ 1  \, .
\label{stirling}
\end{equation}   

Note that for real numbers $a,b$, the notation $a \lesssim b$ is used to indicate that there exists a constant $c>0$ such that $a \, \leq \, c \,  b$, and $a \approx b$ denotes that there exists constants $c_{1}, \, c_{2} > 0$ such that
$c_{1} \, b \, \leq a \leq \, c_{2} \, b$.
Additionally, we use $\lfloor a \rfloor$ to denote the greatest integer less than or equal to $a$.

\section{Existence and Uniqueness to the Steady-State Problem}
\label{steady_Ex_Uq}
In this section we investigate the existence and uniqueness of the solution to the steady-state problem: 
\begin{equation}
\begin{cases}\label{problem}
- \grad\cdot(-\Delta)^\frac{\alpha-2}{2}K \grad \wt{u}(x) &= \ f(x),\ \  x \in \Omega , \\
 \hfill \wt{u}(x) &= 0, \ \  x \in \real^{2}  ,
\end{cases}
\end{equation}
where we assume $K(x) = \begin{bmatrix} k_1 & 0 \\ 0 & k_2\end{bmatrix},$ with $k_1, k_2 >0$.
Importantly, the analysis reveals a block diagonal structure of the coefficient matrix
that determines the solution of \eqref{prob2}.

To derive the weak formulation of (\ref{problem}), multiply (\ref{problem}) through by 
$\omega^{\alt} v(x)$, for $v \in V$ and integrate over $\Omega$ to obtain 
\begin{equation}
\Big((-\Delta)^{\frac{\alpha -2}{2}}K\grad (\omega^\alt{u}),\ \grad(\omega^\alt v)\Big) \ = \
\big( f \, ,\, \omega^{\alt} v \big)       \, .
\label{prob2}
\end{equation}
From Corollaries 3.3 and 3.4 in \cite{erv251},
\[
(-\Delta)^{\frac{\alpha -2}{2}}K\grad \, : \, \omega^\alt \otimes H^{\alt}_{\alt}(\Omega)
 \longrightarrow \Big( H^{1 - \alt}_{\alt - 1}(\Omega) \Big)^{2} \, ,  \mbox{ and } \ 
 \grad \, : \, \omega^\alt \otimes H^{\alt}_{\alt}(\Omega)
 \longrightarrow  \, \omega^{\alt - 1} \otimes \Big( H^{-(1 - \alt)}_{\alt - 1}(\Omega) \Big)^{2} \, . 
\]
Hence,
\[
(-\Delta)^{\frac{\alpha -2}{2}}K\grad (\omega^\alt{u}) \, := \, \mc{U} \, ,  \mbox{ and } 
\grad(\omega^\alt v) \, := \, \omega^{\alt - 1} \mc{V}  \, , \mbox{ with } \mc{U} \in \Big( H^{1 - \alt}_{\alt - 1}(\Omega) \Big)^{2} ,
\mbox{ and } \mc{V}  \in \Big( H^{-(1 - \alt)}_{\alt - 1}(\Omega) \Big)^{2} .
\]
Then,
\begin{equation} \label{mapx2}
\Big((-\Delta)^{\frac{\alpha -2}{2}}K\grad (\omega^\alt{u}),\ \grad(\omega^\alt v)\Big) 
\ = \ \Big( \omega^{\alt - 1} \, \mc{U} \ , \ \mc{V}  \Big) \ = \ \big\langle U \, , \, \mc{V}  \big\rangle_{H^{(1 - \alt)}_{\alt - 1} , H^{-(1 - \alt)}_{\alt - 1}} .
\end{equation}
By the notation $\langle \cdot , \cdot \rangle_{H^{\beta}_\gamma ,  H^{- \beta}_\gamma}$ we denote a duality pairing 
with respect to the $L^{2}_{\gamma}(\Omega)$ pivot space.

Motivated by \eqref{prob2} we introduce the definition for the solution to \eqref{problem}. 
 \begin{defn}
 For $f \in V'$,
we say that $\wt{u}(x) \ = \ \omega^{\alt} u(x)$, with $u  \in V$ 
is a weak solution to (\ref{problem}) if for all $v \in V$
\begin{equation}
B(u , v)  \ = \
\langle f \, ,\, v \rangle_{V' , V} \, ,
\label{prob3}
\end{equation}
where, for notation simplicity, $B(\cdot , \cdot) \, : \, V \times V$ is defined by
\begin{equation}
B( u, v) \, = \, \big\langle (-\Delta)^{\frac{\alpha -2}{2}}K\grad (\omega^\alt{u}) \ , \
 \grad(\omega^\alt v) \big\rangle_{H^{(1 - \alt)}_{\alt - 1} , H^{-(1 - \alt)}_{\alt - 1}}^{*} \ 
\coloneqq \ \big\langle \mc{U} \, , \, \mc{V} \big\rangle_{H^{(1 - \alt)}_{\alt - 1} , H^{-(1 - \alt)}_{\alt - 1}} \, .
\label{defB}
\end{equation}
\end{defn}
Since the test and trial spaces for \eqref{prob3} are the same, existence and uniqueness can be established via the 
Lax-Milgram theorem provided that $B(\cdot , \cdot)$ is bounded and coercive.

\subsection{Boundedness of $B$}
\begin{theorem}
\label{boundedness} Let $C_{cts} \, := \, \max\{k_1,k_2\} \, 2^{\alpha} \, 5$. Then,
for  ${u},{v}\in V$, we have 
\begin{equation}
B(u,v) \, \leq \, C_{cts} \, \|u\|_{V} \|v\|_{V} \, . 
\label{Bcts}
\end{equation}
\end{theorem}
\textbf{Proof}: 
Let $u, v \in V$ and write
\[
u(x) = \sum_{l,n,\mu}u_{l,n,\mu}\P_{l,n,\mu}^{(\alt)}(x),\ \ \text{ and } \ \
v(x) = \sum_{l,n,\mu}v_{l,n,\mu}\P_{l,n,\mu}^{(\alt)}(x) \, .
\]
 where $x = (r, \phi)$. Applying Lemma \ref{Basis Partials}, the gradient of $\omega^\alt v$ is given by 
\begin{align*}
\notag\frac{\partial (\omega^\alt v)(x)}{\partial x_1}&=-\omega^{\frac{\alpha}{2}-1} \sum_{l,n,\mu}\Big((n+\frac{\alpha}{2})\P_{l+1,n,\mu}^{(\alt-1)}(x)+ (n+1)\P_{l-1,n+1,\mu}^{(\alt-1)}(x)\Big)v_{l,n,\mu}\\
\notag\frac{\partial (\omega^\alt v)(x)}{\partial x_2}&= -\omega^{\frac{\alpha}{2}-1}\sum_{l,n,\mu}\Big((\pm)(n+\frac{\alpha}{2})\P_{l+1,n,\mu^\star}^{(\alt-1)}(x)+(\mp)(n+1)\P_{l-1,n+1,\mu^\star}^{(\alt-1)}(x)\Big)v_{l,n,\mu}.
\end{align*}

Reindexing these sums and introducing $V_{1}$ and $V_{2}$ we have
 \begin{equation}\label{gradv}
 \nabla(\omega^\alt v)(x) = -\omega^{\alt -1} \begin{pmatrix} V_1(x)\\V_2(x) \end{pmatrix} \coloneqq -\omega^{\alt-1}\begin{pmatrix} \sum_{l,n,\mu}(v_{l-1,n,\mu}(n+\alt)+v_{l+1,n-1,\mu}n)\P_{l,n,\mu}^{(\alt-1)}(x)\\ 
 \sum_{l,n,\mu}((\pm)v_{l-1,n,\mu}(n+\alt)+(\mp)v_{l+1,n-1,\mu}n)\P_{l,n,\mu^\star}^{(\alt-1)}(x)\end{pmatrix}.
 \end{equation}
Similarly, applying Lemma \ref{Basis Riesz} and introducing $U_{1}$ and $U_{2}$ we have
\begin{align*}
k_1U_1(x) \coloneqq k_1(-\Delta)^{\frac{\alpha-2}{2}}\frac{\partial (\omega^\alt u)(x)}{\partial x_1}& = -k_12^{\alpha -2}\sum_{l,n,\mu}\Big(\frac{\Gamma(n+\alt+1)\Gamma(n+\alt+l+1)}{\Gamma(n+1)\Gamma(n+2+l)}\P_{l+1,n,\mu}^{(\alt-1)}(x)\\
&\ \ \ \ \ \ + \frac{\Gamma(n+\alt+1)\Gamma(n+\alt+l)}{\Gamma(n+1)\Gamma(n+1+l)}\P_{l-1,n+1,\mu}^{(\alt-1)}(x)\Big)u_{l,n,\mu}\\
k_2U_2(x)\coloneqq k_2(-\Delta)^{\frac{\alpha-2}{2}}\frac{\partial (\omega^\alt u)(x)}{\partial x_2}&=-k_22^{\alpha-2}\sum_{l,n,\mu}\Big((\pm)\frac{\Gamma(n+\alt+1)\Gamma(n+\alt+l+1)}{\Gamma(n+1)\Gamma(n+2+l)}\P_{l+1,n,\mu^\star}^{(\alt-1)}(x)\\
&\ \ \ \ \ \ \  + (\mp)\frac{\Gamma(n+\alt+1)\Gamma(n+\alt+l)}{\Gamma(n+1)\Gamma(n+1+l)}\P_{l-1,n+1,\mu^\star}^{(\alt-1)}(x)\Big)u_{l,n,\mu}.
\end{align*}
After reindexing, 
\begin{align}
\notag k_1U_1(x) &= -k_12^{\alpha -2}\sum_{l,n,\mu}\Big(\frac{\Gamma(n+\alt+1)\Gamma(n+\alt+l)}{\Gamma(n+1)\Gamma(n+1+l)}u_{l-1,n,\mu} \\
&\label{rieszU1}\ \ \ \ \ \ \ \ \ \ \ \ \ \ \ \ \ \ \ \ \ + \frac{\Gamma(n+\alt)\Gamma(n+\alt+l)}{\Gamma(n)\Gamma(n+1+l)}u_{l+1,n-1,\mu}\Big)\P_{l,n,\mu}^{(\alt-1)}(x)\\
\notag k_2U_2(x) &=-k_22^{\alpha-2}\sum_{l,n,\mu}\Big((\pm)\frac{\Gamma(n+\alt+1)\Gamma(n+\alt+l)}{\Gamma(n+1)\Gamma(n+1+l)}u_{l-1,n,\mu}\\
&\label{rieszU2}\ \ \ \ \ \ \ \ \ \ \ \ \ \ \ \ \ \ \ \ \ + (\mp)\frac{\Gamma(n+\alt)\Gamma(n+\alt+l)}{\Gamma(n)\Gamma(n+1+l)}u_{l+1,n-1,\mu}\Big)\P_{l,n,\mu^\star}^{(\alt-1)}(x).
\end{align}

\noindent Substituting (\ref{gradv}), (\ref{rieszU1}), and (\ref{rieszU2}) into $B(\cdot , \cdot)$, 
\begin{align}
\notag B(u, v) &= 
\Biggl(\begin{pmatrix}  k_1U_1(x)\\k_2U_2(x) \end{pmatrix},\ -\omega^{\alt-1}\begin{pmatrix}V_1(x)\\V_2(x)\end{pmatrix} \Biggr)\\
\notag& = k_1\int_\Omega -\omega^{\alt-1} U_1(x) \, V_1(x) \, d\Omega + k_2\int_\Omega-\omega^{\alt-1} U_2(x) \, V_2(x) \, d\Omega\\
\label{Bbound}& \coloneqq k_1I_1 + k_2I_2.\end{align}
For $I_1$, expand out the integrand and using the orthogonality property of the basis, 
\begin{align*}
I_1 &= 2^{\alpha-2}  \sum_{l,n,\mu} \Big(u_{l-1,n,\mu} \Big(\frac{n+\alt}{n} \Big)+u_{l+1,n-1,\mu} \Big)
\Big(v_{l-1,n,\mu}(n+\alt)+v_{l+1,n-1,\mu}n \Big)  \\
&\hspace{5cm}\cdot\frac{\Gamma(n+\alt)\Gamma(n+l+\alt)}{\Gamma(n)\Gamma(n+l+1)}
 \| \P_{l,n,\mu}^{(\alt-1)} \|_{L^{2}_{\alt - 1}}^2 .
\end{align*}
Multiplying and dividing by $n^\frac{1}{2}$, and using Cauchy-Schwarz, 
this reduces to 
 \begin{align*}
 I_1 &\leq 2^{\alpha -2}\Big(\sum_{l,n,\mu}\Big(u_{l-1,n,\mu} \Big(\frac{n+\alt}{n} \Big) + u_{l+1,n-1,\mu} \Big)^2n \, 
 \frac{\Gamma(n+\alt)\Gamma(n+l+\alt)}{\Gamma(n)\Gamma(n+l+1)} \, 
 \|\P_{l,n,\mu}^{(\alt-1)} \|^2_{L^2_{\alt-1}}\Big)^\frac{1}{2}\\
 & \hspace{1cm}\cdot\Big(\sum_{l,n,\mu}\Big(v_{l-1,n,\mu}(n+\alt)+v_{l+1,n-1,\mu}(n) \Big)^2
 \frac{1}{n} \, \frac{\Gamma(n+\alt)\Gamma(n+l+\alt)}{\Gamma(n)\Gamma(n+l+1)}
 \|\P_{l,n,\mu}^{(\alt-1)} \|^2_{L^2_{\alt-1}}\Big)^\frac{1}{2}\\
  & \leq 2^{\alpha-1} \Big(\sum_{l,n,\mu}\Big(u^2_{l-1,n,\mu}\frac{(n+\alt)^2}{n} + nu^2_{l+1,n-1,\mu} \Big)
 \frac{\Gamma(n+\alt)\Gamma(n+l+\alt)}{\Gamma(n)\Gamma(n+l+1)}
 \|\P_{l,n,\mu}^{(\alt-1)} \|^2_{L^2_{\alt-1}}\Big)^\frac{1}{2}\\
 & \hspace{1cm}\cdot\Big(\sum_{l,n,\mu}\Big(v^2_{l-1,n,\mu}\frac{(n+\alt)^2}{n}+nv^2_{l+1,n-1,\mu}\Big)
 \frac{\Gamma(n+\alt)\Gamma(n+l+\alt)}{\Gamma(n)\Gamma(n+l+1)}
 \|\P_{l,n,\mu}^{(\alt-1)} \|^2_{L^2_{\alt-1}}\Big)^\frac{1}{2}.
\end{align*}
 Note that these two sums are exactly the same, except one corresponds to $u$ and the other to $v$. 
 Denoting these sums by $\mathcal S_u$ and $\mathcal S_v$, the above can be written as
 \[
 I_1\leq 2^{\alpha-1} (\mathcal{S}_u)^\frac{1}{2}(\mathcal{S}_v)^\frac{1}{2}.
 \]
 \par Splitting the $\mathcal{S}_u$ sum into two pieces, 
 separating each of the coefficients into their own sum, and reindexing yields
\begin{align*}
\mathcal{S}_u &= \sum_{l,n,\mu}u^2_{l,n,\mu}\frac{(n+\alt)^2}{n} \frac{\Gamma(n+\alt)\Gamma(n+l+\alt+1)}{\Gamma(n)\Gamma(n+l+2)}\|\P_{l+1,n,\mu}^{(\alt-1)} \|^2_{L^2_{\alt-1}}\\
&  \hspace{.56cm}+\sum_{l,n,\mu}u^2_{l,n,\mu}(n+1)\frac{\Gamma(n+\alt+1)\Gamma(n+l+\alt)}{\Gamma(n+1)\Gamma(n+l+1)}\|\P_{l-1,n+1,\mu}^{(\alt-1)} \|^2_{L^2_{\alt-1}}
\end{align*}
Using Lemmas \ref{NormShiftAdd} and \ref{NormShiftSub}, when $l\geq 0$, the basis functions can be shifted in norm by:
\begin{align}
\label{normshift1}\|\P_{l+1,n,\mu}^{(\alt-1)} \|^2_{L^2_{\alt-1}} &= \frac{C_{l+1,\mu}}{C_{l,\mu}}\frac{2n + \alt + l}{2n + \alt + l + 1} \frac{n+l+1}{n+\alt + l} \|\P_{l,n,\mu}^{(\alt-1)} \|^2_{L^2_{\alt-1}} \text{ and}\\
\label{normshift2}\|\P_{l-1,n+1,\mu}^{(\alt-1)} \|^2_{L^2_{\alt-1}} & = \frac{C_{l-1,\mu}}{C_{l,\mu}} \frac{2n+\alt + l}{2n +  \alt  + l +1}\frac{n+\alt}{n+1}\|\P_{l,n,\mu}^{(\alt-1)} \|^2_{L^2_{\alt-1}},
\end{align}
where $C_{l,\mu} = \begin{cases} \pi/2, & l \leq 0\\ \pi, & l \geq 1\end{cases}$. 
 Note here that $\frac{C_{l+1,\mu}}{C_{l,\mu}}, \frac{C_{l-1,\mu}}{C_{l,\mu}} \leq 2$. 
 Additionally, Lemma \ref{NormWeightShift} gives the relation between the $L^2_{\alt-1}$ norm and the $L^2_{\alt}$ norm:
 \begin{equation}
 \label{normshift3}
 \|\P_{l,n,\mu}^{(\alt-1)} \|^2_{L^2_{\alt-1}} = 
 \frac{2n+\alt+l+1}{2n+\alt+l}\frac{n+\alt+l}{n+\alt}\|\P_{l,n,\mu}^{(\alt)} \|^2_{L^2_{\alt}}.
 \end{equation}
 Substituting (\ref{normshift1}), (\ref{normshift2}), and (\ref{normshift3}) into $\mathcal{S}_u$ and rearranging, 
 (recall $Q = L^{2}_{\alt}(\Omega)$, $V = H^{\alt}_{\alt}(\Omega)$)
\begin{align*}
\mathcal{S}_u &\leq \sum_{l,n,\mu}2u^2_{l,n,\mu} \frac{\Gamma(n+\alt+1)\Gamma(n+l+\alt+1)}{\Gamma(n+1)\Gamma(n+l+1)}  
\|\P_{l,n,\mu}^{(\alt)} \|^2_{Q}\\
&  \hspace{.5cm}+\sum_{l,n,\mu}2u^2_{l,n,\mu}\frac{\Gamma(n+\alt+1)\Gamma(n+l+\alt+1)}{\Gamma(n+1)\Gamma(n+l+1)}
\|\P_{l,n,\mu}^{(\alt)} \|^2_{Q}\\
& \leq 4\sum_{l,n,\mu}u_{l,n,\mu}^2\frac{\Gamma(n+\alt+1)\Gamma(n+l+\alt+1)}{\Gamma(n+1)\Gamma(n+l+1)}
\|\P_{l,n,\mu}^{(\alt)} \|^2_{Q}\\
& \leq 4 \,  \sum_{l,n,\mu}u^2_{l,n,\mu}(n+1)^\alt(n+l+1)^\alt\|\P_{l,n,\mu}^{(\alt)} \|^2_{Q}\\
& = 4 \,  \|u\|^2_{V},
\end{align*}
where the last inequality comes from applying Stirling's formula \eqref{stirling} to $\frac{\Gamma(n+\alt+1)}{\Gamma(n+1)}$ and 
$\frac{\Gamma(n+l+\alt+1)}{\Gamma(n+l+1)}$. 
Applying the same steps to $\mathcal{S}_v$, would similarly yield 
$\mathcal{S}_v \leq 4 \,  \|v\|^2_{V}$. Hence
\begin{equation}\label{I1Bound}
\notag I_1 \leq 2^{\alpha -1}(\mathcal{S}_u)^\frac{1}{2}(\mathcal{S}_v)^\frac{1}{2}
 \leq  2^{\alpha+1} \,  \|{u}\|_{V} \|{v}\|_{V} \, .
\end{equation}

\par Next, expanding the integrand in $I_2$, using the orthogonality of the basis, and applying Cauchy-Schwarz we get
\begin{align*}
I_2 & = -\int_\Omega \omega^{\alt -1} \, U_2(x) \, V_2(x) \, d\Omega\\
& = 2^{\alpha -2}  \sum_{l,n,\mu}\Big( (\pm) u_{l-1,n,\mu} \Big(\frac{n+\alt}{n} \Big) + (\mp)u_{l+1,n-1,\mu} \Big)
\Big( (\pm)v_{l-1,n,\mu}(n+\alt) + (\mp)nv_{l+1,n-1,\mu} \Big) \\
&\hspace{3cm}  \cdot
\frac{\Gamma(n+\alt)\Gamma(n+l+\alt)}{\Gamma(n)\Gamma(n+l+1)}
\| \P_{l,n,\mu^\star}^{(\alt-1)} \|_{L^{2}_{\alt -1}}^2 \\
&\leq 2^{\alpha -2}\Bigg(\sum_{l,n,\mu}\big((\pm)u_{l-1,n,\mu}\big(\frac{n+\alt}{n}\big)+(\mp)u_{l+1,n-1,\mu}\big)^2n\frac{\Gamma(n+\alt)\Gamma(n+l+\alt)}{\Gamma(n)\Gamma(n+l+1)}\\
& \ \ \ \ \ \ \ \ \ \ \ \ \ \ \ \ \ \ \ \ \ \ \ \ \ \ \ \ \ \ \ \ \ \ \ \ \ \ \ \ \ \ \ \ \ \ \ \ \ \ \ \ \ \ \ \ \ \ \ \ \ \ \ \ \ \ \ \ \ \ \ \ \ \ \ \ \ \ \ \ \ \ \ \ \ 
\cdot \|\P_{l,n,\mu^\star}^{(\alt-1)} \|^2_{L^2_{\alt-1}}\Bigg)^\frac{1}{2}\\
&\hspace{1.05cm}\cdot\Bigg(\sum_{l,n,\mu}\big((\pm)v_{l-1,n,\mu}(n+\alt) + (\mp)nv_{l+1,n-1,\mu}\big)^2\frac{1}{n}\frac{\Gamma(n+\alt)\Gamma(n+l+\alt)}{\Gamma(n)\Gamma(n+l+1)}\\
& \ \ \ \ \ \ \ \ \ \ \ \ \ \ \ \ \ \ \ \ \ \ \ \ \ \ \ \ \ \ \ \ \ \ \ \ \ \ \ \ \ \ \ \ \ \ \ \ \ \ \ \ \ \ \ \ \ \ \ \ \ \ \ \ \ \ \ \ \ \ \ \ \ \ \ \ \ \ \ \ \ \ \ \ \ 
\cdot\|\P_{l,n,\mu^\star}^{(\alt-1)} \|^2_{L^2_{\alt-1}}\Bigg)^\frac{1}{2}\\
& \leq 2^{\alpha-1} \Bigg(\sum_{l,n,\mu}\big(u^2_{l-1,n,\mu}\frac{(n+\alt)^2}{n}+nu^2_{l+1,n-1,\mu}\big)\frac{\Gamma(n+\alt)\Gamma(n+l+\alt)}{\Gamma(n)\Gamma(n+l+1)}\|\P_{l,n,\mu^\star}^{(\alt-1)} \|^2_{L^2_{\alt-1}}\Bigg)^\frac{1}{2}\\
&\hspace{1.05cm}\cdot\Bigg(\sum_{l,n,\mu} \big(v^2_{l-1,n,\mu}\frac{(n+\alt)^2}{n} + nv^2_{l+1,n-1,\mu}\big)\frac{\Gamma(n+\alt)\Gamma(n+l+\alt)}{\Gamma(n)\Gamma(n+l+1)}\|\P_{l,n,\mu^\star}^{(\alt-1)} \|^2_{L^2_{\alt-1}}\Bigg)^\frac{1}{2}.
\end{align*}
\par Similar to before, notice the above sums have the same form, 
the only difference being that one sums over coefficients from ${u}$, and the other over the corresponding coefficients from ${v}$. 
To simplify the analysis, define these sums as $\hat{\mathcal S}_u$ and $\hat{\mathcal S}_v$ respectively and consider $\hat{\mathcal S}_u$. 
Splitting $\hat{\mathcal S}_u$ into two sums and reindexing just as before, we then use the fact that 
$\|\P_{l,n,\mu^\star}^{(\alt-1)} \|^2_{L^2_{\alt-1}} = \|\P_{l,n,\mu}^{(\alt-1)} \|^2_{L^2_{\alt-1}}$ when $l \geq 1$:
\begin{align*}
\hat{S}_u 
& = \sum_{l\geq 0,n,\mu} u^2_{l,n,\mu}\frac{(n+\alt)^2}{n}\frac{\Gamma(n+\alt)\Gamma(n+l+\alt+1)}{\Gamma(n)\Gamma(n+l+2)}\|\P_{l+1,n,\mu^\star}^{(\alt-1)} \|^2_{L^2_{\alt-1}}\\
&\hspace{0.5cm} + \sum_{l\geq 1, n,\mu} u^2_{l,n,\mu}(n+1)\frac{\Gamma(n+\alt+1)\Gamma(n+l+\alt)}{\Gamma(n+1)\Gamma(n+l+1)}
\|\P_{l-1,n+1,\mu^\star}^{(\alt-1)} \|^2_{L^2_{\alt-1}}\\
& = \sum_{l\geq 0,n,\mu} u^2_{l,n,\mu}\frac{(n+\alt)^2}{n}\frac{\Gamma(n+\alt)\Gamma(n+l+\alt+1)}{\Gamma(n)\Gamma(n+l+2)}
\|\P_{l+1,n,\mu}^{(\alt-1)} \|^2_{L^2_{\alt-1}}\\
&\hspace{0.5cm} + \sum_{l\geq 2,n,\mu} u^2_{l,n,\mu}(n+1)\frac{\Gamma(n+\alt+1)\Gamma(n+l+\alt)}{\Gamma(n+1)\Gamma(n+l+1)}
\|\P_{l-1,n+1,\mu}^{(\alt-1)} \|^2_{L^2_{\alt-1}}\\
(\text{for }l = 1) &\hspace{0.5cm} + \sum_{n} u^2_{1,n,-1}(n+1)\frac{\Gamma(n+\alt+1)\Gamma(n+1+\alt)}{\Gamma(n+1)\Gamma(n+2)}
\|\P_{0,n+1,1}^{(\alt-1)} \|^2_{L^2_{\alt-1}}.
\end{align*}
Using \eqref{normshift1} and \eqref{normshift2} we obtain
\begin{align}
\notag\hat{\mathcal S}_u &\leq \sum_{l,n,\mu} 2u^2_{l,n,\mu}\frac{\Gamma(n+\alt+1)\Gamma(n+l+\alt+1)}{\Gamma(n+1)\Gamma(n+l+1)}
 \|\P_{l,n,\mu}^{(\alt)} \|^2_{L^2_{\alt}}\\
\notag&\hspace{0.5cm} + \sum_{l\geq 2, n, \mu} 2u^2_{l,n,\mu}\frac{\Gamma(n+\alt+1)\Gamma(n+l+\alt+1)}{\Gamma(n+1)\Gamma(n+l+1)}\|\P_{l,n,\mu}^{(\alt)} \|^2_{L^2_{\alt}}\\
 \notag&\hspace{0.5cm} + \sum_{n} u^2_{1,n,-1}(n+1)\frac{\Gamma(n+\alt+1)\Gamma(n+1+\alt)}{\Gamma(n+1)\Gamma(n+2)}
 \|\P_{0,n+1,1}^{(\alt-1)}\|^2_{L^2_{\alt-1}}\\
\notag& \leq \sum_{l,n,\mu} 4u^2_{l,n,\mu}\frac{\Gamma(n+\alt+1)\Gamma(n+l+\alt+1)}{\Gamma(n+1)\Gamma(n+l+1)} 
\|\P_{l,n,\mu}^{(\alt)} \|^2_{L^2_{\alt}}\\
\label{su.1}&\hspace{0.5cm} + \sum_{n} u^2_{1,n,-1}(n+1)\frac{\Gamma(n+\alt+1)\Gamma(n+1+\alt)}{\Gamma(n+1)\Gamma(n+2)}
\|\P_{0,n+1,1}^{(\alt-1)} \|^2_{L^2_{\alt-1}} \, .
\end{align}
Next, using Lemma \ref{NormShiftAdd}, 
\begin{align}
\label{norm17.1}\|\P_{0,n+1,1}^{(\alt-1)} \|^2_{L^2_{\alt-1}} &= \big(\frac{2n + \alt }{2n+\alt + 2}\big)
\|\P_{0,n,1}^{(\alt-1)} \|^2_{L^2_{\alt-1}}, \text{ and}\\
\notag\|\P_{0,n,1}^{(\alt-1)} \|^2_{L^2_{\alt-1}} &= \frac{C_{0,1}}{C_{1,1}}\frac{2n+\alt + 1}{2n + \alt}\frac{n+\alt}{n+1}
\|\P_{1,n,1}^{(\alt-1)} \|^2_{L^2_{\alt-1}}\\
\label{norm17.2}& = \frac{C_{0,1}}{C_{1,1}}\frac{2n+\alt + 1}{2n + \alt}\frac{n+\alt}{n+1}
\|\P_{1,n,-1}^{(\alt-1)} \|^2_{L^2_{\alt-1}}.
\end{align}
Substituting (\ref{norm17.1}) and (\ref{norm17.2}) into (\ref{su.1}) (and using $Q = L^{2}_{\alt}(\Omega)$),
\begin{align}
\notag\hat{\mathcal S}_u & \leq \sum_{l,n,\mu} 4u^2_{l,n,\mu}\frac{\Gamma(n+\alt+1)\Gamma(n+l+\alt+1)}{\Gamma(n+1)\Gamma(n+l+1)} 
\|\P_{l,n,\mu}^{(\alt)} \|^2_{Q}\\
\notag&\hspace{0.5cm} + \sum_{n} u^2_{1,n,-1}\frac{\Gamma(n+\alt+1)\Gamma(n+1+\alt)}{\Gamma(n+1)\Gamma(n+1)} 
\big(\frac{2n + \alt +1}{2n+\alt + 2}\big)\frac{C_{0,1}}{C_{1,1}}\frac{n+\alt}{n+1}
\|\P_{1,n,-1}^{(\alt-1)} \|^2_{L^2_{\alt-1}}\\
\notag&\leq \sum_{l,n,\mu} 4u^2_{l,n,\mu}\frac{\Gamma(n+\alt+1)\Gamma(n+l+\alt+1)}{\Gamma(n+1)\Gamma(n+l+1)} 
\|\P_{l,n,\mu}^{(\alt)} \|^2_{Q}\\
\notag&\hspace{0.5cm} + \sum_{l = 1, n\geq 0} 2u^2_{l,n,-1}\frac{\Gamma(n+\alt+1)\Gamma(n+l+\alt+1)}{\Gamma(n+1)\Gamma(n+l +1)}
\|\P_{l,n,-1}^{(\alt)} \|^2_{Q}\\
\label{su18.1}& \leq \sum_{l,n,\mu} 6u_{l,n,\mu}^2 \frac{\Gamma(n+\alt+1)\Gamma(n+l+\alt+1)}{\Gamma(n+1)\Gamma(n+l+1)} 
\|\P_{l,n,\mu}^{(\alt)} \|^2_{Q}.
\end{align}
Again, using Stirling's formula \eqref{stirling}, we obtain
\[
\hat{\mathcal S}_u \leq 6 \, \sum_{l,n,\mu} u^2_{l,n,\mu}(n+1)^\alt(n+l+1)^\alt \|\P_{l,n,\mu}^{(\alt)}(x)\|^2_{Q} 
=  6 \, \|u\|^2_{V}.
\]
An analogous analysis would show $\hat{\mathcal S}_v \leq 
 6 \, \|v\|^2_{V}$.

Returning to $I_2$, we have 
\begin{align*}
I_2 & \leq 2^{\alpha-1} \, (\hat{\mathcal S}_u)^\frac{1}{2} \, (\hat{\mathcal S}_v)^\frac{1}{2}
 \leq 2^{\alpha} \, 3 \,  \|u\|_{V} \, \|v\|_{V}.
\end{align*}
Finally, substitute $I_1$ and $I_2$ into  (\ref{Bbound}):
\begin{align*}
B(u, v) & = k_1I_1 + k_2I_2\\
& \leq k_1 \,  2^{\alpha+1}  \, \|u\|_{V} \, \|v\|_{V} \ + \
k_2 \, 2^{\alpha} \, 3 \,  \|u\|_{V} \, \|v\|_{V}  \\
& \leq \max\{k_1,k_2\} \, 2^{\alpha} \, 5   \, \|u\|_{V} \, \|v\|_{V} \, .
\end{align*}
\mbox{ } \hfill \qed
%
%
\subsection{Coercivity of $B$}
Next, we give a proof for the coercivity of $B$.
\begin{theorem}\label{coe}
 Let $C_{coe} \, := \, 2^{\alpha-4} \min\{k_1,k_2\} $. Then,
for $u \in  V$, we have 
\begin{equation}
B(u, u) \, \geq \, C_{coe} \, \|u\|^2_{V} \, . 
 \label{coercivity}
\end{equation}
\end{theorem}
\textbf{Proof}:
Writing $u$ as $u(x) = \sum_{l,n,\mu}u_{l,n,\mu}\P_{l,n,\mu}^{(\alt)}(x),$ we follow the same procedure as in the beginning of the proof for Theorem \ref{boundedness} to get the expressions
\begin{align}
\frac{\partial (\omega^\alt u)(x)}{\partial x_1}&=-\omega^{\alt-1} \sum_{l,n,\mu}\Big((n+\alt)u_{l-1,n,\mu}+ nu_{l+1,n-1,\mu}\Big)\P_{l,n,\mu}^{(\alt-1)}(x),\\
\frac{\partial (\omega^\alt u)(x)}{\partial x_2}&= -\omega^{\alt-1}\sum_{l,n,\mu}\Big((\pm)(n+\alt)u_{l-1,n,\mu}+(\mp)nu_{l+1,n-1,\mu}\Big)\P_{l,n,\mu^\star}^{(\alt-1)}(x),\\
\notag k_1(-\Delta)^{\frac{\alpha-2}{2}}\frac{\partial (\omega^\alt u)(x)}{\partial x_1}&= -k_12^{\alpha -2}\sum_{l,n,\mu}\Big(\frac{\Gamma(n+\alt+1)\Gamma(n+\alt+l)}{\Gamma(n+1)\Gamma(n+1+l)}u_{l-1,n,\mu} \\
&\ \ \ \ \ \ \ \ \ \ \ \ \ \ \ \ \ \ \ \ \ + \frac{\Gamma(n+\alt)\Gamma(n+\alt+l)}{\Gamma(n)\Gamma(n+1+l)}u_{l+1,n-1,\mu}\Big)\P_{l,n,\mu}^{(\alt-1)}(x),\\
\notag k_2(-\Delta)^{\frac{\alpha-2}{2}}\frac{\partial (\omega^\alt u)(x)}{\partial x_2}&=-k_22^{\alpha-2}\sum_{l,n,\mu}\Big((\pm)\frac{\Gamma(n+\alt+1)\Gamma(n+\alt+l)}{\Gamma(n+1)\Gamma(n+1+l)}u_{l-1,n,\mu}\\
&\ \ \ \ \ \ \ \ \ \ \ \ \ \ \ \ \ \ \ \ \ + (\mp)\frac{\Gamma(n+\alt)\Gamma(n+\alt+l)}{\Gamma(n)\Gamma(n+1+l)}u_{l+1,n-1,\mu}\Big)\P_{l,n,\mu^\star}^{(\alt-1)}(x).
\end{align}

Using the orthogonality of the basis, we obtain the following expressions for the inner products of $(-\Delta)^{\frac{\alpha-2}{2}}\frac{\partial (\omega^\alt u)}{\partial x_1}$ with $\frac{\partial(\omega^\alt u)}{\partial x_1}$ and of $(-\Delta)^{\frac{\alpha-2}{2}}\frac{\partial (\omega^\alt u)}{\partial x_2}$ with $\frac{\partial(\omega^\alt u)}{\partial x_2}$:
\begin{align}
\notag\Big((-\Delta&)^{\frac{\alpha-2}{2}}\frac{\partial(\omega^\alt u)(x)}{\partial x_1},\frac{\partial(\omega^\alt u)(x)}{\partial x_1}\Big) \\
\notag&= 2^{\alpha -2}\sum_{l,n,\mu}\Big(\frac{\Gamma(n+\alt+1)\Gamma(n+\alt+l)}{\Gamma(n+1)\Gamma(n+1+l)}u_{l-1,n,\mu}+ \frac{\Gamma(n+\alt)\Gamma(n+\alt+l)}{\Gamma(n)\Gamma(n+1+l)}u_{l+1,n-1,\mu}\Big)\\
\notag&\ \ \ \ \ \ \ \ \ \ \ \ \ \ \ \ \ \ \ \cdot\Big((n+\alt)u_{l-1,n,\mu}+ nu_{l+1,n-1,\mu}\Big)\|\P_{l,n,\mu}^{(\alt-1)} \|^2_{L^2_{\alt-1}}\\
\notag&=2^{\alpha -2}\sum_{l,n,\mu}\Big(\frac{\Gamma(n+\alt)\Gamma(n+\alt+l)}{\Gamma(n+1)\Gamma(n+1+l)}(n+\alt)^2u^2_{l-1,n,\mu}+2\frac{\Gamma(n+\alt+1)\Gamma(n+\alt+l)}{\Gamma(n)\Gamma(n+1+l)} \\
\notag&\ \ \ \ \ \ \ \ \ \ \ \ \ \ \ \ \ \ \ \cdot u_{l-1,n,\mu}u_{l+1,n-1,\mu}+ \frac{\Gamma(n+\alt)\Gamma(n+\alt+l)}{\Gamma(n)\Gamma(n+1+l)}nu_{l+1,n-1,\mu}^2\Big)
\|\P_{l,n,\mu}^{(\alt-1)} \|^2_{L^2_{\alt-1}}\\
\notag&=2^{\alpha -2}\sum_{l,n,\mu}\Big(\frac{n+\alt}{n}u^2_{l-1,n,\mu}+2u_{l-1,n,\mu}u_{l+1,n-1,\mu}+ \frac{n}{n+\alt}u_{l+1,n-1,\mu}^2\Big)\\
&\ \ \ \ \ \ \ \ \ \ \ \ \ \ \ \ \ \ \ \cdot\frac{\Gamma(n+\alt+1)\Gamma(n+\alt+l)}{\Gamma(n)\Gamma(n+1+l)}\|\P_{l,n,\mu}^{(\alt-1)} \|^2_{L^2_{\alt-1}}\\
\notag\Big((-\Delta&)^\frac{\alpha-2}{2}\frac{\partial(\omega^\alt u)(x)}{\partial x_2},\frac{\partial(\omega^\alt u)(x)}{\partial x_2}\Big) \\
\notag&= 2^{\alpha -2}\sum_{l,n,\mu}\Big((\pm)\frac{\Gamma(n+\alt+1)\Gamma(n+\alt+l)}{\Gamma(n+1)\Gamma(n+1+l)}u_{l-1,n,\mu}+ (\mp)\frac{\Gamma(n+\alt)\Gamma(n+\alt+l)}{\Gamma(n)\Gamma(n+1+l)}u_{l+1,n-1,\mu}\Big)\\
\notag&\ \ \ \ \ \ \ \ \ \ \ \ \ \ \ \ \ \ \ \cdot\Big((\pm)(n+\alt)u_{l-1,n,\mu}+(\mp)nu_{l+1,n-1,\mu}\Big)\|\P_{l,n,\mu^\star}^{(\alt-1)} \|^2_{L^2_{\alt-1}}\\
\notag&=2^{\alpha -2}\sum_{l,n,\mu}\Big(\frac{\Gamma(n+\alt)\Gamma(n+\alt+l)}{\Gamma(n+1)\Gamma(n+1+l)}(n+\alt)^2u^2_{l-1,n,\mu}- 2\frac{\Gamma(n+\alt+1)\Gamma(n+\alt+l)}{\Gamma(n)\Gamma(n+1+l)}\\
\notag&\ \ \ \ \ \ \ \ \ \ \ \ \ \ \ \ \ \ \ \cdot u_{l+1,n-1,\mu}u_{l-1,n,\mu}+\frac{\Gamma(n+\alt)\Gamma(n+\alt+l)}{\Gamma(n)\Gamma(n+1+l)}nu^2_{l+1,n-1,\mu}\Big)
\|\P_{l,n,\mu^\star}^{(\alt-1)} \|^2_{L^2_{\alt-1}}\\
\notag& = 2^{\alpha -2}\sum_{l,n,\mu}\Big(\frac{(n+\alt)}{n}u^2_{l-1,n,\mu} - 2u_{l+1,n-1,\mu}u_{l-1,n,\mu}+\frac{n}{n+\alt}u^2_{l+1,n-1,\mu}\Big)\\
&\ \ \ \ \ \ \ \ \ \ \ \ \ \ \ \ \ \ \ 
\cdot\frac{\Gamma(n+\alt+1)\Gamma(n+\alt+l)}{\Gamma(n)\Gamma(n+1+l)}\|\P_{l,n,\mu^\star}^{(\alt-1)} \|^2_{L^2_{\alt-1}}
\end{align}

Use the linearity of the inner product, and that $\|\P_{l,n,\mu}^{(\alt-1)} \|_{L^2_{\alt-1}} =\|\P_{l,n,\mu^\star}^{(\alt-1)} \|_{L^2_{\alt-1}}$ (except when $l = 0$) to see that 
\begin{align}
B(u,u) & =  \notag\Big(K(-\Delta)^\frac{\alpha-2}{2}\grad(\omega^\alt u)(x) \, , \, \grad(\omega^\alt u)(x) \Big) \\ 
\notag& = 2^{\alpha-2}\sum_{l\geq 1,n,\mu}\Big((k_1+k_2)\frac{n+\alt}{n}u^2_{l-1,n,\mu} + 2(k_1-k_2)u_{l-1,n,\mu}u_{l+1,n-1,\mu}+(k_1+k_2)\frac{n}{n+\alt}\\
\notag&\ \ \ \ \ \ \ \ \ \ \ \ \ \ \ \ \ \ \ \ \ \cdot u^2_{l+1,n-1,\mu}\Big)\frac{\Gamma(n+\alt+1)\Gamma(n+\alt+l)}{\Gamma(n)\Gamma(n+1+l)}\|\P_{l,n,\mu}^{(\alt-1)} \|^2_{L^2_{\alt-1}}\\
\notag&\ \ \ \ \ \ \ \ \ \ \ \ \ \ \ \ \ \ \ \ \ + 2^{\alpha -2}\sum_{n\geq 1} k_1\frac{n}{n+\alt}u_{1,n-1,1}^2\frac{\Gamma(n+\alt+1)\Gamma(n+\alt)}{\Gamma(n)\Gamma(n+1)}\|\P_{0,n,1}^{(\alt-1)} \|^2_{L^2_{\alt-1}} \\
\notag&\ \ \ \ \ \ \ \ \ \ \ \ \ \ \ \ \ \ \ \ \ + 2^{\alpha -2}\sum_{n\geq 1} k_2\frac{n}{n+\alt}u_{1,n-1,-1}^2\frac{\Gamma(n+\alt+1)\Gamma(n+\alt)}{\Gamma(n)\Gamma(n+1)}\|\P_{0,n,1}^{(\alt-1)} \|^2_{L^2_{\alt-1}}\\
\notag& = 2^{\alpha-2}\sum_{l\geq 1,n,\mu}\Big((k_1+k_2)\frac{n+\alt}{n}u^2_{l-1,n,\mu} + 2(k_1-k_2)u_{l-1,n,\mu}u_{l+1,n-1,\mu}+(k_1+k_2)\frac{n}{n+\alt}\\
\notag&\ \ \ \ \ \ \ \ \ \ \ \ \ \ \ \ \ \ \ \ \ \cdot u^2_{l+1,n-1,\mu}\Big)\frac{\Gamma(n+\alt+1)\Gamma(n+\alt+l)}{\Gamma(n)\Gamma(n+1+l)}\|\P_{l,n,\mu}^{(\alt-1)} \|^2_{L^2_{\alt-1}}\\
\label{B1.20}&\ \ \ \ \ \ \ \ \ \ \ \ \ \ \ \ \ \ \ \ \ + 2^{\alpha -2}\sum_{n\geq 1}(k_1u_{1,n-1,1}^2+k_2u_{1,n-1,-1}^2)
\frac{\Gamma(n+\alt)^2}{\Gamma(n)^2} \|\P_{0,n,1}^{(\alt-1)} \|^2_{L^2_{\alt-1}}.\end{align}
Isolating the first sum, call it $\mathcal{S}_1$, notice that it can be written as
\begin{align}
\notag\mathcal{S}_1  &= 2^{\alpha-2}\sum_{l\geq 1,n,\mu}\Big((k_1+k_2)\frac{n+\alt}{n}u^2_{l-1,n,\mu} + (k_1-k_2)2\sqrt{\frac{n+\alt}{n}}u_{l-1,n,\mu}\sqrt{\frac{n}{n+\alt}}u_{l+1,n-1,\mu}\\
\notag&\ \ \ \ \ \ \ \ \ \ \ \ \ \ \ \ \ \ \ \ \ + (k_1+k_2)\frac{n}{n+\alt}u_{l+1,n-1,\mu}^2\Big)\\
\notag&\ \ \ \ \ \ \ \ \ \ \ \ \ \ \ \ \ \ \ \ \ \cdot\frac{\Gamma(n+\alt+1)\Gamma(n+\alt+l)}{\Gamma(n)\Gamma(n+1+l)}\|\P_{l,n,\mu}^{(\alt-1)} \|^2_{L^2_{\alt-1}}\\
\notag&\geq 2^{\alpha-2}\sum_{l\geq1,n,\mu}\Big((k_1+k_2)\frac{n+\alt}{n}u_{l-1,n,\mu}^2-
 |k_1-k_2| \Big(\frac{n+\alt}{n}u_{l-1,n,\mu}^2 + \frac{n}{n+\alt}u_{l+1,n-1,\mu}^2 \Big)\\
\notag&\ \ \ \ \ \ \ \ \ \ \ \ \ \ \ \ \ \ \ \ \ +(k_1+k_2)\frac{n}{n+\alt}u^2_{l+1,n-1,\mu}\Big)\\
\notag&\ \ \ \ \ \ \ \ \ \ \ \ \ \ \ \ \ \ \ \ \ \cdot\frac{\Gamma(n+\alt+1)\Gamma(n+\alt+l)}{\Gamma(n)\Gamma(n+1+l)}\|\P_{l,n,\mu}^{(\alt-1)} \|^2_{L^2_{\alt-1}}\\
\notag&=2^{\alpha -2}\sum_{l\geq 1,n,\mu}2\min\{k_1,k_2\}\Big(\frac{n+\alt}{n}u^2_{l-1,n,\mu} + \frac{n}{n+\alt}u^2_{l+1,n-1,\mu}\Big)\frac{\Gamma(n+\alt+1)\Gamma(n+\alt+l)}{\Gamma(n)\Gamma(n+1+l)}\\
& \ \ \ \ \ \ \ \ \ \ \ \ \ \ \ \ \ \ \ \ \ \cdot\|\P_{l,n,\mu}^{(\alt-1)} \|^2_{L^2_{\alt-1}}.\label{S1.21}
\end{align}
From Lemma \ref{NormWeightShift}, we know that 
\begin{equation} \label{normweight 21}
\|\P_{l,n,\mu}^{(\alt-1)} \|^2_{L^2_{\alt-1}} = \Big(\frac{2n+\alt + l + 1}{2n+ \alt + l}\Big)\Big(\frac{n+\alt+l}{n+\alt}\Big)\|\P_{l,n,\mu}^{(\alt)} \|^2_{L^2_{\alt}}.\end{equation}
Substituting (\ref{S1.21}) and (\ref{normweight 21}) into each of the sums in (\ref{B1.20}), we obtain (using $Q = L^{2}_{\alt}(\Omega)$):
\begin{align}
\notag B(u, u) &\geq 2^{\alpha-1}\min\{k_1,k_2\}\sum_{l\geq 1,n,\mu}\Big(\frac{n+\alt}{n}u^2_{l-1,n,\mu} + \frac{n}{n+\alt}u^2_{l+1,n-1,\mu}\Big)\frac{\Gamma(n+\alt)\Gamma(n+\alt+l+1)}{\Gamma(n)\Gamma(n+1+l)}\\
\notag&\ \ \ \ \ \ \ \ \ \ \ \ \ \ \ \ \ \ \cdot\Big(\frac{2n+\alt + l + 1}{2n+ \alt + l}\Big)\|\P_{l,n,\mu}^{(\alt)} \|^2_{Q}\\
\label{B2.21}&\ \ \ \ \ \ \ \ \ \ \ \ \ \ \ \ \ \ +2^{\alpha-2}\sum_{n\geq 1}(k_1u_{1,n-1,1}^2+k_2u_{1,n-1,-1}^2)\frac{2n+\alt+1}{2n+\alt}
\frac{\Gamma(n+\alt)^2}{\Gamma(n)^2}  \|\P_{0,n,1}^{(\alt)} \|^2_{Q}
\end{align}
Reindexing, and using Lemma \ref{NormShiftSub}, the second sum in (\ref{B2.21}) becomes: 
\begin{align}
\notag&\sum_{n\geq 0} (k_1u_{1,n,1}^2 + k_2u_{1,n,-1}^2)\frac{2n+\alt+3}{2n+\alt+2}
\frac{\Gamma(n+\alt + 1)^2}{\Gamma(n + 1)^2}  \|\P_{0,n+1,1}^{(\alt)} \|^2_{Q}\\
\label{secondSum}&=\sum_{n\geq 0}\frac{1}{2}(k_1u_{1,n,1}^2+k_2u_{1,n,-1}^2)\frac{\Gamma(n+\alt+1)\Gamma(n+\alt+2)}{\Gamma(n+1)\Gamma(n+2)}\|\P_{1,n,1}^{(\alt)} \|^2_{Q}.
\end{align}

Turning our attention then to the first sum in (\ref{B2.21}), in order to get only one set of coefficients here, we break the 
sum into two pieces, consider the pieces individually, reindex, and finally combine the pieces back together to get the desired result. 

The first half of the first sum in (\ref{B2.21}) is given by
\begin{align}
\notag\sum_{l\geq 1,n,\mu}&u^2_{l-1,n,\mu}\Big(\frac{n+\alt}{n}\Big)\frac{\Gamma(n+\alt)\Gamma(n+\alt+l+1)}{\Gamma(n)\Gamma(n+1+l)}\Big(\frac{2n+\alt + l + 1}{2n+ \alt + l}\Big)\|\P_{l,n,\mu}^{(\alt)} \|^2_{Q}\\
\notag&=\sum_{l\geq 0,n,\mu}u^2_{l,n,\mu}\frac{\Gamma(n+\alt+1)\Gamma(n+\alt+l+2)}{\Gamma(n+1)\Gamma(n+2+l)}\Big(\frac{2n+\alt + l + 2}{2n+ \alt + l+1}\Big)\|\P_{l+1,n,\mu}^{(\alt)} \|^2_{Q}\\
\notag&=\sum_{l\geq 0,n,\mu}u^2_{l,n,\mu}\frac{\Gamma(n+\alt+1)\Gamma(n+\alt+l+2)}{\Gamma(n+1)\Gamma(n+2+l)}\Big(\frac{2n+\alt + l + 2}{2n+ \alt + l+1}\Big)\\
\notag&\ \ \ \ \ \ \ \ \ \ \ \ \ \ \ \ \ \ \ \ \ \cdot\frac{C_{l+1,\mu}}{C_{l,\mu}}\Big(\frac{2n+\alt+l+1}{2n+\alt+l+2}\Big)\Big(\frac{n+l+1}{n+\alt+l+1}\Big)\|\P_{l,n,\mu}^{(\alt)} \|^2_{Q}  \ \mbox{ (using \eqref{NormShiftAdd})}  \\
\label{firstHalfFirstSum}&=\sum_{l\geq 0,n,\mu}u^2_{l,n,\mu}\frac{\Gamma(n+\alt+1)\Gamma(n+\alt+l+1)}{\Gamma(n+1)\Gamma(n+1+l)}\Big(\frac{C_{l+1,\mu}}{C_{l,\mu}}\Big)\|\P_{l,n,\mu}^{(\alt)} \|^2_{Q}.
\end{align}

In the same manner, using \eqref{normshift2}),
the second half of the first sum in (\ref{B2.21}) becomes
\begin{align}
\notag\sum_{l\geq 1,n,\mu}&u_{l+1,n-1,\mu}^2\Big(\frac{n}{n+\alt}\Big)\Big(\frac{\Gamma(n+\alt)\Gamma(n+\alt+l+1)}{\Gamma(n)\Gamma(n+1+l)}\Big)\Big(\frac{2n+\alt + l + 1}{2n+ \alt + l}\Big)\|\P_{l,n,\mu}^{(\alt)} \|^2_{Q}\\
\notag&=\sum_{l\geq2,n,\mu}u_{l,n,\mu}^2\Big(\frac{n+1}{n+\alt+1}\Big)\Big(\frac{\Gamma(n+\alt +1)\Gamma(n+\alt+l+1)}{\Gamma(n+1)\Gamma(n+1+l)}\Big)\Big(\frac{2n+\alt + l + 2}{2n+ \alt + l +1}\Big)\\
\notag&\ \ \ \ \ \ \ \ \ \ \ \ \ \ \ \ \ \ \ \ \ \cdot\|\P_{l-1,n+1,\mu}^{(\alt)} \|^2_{Q}\\
\notag&= \sum_{l\geq2,n,\mu}u_{l,n,\mu}^2\Big(\frac{n+1}{n+\alt+1}\Big)\Big(\frac{\Gamma(n+\alt +1)\Gamma(n+\alt+l+1)}{\Gamma(n+1)\Gamma(n+1+l)}\Big)\Big(\frac{2n+\alt + l + 2}{2n+ \alt + l +1}\Big)\\
\notag&\ \ \ \ \ \ \ \ \ \ \ \ \ \ \ \ \ \ \ \ \ \cdot\Big(\frac{C_{l-1,\mu}}{C_{l,\mu}}\Big)\Big(\frac{2n+\alt+l+1}{2n+\alt+l+2}\Big)\Big(\frac{n+\alt+1}{n+1}\Big)\|\P_{l,n,\mu}^{(\alt)} \|^2_{Q}\\
\label{secondHalfFirstSum}&= \sum_{l\geq2,n,\mu}u_{l,n,\mu}^2\Big(\frac{\Gamma(n+\alt +1)\Gamma(n+\alt+l+1)}{\Gamma(n+1)\Gamma(n+1+l)}\Big)\|\P_{l,n,\mu}^{(\alt)} \|^2_{Q}.
\end{align}
Combining (\ref{secondSum}), (\ref{firstHalfFirstSum}), and (\ref{secondHalfFirstSum}), we then see that (\ref{B2.21}) becomes
\begin{align}
\notag B(u,u) &= \Big(K (-\Delta)^\frac{\alpha-2}{2}\nabla(\omega^\alt u),\nabla(\omega^\alt u)\Big)\\
\notag(\text{from }(\ref{firstHalfFirstSum}), l = 0)\ \ & \geq 2^{\alpha-1} \min\{k_1,k_2\}\sum_{n\geq 0} u_{0,n,1}^2\frac{\Gamma(n+\alt+1)\Gamma(n+\alt+1)}{\Gamma(n+1)\Gamma(n+1)}\|\P_{0,n,1}^{(\alt)} \|^2_{Q}\\
\notag(\text{from }(\ref{firstHalfFirstSum}),(\ref{secondSum}), l = 1)\ \ & \ \ \ \ \ \ + 2^{\alpha-2}\sum_{n\geq 0}
\Big( \Big(\frac{k_1}{2} + 2\min\{k_1,k_2\} \Big)u_{1,n,1}^2 + \Big(\frac{k_2}{2}+2\min\{k_1,k_2\} \Big)\\
\notag&\ \ \ \ \ \ \ \ \ \ \ \ \ \ \ \ \ \ \ \ \ \cdot u_{1,n,-1}^2\Big)\frac{\Gamma(n+\alt+1)\Gamma(n+\alt+2)}{\Gamma(n+1)\Gamma(n+2)}\|\P_{1,n,1}^{(\alt)} \|^2_{Q}\\
(\text{from }(\ref{firstHalfFirstSum}),(\ref{secondHalfFirstSum}), l \geq 2)& \ \ \ \ \ \ +2^{\alpha-1} \min\{k_1,k_2\}\sum_{l\geq 2,n,\mu}u_{l,n,\mu}^2\frac{\Gamma(n+\alt+1)\Gamma(n+\alt+l+1)}{\Gamma(n+1)\Gamma(n+1+l)}\|\P_{l,n,\mu}^{(\alt)} \|^2_{Q}.
\end{align}
Bounding the term for $l = 1$, we combine everything into one sum again:
\begin{align}
\label{qwew} B(u,u) &\geq 2^{\alpha-2}\min\{k_1,k_2\}\sum_{l,n,\mu}u_{l,n,\mu}^2\frac{\Gamma(n+\frac{\alpha}{2}+1)\Gamma(n+\frac{\alpha}{2}+l+1)}{\Gamma(n+1)\Gamma(n+1+l)}\|\P_{l,n,\mu}^{(\alt)} \|^2_{Q}.
\end{align}
Using Stirling's formula \eqref{stirling}, 
\begin{align}
\notag B(u,u) &\geq 2^{\alpha-4}  \min\{k_1,k_2\}\sum_{l,n,\mu}u_{l,n,\mu}^2(n+1)^\alt(n+1+l)^\alt\|\P_{l,n,\mu}^{(\alt)} \|^2_{Q}\\
\notag& =2^{\alpha-4} \min\{k_1,k_2\}\|u\|^2_{V}.
\end{align}
\mbox{ } \hfill \qed

With the boundedness and coercivity of $B$ established, we can now state the existence and uniqueness of the solution
to \eqref{prob3}.
\begin{theorem}
 \label{LMthm} 
 For $k_1,\, k_2\in \mathbb{R}^+$, 
 $K = \begin{pmatrix} k_1 & 0 \\ 0 & k_2 \end{pmatrix}$, and
$f\in V'$
there exists a unique $u  \in  V$ satisfying \eqref{prob3}.
Additionally, 
\begin{equation}
\|u\|_{V} \ \leq \ \frac{1}{C_{coe}} \,  \| f \|_{V'} \, . 
\label{ubd1}
\end{equation}
\end{theorem}
\textbf{Proof}:
 Theorems 3.1 and 3.2 established that the bilinear form $B\colon  V \times V \rightarrow \mathbb{R}$ is bounded and coercive. Applying the Lax-Milgram theorem, the conclusion immediately follows.  \\
\mbox{ } \hfill \qed

\subsection{Regularity Analysis}

In \cite{zhe251}, the authors showed that the action of 
$-\grad\cdot(-\Delta)^\frac{\alpha-2}{2}K\grad\omega^\alt\sum_{l,n,\mu}u_{l,n,\mu}\P_{l,n,\mu}^{(\alt)}(x)$ 
is equal to
 \[
 \sum_{(l,n,\mu)}\Big(a_{-1}u_{l-2,n+1,\mu}  + a_0u_{l,n,\mu} + a_1u_{l+2,n-1,\mu}\Big)\P_{l,n,\mu}^{(\alt)}(x),
 \] 
for specific values of $a_{-1}, a_0,$ and $a_1$ depending on $l,n,\mu$. With the RHS function $f(x)$ similarly expanded in the orthogonal basis $\{\P_{l,n,\mu}^{(\alt)}(x)\}$, the authors equated coefficients which resulted in an infinite system of coupled equations, with (at most) three unknowns per equation. 

Specifically, the equations have the form: for $l,n\geq 0, \mu \in \{-1,1\}$, 
\begin{equation} 
\label{compCoeffs} 
a_{-1}u_{l-2,n+1,\mu} + a_0 u_{l,n,\mu} + a_1u_{l+2,n-1,\mu} = f_{l,n,\mu},
\end{equation}
with  $u_{l,n,\mu} = 0$ for $l,n<0$ and $(l,\mu) = (0, -1),$ 
and where $f_{l,n,\mu}$ are the coefficients resulting from the expansion of $f$ in the basis for $L^2_\alt(\Omega)$. 

Up to the constant $\|\P_{l,n,\mu}^{(\alt)} \|^2_{L^2_{\alt}}$, choosing 
$v(x) = \P_{l,n,\mu}^{(\alt)}(x)$ and substituting into \eqref{prob3} results in the
same equation \eqref{compCoeffs}.

The equations \eqref{compCoeffs} provide a convenient framework to investigate the regularity of the solution $u$ of \eqref{prob3} in terms of the regularity of the RHS function $f$.

Making the substitutions 
\begin{equation} 
\label{DlnFtscaling} 
d_{l,n} = \frac{\Gamma(n+1+\alt)}{\Gamma(n+1)}u_{l,n,1},\ \ \ \text{and}\ \ \ \wt{f}_{l,n} = 2^{-(\alpha-2)}\frac{\Gamma(n+1+l)}{\Gamma(n+1+\alt + l)}f_{l,n,1},
\end{equation}
the system of equations for $\mu = 1$ reduces to (see \cite{zhe251}, Section 4)
\begin{alignat}{2}
\notag R1&\ (n=0,l = 0):\ \ &&(k_1+k_2)d_{0,0} = \wt{f}_{0,0}\\
\notag R2&\ (n=0, l = 1): &&(2k_1+(k_1+k_2))d_{1,0} = \wt{f}_{1,0}\\
\notag R3&\ (n=0,l)_{l\geq 2}: &&2(k_1+k_2)d_{l,0} + (k_1-k_2)d_{l-2,1} = \wt{f}_{l,0}\\
\notag R4&\ (n,l=0)_{n\geq1}: &&(k_1-k_2)d_{2,n-1} + (k_1+k_2)d_{0,n} = \wt{f}_{0,n}\\
\notag R5&\ (n, l = 1)_{n\geq 1}: &&(k_1-k_2)d_{3,n-1}+(2k_1+(k_1+k_2))d_{1,n} = \wt{f}_{1,n}\\
\label{SystemD}R6&\ (n,l)_{n\geq1,\, l\geq 2}: && (k_1-k_2)d_{l+2,n-1} + 2(k_1+k_2)d_{l,n} + (k_1-k_2)d_{l-2,n+1} = \wt{f}_{l,n}.
\end{alignat} 
The equations for $\mu = -1$ are analogous, with the exclusion of the case corresponding to $l = 0$. 
Note that the coefficients $d_{l,n}$ can be visualized as an array of unknowns (see Figure \ref{unknownArrayD}),
with the cases denoted $R1$, $R2$, $\ldots$, $R6$, representing a partition of the array (see Figure \ref{part1}).

\begin{figure}[!ht]
\begin{minipage}{.46\linewidth}
\begin{center}
 \includegraphics[height=2.25in]{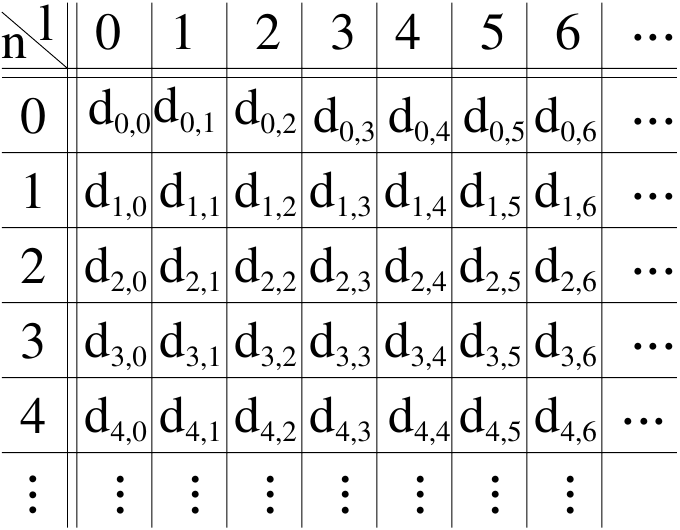}
 \caption{Array for the unknowns $d_{l,n}$.}
\label{unknownArrayD}
\end{center}

\end{minipage} \hfill
\begin{minipage}{.46\linewidth}
 
\begin{center}
 \includegraphics[height=2.25in]{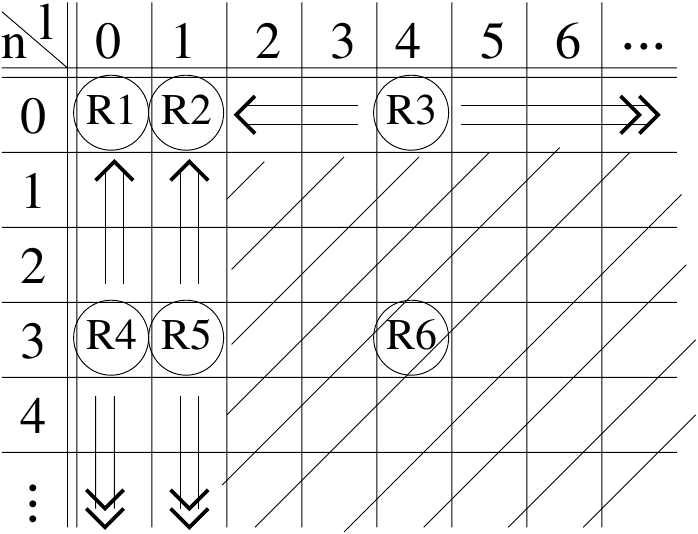}
 \caption{Partition of the index set $\{(n , l)\}_{n \geq 0 , l \geq 0}$ for $\mu = 1$ into $R1, R2, \ldots, R6$.}
\label{part1}
\end{center}
\end{minipage} 
\end{figure}

The system in (\ref{SystemD}) and Figure \ref{unknownArrayD} then give rise to a ``coupling stencil" for the unknowns, as illustrated in Figure \ref{stencil}.

\begin{figure}[!ht]
\begin{center}
 \includegraphics[height=1.6in]{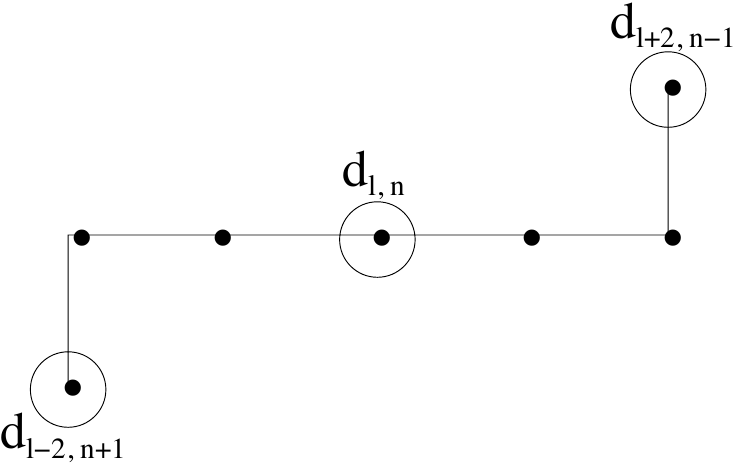}
   \caption{Stencil illustrating the coupling of the unknowns $d_{l , n}$.}
   \label{stencil}
\end{center}
\end{figure}

\par Applying the stencil, beginning at the upper left corner, and successively applying it at the lower left node from the previous node, induces a renumbering of the unknowns $d_{l,n}\rightarrow e_j$ (and of $\wt{f}_{l,n} \rightarrow b_j$) as shown in Figure \ref{relabeling}. 

\begin{figure}[!ht]
\begin{center}
 \includegraphics[height=2.0in]{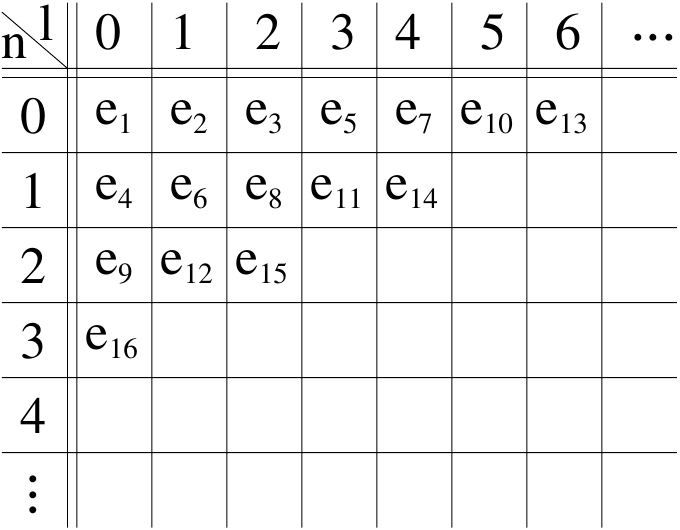}
   \caption{Corresponding array of the relabeled unknowns $e_j$.}
   \label{relabeling}
\end{center}
\end{figure}
\par Further, this results in the equations (\ref{SystemD}) decoupling into a sequence of block diagonal, tridiagonal matrices of dimensions $1,1,2,2, \hdots, m,m, \hdots$. 
For $\mathbf{e} \, = \, \{ e_{j} \}_{j \ge 1}$, $A \mathbf{e} \, = \, \mathbf{b}$, has the following block diagonal structure.
\begin{equation}
A \, = \, \left[ \begin{array}{cccccccc}
k_{1} + k_{2} &  &   &  &  &  &   &    \\
   &  \hspace{-0.5in}   2 k_{1}  +  (k_{1} + k_{2}) &  &   &  &  &  &     \\
   &       A_{3, 4, 2} &  &   &  &  &   &   \\
   &       & \hspace{-0.3in} A_{3, 5, 2} &  &   &  &    &     \\
   &       &     & \hspace{-0.3in}   \ddots &  &   &         &   \\
   &       &     &     & \hspace{-0.3in}  A_{3, 4, m}  &  &       &         \\
   &       &     &     &     &  \hspace{-0.3in}  A_{3, 5,m}  &     &     \\
   &       &     &     &    &     &     \hspace{-0.3in}   \ddots   &   \\
   &       &     &     &    &     &     &       \ddots   
 \end{array}
\right]
\label{Astr}
\end{equation}

The matrices $A_{3, 4, m}$ and $A_{3, 5, m}$ are $m\times m$, symmetric, tridiagonal matrices given by
\begin{equation}
A_{3, 4, m} \, = \, \left[ \begin{array}{ccccccc}
2(k_{1} + k_{2}) &  (k_{1} - k_{2})&   &  &  &  &    \\
 (k_{1} - k_{2})  &     2(k_{1} + k_{2}) &  (k_{1} - k_{2}) &   &  &  &      \\
   &      & \hspace{-1.0in} \ddots &  &     &  &        \\
   &       &  &\hspace{-1.0in} \ddots &     &   &           \\
   &       &     & \hspace{-1.5in} (k_{1} - k_{2})  &  \hspace{-0.5in}   2(k_{1} + k_{2}) &    (k_{1} - k_{2})   &              \\
   &       &     &   &    \hspace{-0.5in}   (k_{1} - k_{2})  &         (k_{1} + k_{2})             \end{array}
\right]
\label{Asub1}
\end{equation}
and
\begin{equation}
A_{3, 5, m} \, = \, \left[ \begin{array}{ccccccc}
2(k_{1} + k_{2}) &  (k_{1} - k_{2})&   &  &  &  &    \\
 (k_{1} - k_{2})  &     2(k_{1} + k_{2}) &  (k_{1} - k_{2}) &   &  &  &      \\
   &      & \hspace{-1.0in} \ddots &  &     &  &        \\
   &       &  &\hspace{-1.0in} \ddots &     &   &           \\
   &       &     & \hspace{-1.5in} (k_{1} - k_{2})  &  \hspace{-0.5in}   2(k_{1} + k_{2}) &    (k_{1} - k_{2})   &              \\
   &       &     &   &    \hspace{-0.5in}   (k_{1} - k_{2})  &         2k_{1} + (k_{1} + k_{2})             \end{array}
\right]
\label{Asub2}
\end{equation}

The matrix $A_{3, 4, m}$ represents the system of equations corresponding to unknowns 
$\mathbf{e}_{3,4,m} = \{ e_{j} \} $, $j \, = \, m(m-1) + 1, \ldots, m^{2}$, and
$A_{3, 5, m}$ represents the system of equations corresponding to unknowns $\mathbf{e}_{3,5,m} = \{ e_{j} \} $, $j \, = \, m^{2} + 1, \ldots, m(m+1)$.

Using the Gershgorin circle theorem it is straight forward to show that the eigenvalues for $A_{3, k, m}$, $k \in \{ 4 , \, 5 \}$,
satisfy
\[
 c_{\lambda \, min} \, := \, 2 \min\{ k_{1} \, , \, k_{2} \} \ \leq \lambda \leq \ 4 \max\{ k_{1} \, , \, k_{2} \} \, := \, c_{\lambda \, max} \, .
\]

With $\bf{b}_{3, k, m}$, corresponding to the components of $\bf{b}$ in $A_{3, k, m} \, \bf{e}_{3, k, m} \, = \, \bf{b}_{3, k, m}$, then
\begin{equation}
\bf{e}_{3, k, m} \, = \, A_{3, k, m}^{-1} \, \bf{b}_{3, k, m} \ \ \Longrightarrow \ 
\| \bf{e}_{3, k, m} \|_{2} \ \leq \ \| A_{3, k, m}^{-1} \|_{2} \, \| \bf{b}_{3, k, m} \|_{2} \ 
 \leq \ c_{\lambda \, min}^{-1} \,  \| \bf{b}_{3, k, m} \|_{2}  \, .
\label{ghrt1}
\end{equation} 
Hence,
\begin{equation} 
\label{changeofbasis} 
\sum_{l=1}^\infty \sum_{n=0}^\infty d_{l,n}^2 
\leq  c_{\lambda \, min}^{-2} \sum_{l = 1}^\infty \sum_{n=1}^\infty \wt{f}^2_{l,n} 
= c_{\lambda \, min}^{-2} 
\sum_{l=1}^\infty \sum_{n=0}^\infty \bigg(2^{-(\alpha-2)}\frac{\Gamma(n+1+l)}{\Gamma(n+1+\alt+l)}f_{l,n,1}\bigg)^2 \, .
\end{equation}

\textbf{Note}: The blocks in $A$ represent a decoupling of the basis functions with respect to the degree of the
radial variable $r$. Specifically, the block $A_{3, 4, 2}$ represents the basis functions in $r$ of degree $2$, 
$A_{3, 4, 3}$ the basis functions in $r$ of degree $3$, $\ldots$, $A_{3, 4, m}$ the basis in $r$ of degree $2 ( m - 1)$,
and $A_{3, 5, m}$ the basis in $r$ of degree $2 ( m - 1) \, + \, 1$.

 
\begin{theorem}
 \label{regthm1} 
 For $k_1,\, k_2\in \mathbb{R}^+$, 
 $K = \begin{pmatrix} k_1 & 0 \\ 0 & k_2 \end{pmatrix}$, and $f \in H^r_\alt(\Omega)$, $r \geq \, - \alt$,
 there exists a unique solution $u \in H^{r+\alpha}_\alt(\Omega)$ satisfying \eqref{prob3}.
\end{theorem}
\textbf{Proof}: The existence and uniqueness of $u \in H^{\alt}_{\alt}(\Omega)$ was established in Theorem \ref{LMthm}. Now,
write $f(x) = f_{-1}(x) + f_1(x) \in H^r_\alt(\Omega)$, where 
$f_\mu(x)  = \sum_{l,n}f_{l,n,\mu}\P_{l,n,\mu}^{(\alt)}(x)$, and $u(x) = u_{-1}(x) + u_1(x)$ where
$u_\mu(x) = \sum_{l,n}u_{l,n,\mu}\P_{l,n,\mu}^{(\alt)}(x)$ for $\mu = \pm 1$. We focus our attention on the analysis for
$f_{1}$ and $u_{1}$, with the results extending in an analogous manner to $f_{2}$ and $u_{2}$.

Let $h^2_{l,n} \, = \, \|\P_{l,n,\mu}^{(\alt)}(x)\|^2_{L^2_{\alt}}$. Then,
\begin{align}
\hspace{-0.95cm}\|u_1\|^2_{H^{r+\alpha}_\alt} &= \sum_{l,n} (n+1)^{r+\alpha}(n+l+1)^{r+\alpha}u_{l,n,1}^2h^2_{l,n} \nonumber \\
\text{(from (\ref{DlnFtscaling}))}\ & = \sum_{l,n}(n+1)^{r+\alpha}(n+l+1)^{r+\alpha}\Big(\frac{\Gamma(n+1)}{\Gamma(n+1+\alt)} d_{l,n,1}\Big)^2h_{l,n}^2  \nonumber  \\
& \leq 4 \sum_{l,n} (n+1)^{r}(n+l+1)^{r+\alpha}d_{l,n,1}^2h^2_{l,n}  \ \mbox{ (using \eqref{stirling})} \, .  \label{poil2}
\end{align}
From \eqref{ghrt1} and \eqref{changeofbasis} it does not follow that
$d_{l,n,1}^2 \ \leq \ c_{\lambda min}^{-2} \wt{f}^2_{l,n}$ for all $l, \, n$. However, 
for W a diagonal matrix, with $w_{i i} > 0$, note that
\[
W \, A_{3, k, m} W^{-1} \,  W \, \bf{e}_{3, k, m} \ = \ W \, \bf{b}_{3, k, m} , \ k \in \{4 , \, 5\} .
\]
With $\mathcal{A} \ = \ W \, A_{3, k, m} W^{-1}$, note that $\mathcal{A}$ is symmetric, positive definite, 
and has the same eigenvalues as $A_{3, k, m}$.
Thus, $c_{\lambda max}^{-1} \, \leq \, \| \mathcal{A}^{-1} \|_{2} \, \leq \, c_{\lambda min}^{-1}$, and subsequently,
$\| W \, \bf{e}_{3, k, m} \|^{2} \, \leq \, c_{\lambda min}^{-1} \, \| W \, \bf{b}_{3, k, m} \|^{2}$. 
This enables us to handle the weights $(n+1)^{r}(n+l+1)^{r+\alpha} \, h^2_{l,n}$ in \eqref{poil2}. Thus,
\begin{equation}
\|u_1\|^2_{H^{r+\alpha}_\alt} \leq 4 \, \sum_{l,n}(n+1)^r(n+l+1)^{r+\alpha}
\Big(\frac{2^{-(\alpha -2)}\Gamma(n+l+1)}{\Gamma(n+1+\alt + l)} f_{l,n,1}\Big)^2 \, h^2_{l,n}. 
\label{jhgp1}
\end{equation}
Again using \eqref{stirling} and simplifying,
\[
\|u_1\|^2_{H^{r+\alpha}_\alt}  \leq 16 \,  2^{-2(\alpha-2)} \, \sum_{l,n} (n+1)^r(n+l+1)^r \, f_{l,n,1}^2 \, h^2_{l,n}
 =  256 \, 2^{-2 \alpha} \|f_1\|^2_{H^r_\alt} < \infty. 
\]
Hence, $u_1 \in H^{r+\alpha}_\alt(\Omega)$. Similarly,  $u_2 \in H^{r+\alpha}_\alt(\Omega)$ and, consequently, 
$u(x) \in H^{r+\alpha}_\alt(\Omega)$.  \\
\mbox{  }  \hfill \qed

\section{Approximation to the Steady-State Problem}
\label{steady_approximation}
In this section we investigate finite dimensional approximations to $u$ satisfying \eqref{prob3}. We begin by
establishing existence and uniqueness of such approximations.

\begin{corollary} 
\label{finite_dimensional_steady}
Let $X_R \subset V$ be a finite-dimensional subspace 
and suppose $K$ and $f$ satisfy the conditions of Theorem \ref{LMthm}. 
Then, there exists a unique $u_R \in X_R$ satisfying for all $v \in X_{R}$
\begin{equation}
B(u_R, v) \ = \ \langle f \, , \, v \rangle_{V' ,  V} \, .
\label{ss_apx}
\end{equation} 
\end{corollary}
\textbf{Proof}:
The continuity and coercivity of $B(\cdot, \cdot)$ on $X_R \times X_R$ follow directly from $X_R\subset  V$. 
Applying the Lax-Milgram theorem establishes the existence and uniqueness of $u_R$. \\
\mbox{ } \hfill \qed

\textbf{Approximation Space $X_R$}\\
For $R\in\N$, let the approximation space $X_R \subset V$ be defined as
the span of basis polynomials whose radial degree is at most $R$. 
Formally, if $R$ is odd, 
\begin{equation}
X_R =  \text{span}\Big\{ \bigcup_{l = 0}^{R} \bigcup_{n = 0}^{\lfloor \frac{l}{2}\rfloor} \{\P^{(\alt)}_{l - 2n, n, 1}(x) \} \cup \bigcup_{l=1}^{R+1}\bigcup_{n=0}^{\lfloor \frac{l-1}{2}\rfloor} \{\P^{(\alt)}_{l-2n,n,-1}(x) \} \Big\},
\label{defR}
\end{equation}
and if $R$ is even,
\begin{equation}\notag
X_R =  \text{span}\Big\{ \bigcup_{l = 0}^{R+1} \bigcup_{n = 0}^{\lfloor \frac{l}{2}\rfloor} \{\P^{(\alt)}_{l - 2n, n, 1}(x) \} \cup \bigcup_{l=1}^{R+2}\bigcup_{n=0}^{\lfloor \frac{l-1}{2}\rfloor} \{\P^{(\alt)}_{l-2n,n,-1}(x) \} \Big\}.
\end{equation}

\begin{theorem}
\label{errorthm1} 
Suppose $f \in H^r_{\alt}(\Omega)$ for some $r\geq -\alt$. Let $u \in  V$ be the unique solution 
to \eqref{prob3} and $u_R \in X_R$ be the solution to \eqref{ss_apx}. 
Then, for any $s\in \R$ such that $s \leq \alpha + r $, we have 
\[
\|u - u_R \|_{H^s_{\alt}} \leq 2^{(8 - 2 \alpha)} \, \, \Big (\frac{1}{R + 2}\Big)^{(\alpha + r  - s)/2} \, \|f \|_{H^r_{\alt}}.
\]
\end{theorem}
\textbf{Proof}:
Without loss of generality, assume $R$ is odd and write
\[
u(x) \, = \, \sum_{l, n, \mu} u_{l,n,\mu} \, \P^{(\alt)}_{l,n,\mu}(x) \, .
\]
Since $ X_R \subset  V$, we can express the error as 
\begin{align}
\|u - u_R \|^2_{H^s_{\frac{\alpha}{2}}} &= \sum_{l=R+1}^\infty \sum _{n=0}^{\lfloor \frac{l}{2}\rfloor}(n+1)^s(n+(l-2n)+1)^su_{l-2n,n,1}^2h^2_{l-2n,n}  \nonumber  \\
& \quad + \sum_{l = R+2}^\infty \sum_{n=0}^{\lfloor \frac{l-1}{2}\rfloor} (n+1)^s(n+(l-2n)-1)^su^2_{l-2n,n,-1}h^2_{l-2n,n} 
 \nonumber \\
& \coloneqq S_{1} + S_{-1} \, .  \label{poiu1}
\end{align}
Considering $S_1$, and using \eqref{DlnFtscaling},
\begin{align*}
S_1 & = \sum_{l = R+1}^\infty \sum_{n=0}^{\lfloor \frac{l}{2}\rfloor}(n+1)^s(n+(l-2n)+1)^s\Big(\frac{\Gamma(n+1)}{\Gamma(n+1+\frac{\alpha}{2})}d_{l-2n,n,1}\Big)^2h_{l-2n,n}^2\\
& \leq 2^{2} \, \sum_{l = R+1}^\infty\sum_{n=0}^{\lfloor \frac{l}{2}\rfloor}(n+1)^{s-\alpha}(n+(l-2n)+1)^sd^2_{l-2n,n,1}h^2_{l-2n,n}\ \mbox{ (using \eqref{stirling})} \\
&\leq2^{2} \,  \sum_{l = R+1}^\infty\sum_{n=0}^{\lfloor \frac{l}{2}\rfloor}(n+1)^{s-\alpha}(n+(l-2n)+1)^s \Big(2^{-(\alpha-2)}\frac{\Gamma(n+1+l-2n)}{\Gamma(n+1+\frac{\alpha}{2}+l-2n)} f_{l-2n,n,1} \Big)^2 h^2_{l-2n,n}  \,  ,
\end{align*}
where the last step follows by the same argument that justified equation \eqref{jhgp1} following from
\eqref{poil2}.

Using \eqref{stirling} we then obtain 
\begin{align*}
S_1 & \leq 2^{-2(\alpha-2)} \, 2^{4} \,  \sum_{l = R+1}^\infty\sum_{n=0}^{\lfloor \frac{l}{2}\rfloor}(n+1)^{s-\alpha}(n+(l-2n)+1)^{s-\alpha}f^2_{l-2n,n,1}h^2_{l-2n,n} \, .
\end{align*}
Now, for $l \geq (R+1)$ and $0\leq n \leq \lfloor \frac{l}{2}\rfloor$, we have that $(n+1)^{s-\alpha -r } (l-n+1)^{s-\alpha-r} \leq (R+2)^{s-\alpha -r}$. Hence, 
\begin{align}
S_1 &\leq 2^{(8 - 2 \alpha)} \, 
\sum_{l = R+1}^\infty\sum_{n=0}^{\lfloor \frac{l}{2}\rfloor}(n+1)^{s-\alpha-r}(n+(l-2n)+1)^{s-\alpha-r} \cdot \nonumber \\
& \hspace*{2.3in} \Big((n+1)^r(n+(l-2n)+1)^rf^2_{l-2n,n,1}h^2_{l-2n,n}\Big)   \nonumber \\
& \leq 2^{(8 - 2 \alpha)} \,
\sum_{l = R+1}^\infty\sum_{n=0}^{\lfloor \frac{l}{2}\rfloor}(R+2)^{s-\alpha-r}(n+1)^r (n+(l-2n)+1)^r f^2_{l-2n,n,1}h^2_{l-2n,n}
\nonumber \\
& \leq \ 2^{(8 - 2 \alpha)} \,
\Big(\frac{1}{R+2}\Big)^{\alpha + r -s}
 \sum_{l =0}^\infty\sum_{n=0}^{\infty}(n+1)^r(n+l+1)^rf^2_{l,n,1}h^2_{l,n}  \, .  \label{poiu2} 
\end{align}
An analogous analysis establishes
\begin{equation}
S_{2} \ \leq \  2^{(8 - 2 \alpha)} \, \Big(\frac{1}{R+3}\Big)^{\alpha + r -s}
 \sum_{l =1}^\infty\sum_{n=0}^{\infty}(n+1)^r(n+l+1)^rf^2_{l,n, -1}h^2_{l,n}  \, .  \label{poiu3} 
\end{equation}
Finally, combining \eqref{poiu1}, \eqref{poiu2} and \eqref{poiu3} we obtain
\begin{align*}
\|u - u_R \|^2_{H^s_{\frac{\alpha}{2}}} &\leq \
2^{(8 - 2 \alpha)} \, \Big(\frac{1}{R+2}\Big)^{\alpha + r -s}
 \sum_{l  , n ,  \mu} (n+1)^r(n+l+1)^rf^2_{l, n, \mu}h^2_{l,n}    \\
&= \ 2^{(8 - 2 \alpha)} \, \Big(\frac{1}{R+2}\Big)^{\alpha + r -s} \, \| f \|^{2}_{H^{r}_{\alt}} \, .
\end{align*}
\mbox{ } \hfill \qed

\section{Existence and Uniqueness to the Time-Dependent Problem}
\label{evExEq}
In this section we extend problem \eqref{problem} to the time dependent setting and consider
\begin{equation}
\begin{cases}\label{pde0}
\wt{u}_{t}(x, t) \, - \, \grad\cdot(-\Delta)^\frac{\alpha-2}{2}K \grad \wt{u}(x, t) &= \ f(x, t),\ \  (x,t) \in \Omega \times (0,T],\\
 \hfill \wt{u}(x,t) &= \ 0, \ \  (x,t) \in \real^{2} \backslash \Omega  \times (0,T],\\
\hfill  \wt{u}(x,0) &= \ \left\{ \begin{array}{rl}
                                                    g(x), &  x \in \Omega  \\
                                                    0, & x \in \real^{2} \backslash \Omega \ \ \, . 
                                           \end{array}  \right.
\end{cases}
\end{equation}
where again we assume $K(x) = \begin{bmatrix} k_1 & 0 \\ 0 & k_2\end{bmatrix},$ with $k_1, k_2 >0$. Again, to accurately 
capture the boundary behavior of the solution we let $\wt{u}(x, t) \, = \, \omega^{\alt} u(x, t)$, which leads to the
system of equations
\begin{equation}
\begin{cases}\label{pde}
\omega^\alt u_t(x,t) - \grad\cdot(-\Delta)^\frac{\alpha-2}{2}K \grad \omega^\alt u(x,t) &= \ f(x, t),\ \  (x,t) \in \Omega \times (0,T],\\
\hfill \omega^\alt  u(x,t) &= \ 0, \ \  (x,t) \in \real^{2} \backslash \Omega \times (0,T],\\
\hfill \omega^\alt u(x,0) &= \ \left\{ \begin{array}{rl}
                                                    g(x), &  x \in \Omega  \\
                                                    0, & x \in \real^{2} \backslash \Omega \ \ \, .
                                                    \end{array} \right.
\end{cases}
\end{equation}

 \begin{defn}
 For $f \in L^2(0,T ; L^{2}(\Omega))$ and $g \in L^2_\alt(\Omega)$,
we say that $u  \in L^2(0,T ; H^\alt_\alt(\Omega))$, with 
$u_t \in L^2(0,T; H^{-\alt}_\alt(\Omega))$, is a weak solution to (\ref{pde}) if for all $v \in H^\alt_\alt(\Omega)$
\begin{align}
(i) \  \langle \omega^{\alt} u(x, t)_t , v(x)  \rangle_{H^{-\alt}_{\alt} , H^{\alt}_{\alt} } \ + \ B(u(x, t) , v(x) ; t) &= \
(f(x, t) , v(x))_{L^{2}_{\alt}}  \, ,  
\ \  a.e. \ t \in (0 , T]  \, ,   \label{wsoltdp1}  \\
(ii)  \hspace{2.0in}  \mbox{ and }  \ \ \omega^{\alt} u(x, 0) &= \ g(x) \, , \ x \in \Omega \, ,  \label{wsoltdp2} 
\end{align}
with $B(\cdot, \cdot; t)$ defined in \eqref{defBl}.
\end{defn}

The analysis in this section follows the framework for the analysis of parabolic problems in \cite{ern041}. We investigate existence and uniqueness for the case where $g(x) = 0$, as the non-homogeneous initial data case follows directly by  a standard change of variable argument. 

\subsection{Notation and Preliminaries}
\label{ssec_NandP}
\begin{align*}
& \mbox{Recall } \ \ V  =  H^\alt_\alt(\Omega) \, ,  \ \  V'  =  H^{-\alt}_\alt(\Omega) \, ,   \ \mbox{and let} \ \  
Y \, := \,  L^{2}(0 , T ; \, V)  \, ,  \ \ Y' \, = \,  L^{2}(0 , T ; \, V') \, , \\
& X \, := \, \big\{ u \in L^{2}(0 , T ; \, V) \, \st \, \partial_{t} u \in L^{2}(0 , T ; \, V') , 
 \  \ \omega^{\alt} u(\cdot, 0) \, = \, 0  \big\} \, ,  
\end{align*}
with associated norms,
\begin{align*}
\| u \|_{X}^{2} := \int_{0}^{T} \big( \| u(t) \|_{V}^{2} \ + \  \| \partial_{t} u(t) \|_{V'}^{2} \big) \, dt \, , & \ \
\| z \|_{Y}^{2} := \int_{0}^{T}  \| z(t) \|_{V}^{2} \, dt \, ,   \\
\mbox{and  for } \ g \in  Y' , \ z \in Y ,  \ \ \langle g , z \rangle_{Y' , Y} &= \ \int_{0}^{T} \langle g(t) \, , \, z(t) \rangle_{V' , V} \, dt \, .
\end{align*}

\textbf{Remark}: For $u \in L^{2}(0 , T ; \, V)$,  and $\partial_{t} u \in L^{2}(0 , T ; \, V')$ from Lemma \ref{lemSpace} 
$\omega^{\alt} u  \in L^{2}(0 , T ; \, V)$, and from Lemma \ref{lmafftf}  and Lemma \ref{lmafder1}  
$\omega^{\alt} \partial_{t} u \in L^{2}(0 , T ; \, V')$. 
Then, from \cite[Lemma 64.40]{ern211}
it follows that $\omega^{\alt} u  \in C([0 , T] ; \, L^{2}_{\alt}(\Omega))$. Hence the initial condition in $X$ is well posed.

\begin{defn} (A function times a functional.) \\
For $\xi \in V'$, $v \in V$ define $\omega^{\alt} \xi$ as: 
\begin{equation}
 \langle \omega^{\alt} \xi \, , \, v \rangle_{V' , V} \, := \, \langle \xi \, , \,  \omega^{\alt} v \rangle_{V' , V} \, .
\label{defftf}
\end{equation}
\end{defn}
The well posedness of \eqref{defftf} is verified by Lemma \ref{lmafftf}  in the appendix.

\begin{lemma} \label{lmafder1}
For $f \in X$, it follows that $\partial_{t} \, \omega^{\alt} f \ = \ \omega^{\alt} \partial_{t} f$.
\end{lemma}
The proof of Lemma \ref{lmafder1} is given in the appendix.

\subsubsection{Bilinear forms and linear operators}
\label{sssec_bflo}
We extend the definition of $B(\cdot , \cdot)$ in \eqref{defB} to: for $t \in (0,T]$, let 
$B(\cdot , \cdot ; t) \colon V \times V \rightarrow \real$ be defined by
\begin{equation}
B(u(x, t) , z(x, t) ;  t) \coloneqq \Big\langle  K \, (-\Delta)^\frac{\alpha-2}{2} \grad \omega^{\alt} u(x,t) ,
\grad \omega^{\alt} z(x, t)\Big\rangle_{H^{1 - \alt}_{\alt - 1} , H^{-(1 - \alt)}_{\alt - 1} }^{*} \, ,  \label{defBl}  
\end{equation}
and let
$b(\cdot , \cdot) \colon X \times Y \rightarrow \real$ be defined by
\begin{equation}
b(u(x, t) , z(x, t)) \coloneqq  
\int_0^T \Big( \langle \omega^\alt \partial_t u(x, t) \,  , \, z(x, t) \rangle _{V' , V} + B(u(x, t) , z(x, t) ; t) \Big) \, dt.  \label{defbl}
\end{equation}
Introduce the linear operators, $A(t)$ defined by: for $g , \, h \in V$
\begin{equation}
   \langle A(t) g , h \rangle_{V' , V} \, \coloneqq 
  \Big\langle  K \, (-\Delta)^\frac{\alpha-2}{2} \grad \omega^{\alt} g(x) ,
\grad \omega^{\alt} h(x) \Big\rangle_{H^{1 - \alt}_{\alt - 1} , H^{-(1 - \alt)}_{\alt - 1} }^{*} \ = \ B(g , h ;  t) \, ,
 \label{defAo1}
\end{equation}
and $\mc{A} \, : \, X \rightarrow Y'$ defined by:  for $u \in X$  and  $z \in Y$
\begin{equation}
  \langle \mc{A} u , z \rangle_{Y' , Y} \, \coloneqq 
  b(u , z) \, .
 \label{defmcAo1}
\end{equation}

\subsubsection{Properties of the bilinear forms and linear operators}
\label{sssec_pbflo}
\begin{lemma} \label{pBf1}
The bilinear $B( \cdot , \cdot ; t)$ defined by \eqref{defBl} satisfies: for all $g , \, h \in V$
\begin{equation}
 C_{coe} \| g \|_{V}^{2} \leq B(g , g ; t) \, , \ \mbox{and } \ | B(g , h ; t) | \leq C_{cts} \| g \|_{V} \| h \|_{V} \, .  \label{propB}
\end{equation} 
\end{lemma}
Properties \eqref{propB} are just a restatement of \eqref{Bcts} and \eqref{coercivity}.

\begin{lemma}  \label{pAo1}
The linear operator $A$ defined by \eqref{defAo1} satisfies: for all $g \in V$
\begin{equation}
 C_{coe} \| g \|_{V}^{2} \leq  \langle A(t) g , g \rangle_{V' , V}  \, , \ \mbox{ and } 
 \|A(t)\|_{\mc{L}(V , V')}  \leq C_{cts}  \, .   \label{propAo1}  
\end{equation} 
In addition, $A(t)^{-1} \, : \, V' \rightarrow V$ exists and satisfies: for all $\xi \in V'$
\begin{equation}
 \frac{C_{coe}}{C_{cts}^{2}} \| \xi \|_{V'}^{2} \leq  \langle \xi , A^{-1}(t) \xi \rangle_{V' , V}  \, , \ \mbox{ and } 
 \|A^{-1}(t)\|_{\mathcal L(V' , V)}  \leq \frac{1}{C_{coe}}  \, .   \label{propAinvo1}  
\end{equation} 
\end{lemma}
The proof of Lemma \ref{pAo1} is given in the appendix.

\begin{lemma} \label{binfsup}
The bilinear form $b(\cdot , \cdot)$ defined in \eqref{defbl} is continuous and satisfies an inf-sup condition. Specifically, there 
exists a constant $\tilde{C} > 0$ such that for all $u \in X$
\begin{equation}
| b(u , v) | \, \leq \, \sqrt{2} (1 + C_{cts}) \| u \|_{X} \, \| v \|_{Y} \, , \mbox{ for all } v \in Y , \ \mbox{ and }
\sup_{0 \neq v \in Y} \frac{ | b(u , v) | }{\| v \|_{Y}} \, \geq \, \frac{C_{coe} \tilde{C}}{2 \, C_{cts}} \| u \|_{X} \, .
\label{binfeq}
\end{equation}
\end{lemma}
\textbf{Proof}:  To establish that $b(\cdot , \cdot)$ is bounded, consider
\begin{align*}
|b(u,v)| &= \ 
\Big| \int_0^T \Big( \langle \omega^\alt \partial_t u(t) \,  , \, v(t) \rangle _{V' , V} + B(u(t) , v(t) ; t) \Big) \, dt \Big|  \\
& \leq \int_0^T \Big( \| \omega^\alt \partial_t u(t)\|_{V'} + C_{cts} \|u(t)\|_{V} \Big)\|v(t)\|_{V} \, dt\\
& \leq \sqrt{2}(1+C_{cts})\Big(\int_0^T  \| \omega^\alt \partial_t u(t)\|^{2}_{V'} + \|u(t)\|^2_{V}\, dt\Big)^{1/2}
\Big(\int_0^T \|v(t)\|^2_{V}\, dt \Big)^{1/2}\\
& =  \sqrt{2}(1+C_{cts})\|u\|_X \|v\|_Y.
\end{align*}

To establish the \textit{inf-sup} property, let $u\in X$, 
$\mu = ( \frac{C_{cts}^{4}}{2 \, C_{coe}^{3}} + \frac{C_{coe}}{2 \, C_{cts}}) / C_{coe} \,  > \, 0$, and fix $v$ to be 
\begin{equation}
v = A^{-1}(t) \omega^\alt \partial_t u \, + \, \mu \, u \, \in Y.
\label{v4binf}
\end{equation}

Note that 
\begin{align}
\notag \|v\|^2_Y &\leq 2\int_0^T \big( \|A^{-1}(t) \omega^\alt \partial_t u(t)\|^2_{V} + \mu^2 \|u(t)\|_{V}^2 \big) \, dt  \\
\notag & \leq 
\Big(\frac{2}{C_{coe}^2} + 2\mu^2\Big)\Big(\int_0^T \big( \| \omega^\alt \partial_t u(t)\|^2_{V'} +  \|u(t)\|_{V}^2 \big) \, dt  \Big) \, , \\
\mbox{i.e., } \ \|v\|_Y & \leq  \ \Big(\frac{2}{C_{coe}^2} + 2\mu^2\Big)^{1/2}\|u\|_X \coloneqq \wt C\|u\|_X \, . \label{ctilde}
\end{align}

With $v$ given by \eqref{v4binf},
\begin{align}
b(u , v) &= \  \int_{0}^{T} 
\Big( \langle \omega^{\alt} \partial_{t} u(t) \, , \, A^{-1}(t) \omega^\alt \partial_t u(t) \, + \, \mu \, u(t) \rangle_{V' , V}  + \ 
B(u(x, t) \, , \, A^{-1}(t) \omega^\alt \partial_t u(t) \, + \, \mu \, u(t) ) \Big) \, dt\nonumber \\
&= \ \int_0^T 
\Big( \langle \omega^{\alt} \partial_{t} u(t) \, , \, A^{-1}(t) \omega^\alt \partial_t u(t) \, + \, \mu \, u(t) \rangle_{V' , V} \nonumber \\
& \hspace{1.0in} \ + \ 
\langle A(t) u(t) \, , \, A^{-1}(t) \omega^\alt \partial_t u(t) \, + \, \mu \, u(t) \rangle_{V' , V} \Big) \, dt  \nonumber \\
&= \ \int_0^T 
\Big( \langle \omega^{\alt} \partial_{t} u(t) \, + \, A(t) u(t) \ , \ A^{-1}(t) \omega^\alt \partial_t u(t) \, + \, \mu \, u(t) \rangle_{V' , V}  \Big) \, dt  \, .\label{bint}
\end{align}

Next we examine each of the terms in \eqref{bint} individually. Using integration by parts \cite[Lemma 64.40]{ern211} and
Lemma \ref{lmafder1},
\begin{align*}
\int_0^T \langle  \omega^\alt  \partial_t u(t) \, , \, \mu \, u(t) \rangle_{V', V} \, dt 
\ = \ ( \omega^\alt &u(T)\, ,\, \mu \, u(T)) _{L^2_\alt} - ( \omega^\alt u(0)\, ,\, \mu \, u(0))_{L^2_\alt} \\
& - \int_0^T \langle \omega^\alt  \partial_t u(t) \, ,\, \mu \, u(t)\rangle_{V', V}\, dt .
\end{align*}
Thus, 
\begin{equation}
 \int_0^T  \langle  \omega^\alt  \partial_t u(t) \, , \, \mu \, u(t) \rangle_{V', V} \, dt 
\ = \ \frac{\mu}{2} ( \omega^\alt u(T)\, ,\, u(T))_{L^2_\alt} \ \geq \, 0 \, . 
 \label{int1}
\end{equation}

Next, using \eqref{propAinvo1},
\begin{equation}\label{int2}
\int_0^T \langle \omega^\alt \partial_t  u(t) \, ,\, A^{-1}(t) \omega^\alt \partial_t  u(t)  \rangle _{V', V}\, dt  \ \geq \ 
\frac{C_{coe}}{C_{cts}^2} \, \int_0^T \| \omega^\alt \partial_t  u(t)  \|^2_{V'} \, dt \, .
\end{equation}

Also, using \eqref{propAo1},
\begin{equation}\label{int3}
\int_0^T \langle A(t) u(t) \, ,\, \mu \, u(t) \rangle_{V', V} \, dt \ \geq \ C_{coe} \, \mu \int_0^T \|u(t)\|^2_{V} \, dt \, .
\end{equation}

Finally, using \eqref{propAo1}, \eqref{propAinvo1}, and Young's inequality,
\begin{align}
\notag\int_0^T \langle A(t) u(t) \, ,\, A^{-1}(t) \omega^\alt \partial_t u(t) \rangle_{V', V} \, dt & \ 
\leq \frac{C_{cts}}{C_{coe}} \, \int_0^T \|u(t)\|_{V} \, \|\omega^\alt \partial_t  u(t) \|_{V'} \, dt  \\
 \leq \ \frac{C_{coe}}{2 \, C_{cts}^{2}} &  \int_0^T \| \omega^\alt \partial_t u(t) \|^2_{V'}\, dt \
 + \ \frac{C_{cts}^{4}}{2 \, C_{coe}^{3}} \int_0^T \|u(t)\|^2_{V} \, dt  \, . \label{int4} 
\end{align}

Substituting \eqref{int1} - \eqref{int4} into \eqref{bint}, and using the value for $\mu$
\begin{align}
\notag b(u,v) & \geq  \Big(\frac{C_{coe}}{C_{cts}^2} - \frac{C_{coe}}{2 \, C_{cts}^2} \Big) 
\int_0^T \| \omega^\alt \partial_t u(t) \|^2_{V'}\, dt 
\ + \ \Big(C_{coe} \, \mu - \frac{C_{cts}^{4}}{2 \, C_{coe}^{3}} \Big) \int_0^T \| u(t) \|^2_{V}\, dt \\
& = \frac{C_{coe}}{2C_{cts}} \Big(\int_0^T \| \omega^\alt \partial_t u(t) \|^2_{V'} + \|u(t)\|^2_{V}\, dt \Big)
\ = \ \frac{C_{coe}}{2C_{cts}}\|u\|^2_X.
\end{align}

Further, using \eqref{ctilde}, we have 
$b(u,v) \geq \frac{C_{coe} \wt C}{2C_{cts}} \|u\|_X \|v\|_Y$. Hence,
\[
\sup_{0\neq v \in Y} \frac{|b(u,v)|}{\|v\|_Y} \geq  \frac{C_{coe} \wt C}{2C_{cts}}\|u\|_X.
\]
\mbox{ } \hfill \qed

\begin{theorem} \label{propmcA1}
The operator $\mc{A}$ defined by \eqref{defmcAo1} is a bounded, one-to-one mapping,  with closed range, $Range(\mc{A}) \subset Y'$.
\end{theorem}
\textbf{Proof}: 
To establish that $\mc{A}$ is a bounded operator, consider
\begin{align*}
\| \mc{A} \|_{\mc{L}(X , Y')}  &= \ \sup_{0 \neq u \in X} \frac{ \| \mc{A}(u) \|_{Y'}}{ \| u \|_{X}} 
\ = \ \sup_{0 \neq u \in X} \frac{ 1 }{ \| u \|_{X}} 
\sup_{0 \neq v \in Y} \frac{ |  \langle \mc{A}(u) \, , \, v \rangle_{Y' , Y} | }{  \| v \|_{Y}}   \\
&= \ \sup_{0 \neq u \in X} \sup_{0 \neq v \in Y} \frac{ |  b(u \, , \, v ) | }{ \| u \|_{X} \, \| v \|_{Y}}  
 \ \leq  \  \sup_{0 \neq u \in X} \sup_{0 \neq v \in Y} \frac{ \sqrt{2} (1 + C_{cts}) \| u \|_{X} \, \| v \|_{Y} }{ \| u \|_{X} \, \| v \|_{Y}}  \\
&\leq \  \sqrt{2} (1 + C_{cts}) \, .
\end{align*}
Suppose next that $\mc{A}(u_{1}) \, = \, \mc{A}(u_{2})$. Then $\mc{A}(u_{1} - u_{2}) \, = \, 0$ in $Y'$, and using \eqref{binfeq}
\[
0 \, = \, \sup_{0 \neq v \in Y} \frac{ | \langle \mc{A}(u_{1} - u_{2}) \, , \, v \rangle_{Y' , Y} | }{ \| v \|_{Y} } 
\, = \, \sup_{0 \neq v \in Y} \frac{ | b(u_{1} - u_{2} \, , \, v) | }{ \| v \|_{Y} }
\, \geq \, \frac{C_{coe} \tilde{C}}{2 \, C_{cts}} \| u_{1} - u_{2} \|_{X} \, .
\]
Hence $u_{1} - u_{2} \, = \, 0$, which establishes that $\mc{A}$ is one-to-one.

To establish that $Range(\mc{A})$ is closed, suppose that $\{ \mc{A}(u_{i}) \}_{i = 1}^{\infty}$ denotes a Cauchy sequence 
in $Range(\mc{A})$. Note that for $u_{i} \in X$,
\[
 | \langle \mc{A}(u_{i})  \, , \, v  \rangle_{Y' , Y}| \ = \ | b(u_{i} , v) | \ \leq \ \sqrt{2} (1 + C_{cts}) \| u_{i} \|_{X} \, \| v \|_{Y} \, .
 \]
  Hence, $\mc{A}(u_{i})$ is a bounded linear functional on $Y$. 
 To see that $ \{ u_{i} \}_{i = 1}^{\infty}$ is  a Cauchy sequence in $X$, 
 \begin{align*}
 \frac{C_{coe} \, \tilde{C}}{2 \, C_{cts}} \, \| u_{i} - u_{j} \|_{X} 
 &\leq \ \sup_{0 \neq v \in Y} \frac{ | b( u_{i} - u_{j}  \, , \, v ) |}{ \| v \|_{Y}} 
 \ = \ \sup_{0 \neq v \in Y} \frac{ | \langle \mc{A}(u_{i}) - \mc{A}(u_{j}) \, , \, v \rangle_{Y' , Y}  |}{ \| v \|_{Y}}  \\
&\leq \ \sup_{0 \neq v \in Y} \frac{ \| \mc{A}(u_{i}) - \mc{A}(u_{j}) \|_{Y'} \, \| v \|_{Y}  |}{ \| v \|_{Y}} 
\ = \ \| \mc{A}(u_{i}) - \mc{A}(u_{j}) \|_{Y'} \, .
\end{align*}
 
 As $X$ is a closed space, there exists a unique $u \in X$ such that $\{ u_{i} \}_{i = 1}^{\infty}$ converges to $u$ in $X$.
To see that $\{ \mc{A}(u_{i}) \}_{i = 1}^{\infty}$ converges to $\mc{A}(u)$ in $Y'$, note that
\begin{align*}
\| \mc{A}(u_{i}) - \mc{A}(u) \|_{Y'} &= \ 
\sup_{0 \neq v \in Y} \frac{ | \langle \mc{A}(u_{i}) - \mc{A}(u) \, , \, v \rangle_{Y' , Y}  |}{ \| v \|_{Y}}  
\ = \ 
\sup_{0 \neq v \in Y} \frac{ | b(u_{i} - u \, , \, v )  |}{ \| v \|_{Y}}   \\
&\leq \ 
\sup_{0 \neq v \in Y} \frac{ \sqrt{2} (1 + C_{cts}) \,  \| u_{i} - u \|_{Y} \, \| v \|_{Y}  }{ \| v \|_{Y}}   
\ = \  \sqrt{2} (1 + C_{cts}) \,  \| u_{i} - u \|_{Y} \, .
\end{align*}

As $\lim_{i \rightarrow \infty} \{ u_{i} \} \, = \, u$, then $\lim_{i \rightarrow \infty} \{ \mc{A}(u_{i}) \} \, = \, \mc{A}(u) \in Range(\mc{A})$.
Hence $Range(\mc{A})$ is closed. \\
\mbox{ } \hfill \qed

A precise characterization of the $Range(\mc{A}) \subset Y'$ is an open question.

We are now in a position to make an existence and unique statement for the solution of \eqref{pde}.
\begin{corollary} \label{soltpde}
Given $\mc{F} \in Range(\mc{A})$ there exists a unique $u \in X$ satisfying for all $v \in Y$
\[
\int_0^T \Big( \langle \omega^\alt \partial_t u(x, t) \,  , \, v(x, t) \rangle _{V' , V} + 
\Big\langle  K \, (-\Delta)^\frac{\alpha-2}{2} \grad \omega^{\alt} u(x,t) ,
\grad \omega^{\alt} v(x, t)\Big\rangle_{H^{1 - \alt}_{\alt - 1} , H^{-(1 - \alt)}_{\alt - 1} }^{*} \Big) \, dt \ = \ 
\langle \mc{F} , v \rangle_{Y' , Y} \, .
\]
In addition,  $ \| u \|_{X}  \leq \, \frac{2 \, C_{cts}} {C_{coe} \tilde{C}} \, \| \mc{F} \|_{Y'}$.
\end{corollary}
\textbf{Proof}: The existence and uniqueness of the solution $u$ follows from Theorem \ref{propmcA1}. To obtain the stated
bound for $u$, using \eqref{binfeq},
\[
 \frac{C_{coe} \, \tilde{C}}{2 \, C_{cts}} \, \| u \|_{X} \
 \leq \ \sup_{0 \neq v \in Y} \frac{ | b(u , v) | }{\| v \|_{Y}}
  \ = \ \sup_{0 \neq v \in Y} \frac{ | \mc{F}(v) | }{\| v \|_{Y}} 
  \ \leq \ \sup_{0 \neq v \in Y} \frac{ \| \mc{F} \|_{Y'} \, \| v \|_{Y} }{\| v \|_{Y}}   
\  = \   \| \mc{F} \|_{Y'} \, .
\]
\mbox{ } \hfill \qed


\section{Approximation to the Time-Dependent Problem}
\label{time_approx}
In this section we investigate the numerical approximation of \eqref{wsoltdp1}, \eqref{wsoltdp2}. For the temporal
discretization we consider the backward Euler discretization. As used in Section \ref{steady_approximation}, we consider
a spectral approximation for the spatial discretization.

 For $N\in\N$, let $\dt = \frac{T}{N}$, $t_n = n\dt$, $n = 0, \hdots, N$, and for $\mathcal{U}^{(n)}(x) \in X_R$
 (see \eqref{defR}), let
$\mathcal{U}^{(n)}(x) \approx u(x,t_n)$. (Recall $Q = L^{2}_{\alt}(\Omega)$, and $V  = H^\alt_\alt(\Omega)$.)
 Discretizing \eqref{wsoltdp1} in time using the backward Euler approximation
we obtain the approximating equations for $n = 1, 2, \ldots, N$:
 \begin{equation}
\Big( \omega^{\alt} \frac{\U{n}(x) - \U{n-1}(x)}{\Delta t}, v_m \Big)_{Q} + B(\U{n}(x) ,  v_m(x) ; t_{n}) 
= \big( f(x,t_n), v_m(x) \big)_{Q}, \label{approx}
\end{equation} 
where $v_m(x) \in X_R$. 
 
We begin by establishing the existence and uniqueness of  $\U{n}(x)$.
 \begin{lemma} 
 For all $1 \leq n \leq N$, there exists a unique $\mathcal{U}^{(n)}(x) \in X_R$ solving (\ref{approx}).
 \end{lemma} 
\textbf{Proof}:
As in the proof for Theorem \ref{ErrThm}, let $P_R\colon Q \rightarrow X_R$ denote the orthogonal projection of $Q$ onto $X_R$ and $\U{0}(x) \coloneqq P_R(g)(x)$. Fix $1\leq n \leq N$ and assume for induction that $\U{k}(x) \in X_R$ has been determined for all $1\leq k \leq n-1$. Define $a(\cdot,\cdot) \colon X_R \times X_R\rightarrow \R$ and $\ell(\cdot)\colon X_R\rightarrow \R$ by
\begin{align*}
a(z,v)\ &\coloneqq\ (\omega ^\alt z(x), v(x))_{Q} \ +\ \dt \, B(z(x), v(x); t_n),\\
\ell(v)\ & \coloneqq\ \dt (f(x,t_n), v(x))_{Q} \ + \ (\omega^{\alt}\U{n-1}(x),v(x))_{Q}.
\end{align*}
Note that both $a(\cdot, \cdot)$ and $\ell(\cdot)$ are bounded, as for any $z,v\in X_R$, 
\[
a(z,v) \leq \|\omega^\alt z\|_{Q} \|v\|_{Q} + \dt C_{cts}\|z\|_{V}\|v\|_{V} \leq \|z\|_{Q} \|v\|_{Q} + \dt C_{cts}\|z\|_{V}\|v\|_{V} \leq (1+\dt C_{cts})\|z\|_{V}\|v\|_{V}.
\]
and
\[
\ell(v) \leq (\dt \|f(\cdot,t_n)\|_{Q} + \|\omega^\alt\U{n-1}\|_{Q})\|v\|_{V}.
\]
As $X_R \subset V$ is a finite-dimensional subspace, all norms on $X_R$ are equivalent. Hence, as $\|\cdot\|_{L^2_\alpha}$ defines a norm on $X_R$, there exists some constant $C_{X_R}>0$ such that $\|z\|^2_{L^2_\alpha} \geq C_{X_R}\|z\|^2_{V}$ for all $z \in X_R$.
Thus, for any $z\in X_R$, we have
\begin{align*}
a(z,z) & = (\omega ^\alt z(x), z(x))_{Q} \ +\ \dt \, B(z(x), z(x); t_n)\\
& \geq \int_\Omega \omega^\alpha z^2(x) d\Omega + \dt C_{coe} \|z\|^2_{V} = \|z\|^2_{L^2_\alpha} + \dt C_{coe}\|z\|^2_{V}\\
& \geq (C_{X_R} + \dt C_{coe})\|z\|^2_{V}.
\end{align*}
Thus, by Lax-Milgram, there exists a unique $\U{n}(x)\in X_R$ satisfying 
\[
a(\U{n},v_m) = \ell(v_m),\ \forall v_m\in X_R.
\]
\mbox{ } \hfill \qed

 
\begin{theorem}
 \label{ErrThm} 
 Let $u(x,t)$ satisfy \eqref{wsoltdp1}, \eqref{wsoltdp2}
  and $\U{n}(x)$ be given by (\ref{approx}) with $u(x, 0) \, = \, g(x) \, = \, \U{0} \in X_R$, and assume  
  $u \in C^{2}(0 , T ; H^r_{\alpha/2}(\Omega))$ for some $r \geq \alpha/2$. Then at time $t_q$,  $q \in \{1, \hdots, N\}$, 
  with $C_{cts}$, $C_{coe}$, and $C_{\omega}$ given by \eqref{Bcts}, \eqref{coercivity}, \eqref{defComega}, respectively, we have
\begin{align}
\notag& \|\omega^{\alpha/4} (u(\cdot , t_q) - \mathcal{U}^{(q)}) \|^2_{L^{2}_{\alt}}
\ + \ \dt \, C_{coe} \sum_{n=1}^q \| u(\cdot , t_n) - \mathcal{U}^{(n)} \|^2_{H^{\alt}_{\alt}}  \\
\notag& \leq 
\frac{8 (\dt)^2}{3 \, C_{coe}} \|\omega^{\alt} u_{tt}\|^2_{L^2(0 , t_q ; H^{-\alt}_{\alt}(\Omega))} \ + \ 
\frac{16 \, C_{\omega}^2}{(\dt)^2 \, C_{coe}} \Big(\frac{1}{R+2}\Big)^{r+\alpha/2} \, \dt \,
 \sum_{n=1}^{q} \|u(\cdot , t_n)\|^2_{H^r_{\alt}} \\
\notag & \quad \ + \  2 \Big(\frac{1}{R+2}\Big)^{r} \|u(\cdot , t_q)\|^2_{H^{r}_{\alt}} 
\  + \ 2 \Big(\frac{1}{R+2}\Big)^{r-\alpha/2} \Big(\frac{4 \, C_{cts}^2}{C_{coe}} + C_{coe} \Big) \dt \, 
\sum_{n=1}^q \| u(\cdot , t_n)\|^2_{H^{r}_{\alt}}  \, .
\end{align}
\end{theorem}
\textbf{Proof}: 
Let $e(x,t_n) \ = \ u(x,t_n) \, - \, \U{n}(x)$. 
Rewriting \eqref{wsoltdp1} in the form of (\ref{approx}) and then subtracting (\ref{approx}) we obtain, 
\begin{align}
\Big\langle  \omega^{\alt} \frac{e(x , t_n) - e(x , t_{n-1})}{\Delta t} , v_m(x) \Big\rangle_{V', V}
 &+ \ B(e(x , t_n)  ,  v_m(x) ; t_{n})  \nonumber \\
= \
\Big\langle \omega^{\alt} \Big( u_t(x , t_n) & - \, \frac{u(x , t_n) - u(x , t_{n-1}) }{\Delta t} \Big) \, , \, v_m(x) 
\Big\rangle_{V', V} \, .
\label{step1ii} 
\end{align}

Let $P_R\colon Q \ra  X_R$ to be the orthogonal projection of $Q$ onto $X_R$. That is, 
for $z(x) \in Q$
\[
\Big( z(x) \,  -  \, P_R(z)(x) \, , \,   v_{m}(x) \Big)_{ Q} = 0, \ \ \forall \,  v_{m}(x) \in  X_R.
\]   

Introduce  $\xi^{(n)} \in V$ and $E^{(n)} \in X_R$ via
\begin{align}
\notag e(x , t_n) &= \ u(x , t_n) \, -  \, \U{n}(x) \
 = \ (u(x , t_n) \,  - \, P_R( u(\cdot , t_n))(x)) \ + \ (P_R (u(\cdot ,t _n))(x) \, - \, \U{n}(x))   \\
&\coloneqq   \xi^{(n)}(x) \ + \ E^{(n)}(x) \, . \label{star star}
\end{align}

Now, substituting (\ref{star star}) into (\ref{step1ii}) and rearranging we obtain,
\begin{align}
\notag
& \Big( \omega^{\alt} \frac{E^{(n)}(x) - E^{(n-1)}(x) }{\Delta t} ,  v_m(x) \Big)_{Q} \ + 
\ B(E^{(n)}(x) \, , \, v_m(x) ; t_{n})  \\
&= \ 
\Big\langle \omega^{\alt} \Big( u_t(x , t_n) \, - \, \frac{u(x , t_n) - u(x , t_{n-1})} { \Delta t} \Big) \,
 , v_m(x) \Big\rangle_{V', V}  \nonumber \\
& \  - \, \Big\langle \omega^{\alt} \,  \frac{ \xi^{(n)}(x) -  \xi ^{(n-1)}(x)}{\Delta t} , v_m(x) \Big\rangle_{V', V} 
 \ - \ B(\xi^{(n)}(x)  , v_m(x) ; t_{n})  \label{step2} .
\end{align}

Choosing $v_m(x) \, = \, E^{(n)}(x)$, multiplying through by $\Delta t$, and simplifying leads to 
\begin{align}
\notag & ( \omega^{\alt}  ( E^{(n)}(x)  -  E^{(n-1)}(x) ) \, , \,   E^{(n)}(x) )_{Q} \  
+ \ \Delta t \, B(E^{(n)}(x) \, , \, E^{(n)}(x) ; t_{n}) \\
& = \Delta t \, \Big\langle \omega^{\alt} \Big( u_t(x , t_n) \, - \, \frac{u(x , t_n) - u(x , t_{n-1})} { \Delta t} \Big) \,
  , E^{(n)}(x) \Big\rangle_{V', V}   \nonumber \\
& \   - \, \Big\langle \omega^{\alt} \,  ( \xi^{(n)}(x) -  \xi ^{(n-1)}(x) ) \,  , \, E^{(n)}(x) \Big\rangle_{V', V}    
 \ - \ \Delta t  \, B(\xi^{(n)}(x) \, , \, E^{(n)}(x) ; t_{n})  \, .   \label{step3}
\end{align}
Next we consider each term in (\ref{step3}).
\begin{align}
\notag
 & ( \omega^{\alt}  ( E^{(n)}(x)  -  E^{(n-1)}(x) ) \, , \,   E^{(n)}(x) )_{Q} 
\ =  \ (\omega^{\alpha/4}(E^{(n)}(x) - E^{(n-1)}(x)) \, , \, \omega^{\alpha/4}E^{(n)}(x) )_{Q}  \\
&= \frac{1}{2} \|\omega^{\alpha/4}(E^{(n)} - E^{(n-1)})\|_{Q}^2 \ +  \ 
\frac{1}{2} \|\omega^{\alpha/4}E^{(n)}\|_{Q}^2 
\ - \ \frac{1}{2}\|\omega^{\alpha/4}E^{(n-1)}\|_{Q}^2   \, . \label{term1}
\end{align}
Using the coercivity of $B(\cdot , \cdot ; \cdot)$, 
\begin{equation}
B(E^{(n)}(x) \, , \, E^{(n)}(x) ; t_{n}) \geq C_{coe} \, \|E^{(n)}\|^2_{V}  \, . \label{term2}
\end{equation}
Using Lemma \ref{lemInt}, 
\begin{align}
\notag
\Big\langle \omega^{\alt} \Big( u_t(x , t_n) &- \, \frac{u(x , t_n) - u(x , t_{n-1})} { \Delta t} \Big) \, , \, 
 E^{(n)}(x) \Big\rangle_{V', V}  \\
 & \leq \Big\|  \omega^{\alt} \Big( u_t(\cdot , t_n) \,  - \, \frac{u(\cdot , t_n) - u(\cdot , t_{n-1})} { \Delta t} \Big) \Big\|_{V'} 
\,  \|E^{(n)}\|_{V}   \nonumber \\
\notag& \leq \frac{2}{C_{coe}} \, 
\Big\|  \omega^{\alt} \Big( u_t(\cdot , t_n) \,  - \, \frac{u(\cdot , t_n) - u(\cdot , t_{n-1})} { \Delta t} \Big) \Big\|_{V'} ^2
\ +  \ \frac{C_{coe}}{8}\|E^{(n)}\|_{V}^2\\
& \leq \frac{2}{C_{coe}} \, \frac{\Delta t}{3} \, 
\int_{t_{n-1}}^{t_n} \|\omega^{\alt} \, u_{tt}(\cdot , \tau) \|_{V'}^2 \, d\tau 
\ + \ \frac{C_{coe}}{8} \, \|E^{(n)}\|_{V}^2 \, .\label{term3}
\end{align}

By the continuity of $B(\cdot , \cdot \, ; t_{n})$ and Young's inequality, we have for the last term in (\ref{step3}),
\begin{equation}
B(\xi^{(n)}(x) ,  E^{(n)}(x) ; t_{n}) \leq \frac{2 \, C_{cts}^2}{C_{coe}} \, \|\xi^{(n)}\|^2_{V} \
+ \ \frac{C_{coe}}{8}\|E^{(n)}\|^2_{V}.
\end{equation}
For the second term on the RHS of (\ref{step3}), note that 
\begin{align*}
 \big\langle \omega^{\alt} ( \xi^{(n)}(x) - \xi^{(n-1)}(x) )  \, , \, E^{(n)}(x) \big\rangle_{V', V} 
 &\leq  \
\big|  \big\langle \omega^{\alt}  \xi^{(n)}(x)  \, , \, E^{(n)}(x) \big\rangle_{V', V} \big|  \nonumber  \\
& \ \ +
\big|  \big\langle \omega^{\alt}  \xi^{(n-1)}(x)  \, , \,  E^{(n)}(x) \big\rangle_{V', V} \big|.
\end{align*}
Then, using Corollary \ref{corXiBound} and Young's inequality
\begin{align}
\notag
\big\langle \omega^{\alt}  \xi^{(n)}(x)  \, , \, E^{(n)}(x) \big\rangle_{V', V} & 
 \leq \| \omega^{\alt}  \xi^{(n)} \|_{V'} \, \| E^{(n)}\|_{V} \\
\label{eq32v}& \leq \frac{2 \, C_{\omega}^2}{\dt \, C_{coe}} \, \|\xi^{(n)}\|^2_{V'} \ 
+ \frac{\dt \, C_{coe}}{8} \, \| E^{(n)}\|^2_{V}
\end{align}
Similarly, 
\begin{equation}
\big\langle \omega^{\alt}  \xi^{(n-1)}(x)  \, , \, E^{(n)}(x) \big\rangle_{V', V} 
 \leq 
 \frac{2 \, C_{\omega}^2}{\dt \, C_{coe}} \, \|\xi^{(n-1)}\|^2_{V'} \ 
+ \frac{\dt \, C_{coe}}{8} \, \| E^{(n)}\|^2_{V} \, .
\label{eq33v}
\end{equation}
Substituting \eqref{term1} - \eqref{eq33v} into \eqref{step3} and simplifying yields
\begin{align}
\notag \|\omega^{\alpha/4}E^{(n)}\|_{Q}^2  &- \  \|\omega^{\alpha/4}E^{(n-1)}\|_{Q}^2 
\ + \ \|\omega^{\alpha/4}(E^{(n)} - E^{(n-1)})\|_{Q}^2 \ + \ \dt \,  C_{coe} \|E^{(n)}\|^2_{V}\\
\notag& \leq \frac{4 (\Delta t)^2}{3 \, C_{coe}}\int_{t_{n-1}}^{t_n} \|\omega^{\alt} \, u_{tt}(\cdot,\tau)\|^2_{V'} \, d\tau 
\ + \
\frac{4 \, C_{\omega}^2}{\Delta t \, C_{coe}} \, \|\xi^{(n)}\|^2_{V'}\\
\notag& \ \ \ \ \ \  + \frac{4 \, C_{\omega}^2}{\Delta t \, C_{coe}}\, \|\xi^{(n-1)}\|^2_{V'} \ 
+ \ \frac{4\Delta t \, C_{cts}^2}{C_{coe}} \, \|\xi^{(n)}\|^2_{V}  \, .
\end{align}
Summing from $n = 1$ to $n = k$, we have
\begin{align}
\notag \|\omega^{\alpha/4}E^{(k)}\|_{Q}^2 &- \ \|\omega^{\alpha/4}E^{(0)}\|_{Q}^2 
\ + \ \sum_{n=1}^k\|\omega^{\alpha/4}(E^{(n)} - E^{(n-1)})\|_{Q}^2 
\ + \ \Delta t \, C_{coe} \sum_{n=1}^k\|E^{(n)}\|^2_{V} \\
\notag& \leq \frac{4(\Delta t)^2}{3 C_{coe}} \int_{0}^{t_k} \|\omega^{\alt} u_{tt}(\cdot , \tau)\|^2_{V'} d\tau 
\ + \ \frac{8 \, C_{\omega}^2}{\Delta t \, C_{coe}} \sum_{n=1}^{k}\|\xi^{(n)}\|^2_{V'} \\
\notag& \ \ \ \ \ \ + \frac{4 \, C_w^2}{\Delta t \, C_{coe}} \, \|\xi^{(0)}\|^2_{V'} \ + \
 \frac{4  \Delta t \, C_{cts}^2}{C_{coe}} \sum_{n=1}^k \|\xi^{(n)}\|^2_{V} \, .
\end{align}
Since $u(x , 0) \in X_R$, then $\xi ^{(0)} \, = \, u(x , 0) - P_R(u(x , 0)) = 0$. 
Similarly, as $\mathcal{U}^{(0)}(x) \, = \, u(x , 0)$, then $E^{(0)} \, = \, P_R(u(x , 0)) - \mathcal{U}^{(0)}(x) = 0$. 
Thus we have 
\begin{align}
\notag &  \|\omega^{\alpha/4}E^{(k)}\|_{Q}^2 \ + \ 
\sum_{n=1}^k \|\omega^{\alpha/4} (E^{(n)} - E^{(n-1)}) \|_{Q}^2   
 \ + \ \Delta t \,  C_{coe} \sum_{n=1}^k \|E^{(n)}\|^2_{V} \\
& \quad \leq \frac{4 \, (\dt)^2}{3 \, C_{coe}} \|\omega^{\alt}  u_{tt}\|^2_{L^2(0,t_k ; V' )} 
\ + \ \frac{8 \, C_{\omega}^2}{\dt \, C_{coe}} \sum_{n=1}^{k} \|\xi^{(n)}\|^2_{V'}
 + \frac{4 \dt C_{cts}^2}{C_{coe}} \sum_{n=1}^k \|\xi^{(n)}\|^2_{V}   \, .   \label{s2}
\end{align}

From Lemma \ref{lemXi}, 
\begin{equation}
\|\xi^{(q)}\|^2_{V'} \leq \Big( \frac{1}{R+2}\Big)^{r+\alpha/2} \| u(\cdot , t_q)\|^2_{H^r_{\alt}} ,
\ \mbox{and } \
\|\xi^{(q)}\|^2_{V} \leq \Big(\frac{1}{R+2}\Big)^{r-\alpha/2} \|u(\cdot , t_q)\|^2_{H^r_{\alt}} \label{s4}.
\end{equation}
Substituting (\ref{s4}) into (\ref{s2}), 
\begin{align}
\notag  \|\omega^{\alpha/4}E^{(k)}\|_{Q}^2  &+ \ 
 \sum_{n=1}^k \|\omega^{\alpha/4}(E^{(n)} - E^{(n-1)})\|_{Q}^2 
 \ + \ \Delta t \, C_{coe} \sum_{n=1}^k \|E^{(n)}\|^2_{V} \\
\notag& \leq\frac{4 (\dt)^2}{3 \, C_{coe}} \, \|\omega^{\alt} u_{tt}\|^2_{L^2(0,t_k ;  V' )} \
 + \ \frac{8 \, C_{\omega}^2}{\dt \,  C_{coe}} \Big(\frac{1}{R+2}\Big)^{r+\alpha/2} \sum_{n=1}^{k} \| u(\cdot , t_n)\|^2_{H^r_{\alt}}\\
& \ \ \ \ \ \ + \frac{4 \dt \, C_{cts}^2 }{C_{coe}} \Big(\frac{1}{R+2}\Big)^{r-\alpha/2}\sum_{n=1}^k \|u(\cdot , t_n)\|^2_{H^{r}_{\alt}} \, . \label{hr2}
\end{align}

Using the triangle inequality and \eqref{hr2}
\begin{align}
\notag& \|\omega^{\alpha/4} (u(\cdot , t_q) - \mathcal{U}^{(q)}) \|^2_{Q}
\ + \ \dt \, C_{coe} \sum_{n=1}^q \| u(\cdot , t_n) - \mathcal{U}^{(n)} \|^2_{V}  \\
\notag & \leq 2 \|\omega^{\alpha/4}\xi^{(q)}\|^2_{Q} \ + \ 2 \|\omega^{\alpha/4} E^{(q)}\|^2_{Q} 
\ + \ \dt \, C_{coe} \sum_{n=1}^q\big( 2 \|\xi^{(n)}\|^2_{V} \ + \ 2 \|E^{(n)}\|^2_{V}\big) \\
\notag& \leq 2 \| \xi^{(q)}\|^2_{Q} 
\ + \ \frac{8 (\dt)^2}{3 \, C_{coe}} \|\omega^{\alt} u_{tt}\|^2_{L^2(0,t_q ;   V' )} 
\ + \ \frac{16 \, C_{\omega}^2}{\dt \, C_{coe}} \Big(\frac{1}{R+2}\Big)^{r+\alpha/2}\sum_{n=1}^{q} \| u(\cdot , t_n)\|^2_{H^r_{\alt}}\\
\notag&\hspace{1cm} + \ \frac{8 \dt \, C_{cts}^2}{C_{coe}} \Big(\frac{1}{R+2}\Big)^{r-\alpha/2}\sum_{n=1}^q \|u(\cdot , t_n)\|^2_{H^{r}_{\alt}} 
\ + \ 2 \dt \, C_{coe} \sum_{n=1}^q \|\xi^{(n)}\|^2_{V} \\
\notag& \leq 2 \Big(\frac{1}{R+2}\Big)^{r} \|u(\cdot , t_q)\|^2_{H^{r}_{\alt}} \ + \
\frac{8 (\dt)^2}{3 \, C_{coe}} \|\omega^{\alt} u_{tt}\|^2_{L^2(0 , t_q ; V' )} \\
\notag &  \hspace{0.3in}  + \ \frac{16 \, C_{\omega}^2}{(\dt)^2 \, C_{coe}} \Big(\frac{1}{R+2}\Big)^{r+\alpha/2} \, \dt \,
 \sum_{n=1}^{q} \|u(\cdot , t_n)\|^2_{H^r_{\alt}} \\
\notag & \hspace{0.3in} + \ 2 \Big(\frac{1}{R+2}\Big)^{r-\alpha/2} \Big(\frac{4 \, C_{cts}^2}{C_{coe}} + C_{coe} \Big) \dt \, 
\sum_{n=1}^q \| u(\cdot , t_n)\|^2_{H^{r}_{\alt}}  \, ,
\end{align}
where in the last step we have used Lemma \ref{lemXi}.\\ 
\mbox{ } \hfill \qed

For $s \ge 0$, define $\vertiii{g}^2_{H^{s}_{\alt}}\coloneqq \dt \sum_{n=0}^N \|g(t_n)\|^2_{H^{s}_{\alt}}$. 
Note that for $g \in L^2(0,T ;  H^{s}_{\alt}(\Omega))$, 
\[
\vertiii{g}^2_{H^{s}_{\alt}} \ra \| g\|^2_{L^2(0,T ;  H^{s}_{\alt})} <\infty \text{ as } \dt\ra 0 .
\] 
\begin{corollary} \label{cor1}
Under the same assumptions as in Theorem \ref{ErrThm}, we have 
\begin{align*}
\notag & \max_{0\leq q \leq N} \|\omega^{\alpha/4} (u(\cdot , t_q) - \mathcal{U}^{(q)}) \|^2_{L^{2}_{\alt}} \
  +  \ C_{coe} \vertiii{ u \, - \, \mathcal{U}^{(\cdot)}}^2_{H^{\alt}_{\alt}} \\
&\leq \frac{8(\dt)^2}{3 C_{coe}} \|\omega^{\alt} u_{tt}\|^2_{L^2(0,T ; H^{-\alt}_{\alt}(\Omega))} 
+ \frac{16 \, C_{\omega}^2}{(\dt)^2  C_{coe}}\Big(\frac{1}{R+2}\Big)^{r+\alpha/2} \vertiii{u}^2_{H^r_{\alt}}
 + 2\Big(\frac{1}{R+2}\Big)^{r} \max_{0\leq q \leq N}\| u \|^2_{H^{r}_{\alt}}  \\
\notag &\hspace{1cm}+ 2\Big(\frac{1}{R+2}\Big)^{r-\alpha/2} (\frac{4 \, C_{cts}^2}{C_{coe}} + C_{coe})\vertiii{u }^2_{H^{r}_{\alt}}.
\end{align*}
\end{corollary}
\textbf{Proof}: The proof follows directly from Theorem \ref{ErrThm} and the definition of $\vertiii{\cdot}$. \\
\mbox{ } \hfill \qed

\section{Numerical Results}
\label{numerical_results}
In this section, we present four numerical examples. We begin with the steady-state problem \eqref{problem}. In Example \ref{ex1}, we choose a source function $f(x)\in C^\infty(\Omega)$ and vary the value of $\alpha$ to illustrate the influence of $\alpha$ on the solution. 
In Example \ref{ex2}, we choose a source function $f(x)\in H^{7/2- \eps}_\alt(\Omega)$ to compute convergence rates and compare them to the theoretical results from Theorem \ref{errorthm1}. 
The parabolic problem \eqref{pde} is considered in Example \ref{ex3}, where we choose two values of $\alpha$ and present four snapshots of the evolution of the solution. In particular, we choose a function that evolves to the steady-state solution from Example \ref{ex1}. Finally, in Example \ref{ex4}, we compute convergence rates for a solution 
$u(x,t)\in L^2(0,T; H^{7/2-\eps}_\alt(\Omega))$ and compare results to the theoretical results in Theorem~\ref{ErrThm}. 
Throughout these examples, we fixed the matrix $K = \begin{pmatrix} 3 & 0 \\ 0 & 9 \end{pmatrix}$.

\begin{example}\label{ex1}
For the steady state problem \eqref{problem} we take as the source function,
$f\in C^\infty(\Omega)$, given by 
\[
f(r,\phi) = \begin{cases}
	4 \exp\Big\{\frac{-0.2^2}{0.2^2 - r^2}\Big\}, & r < 0.2\\
	0, & r \geq 0.2
\end{cases}.
\]
The solution $\wt{u} \, = \, \omega^{\alt} u$ is approximated by solving \eqref{ss_apx} for $u_{R}$, the 
approximation of $u$.
We demonstrate the effect of $\alpha$ on the solution by computing approximations for $\alpha = 1.1,\, 1.4,\, 1.7,\, 2.0.$ Figure \ref{alphaFig} shows each of these computations, computed with $R$ (the maximum radial degree) equal to 40. 

\begin{figure}[H]
\centering

\begin{subfigure}{0.4\textwidth}
  \centering
  \includegraphics[width=\linewidth]{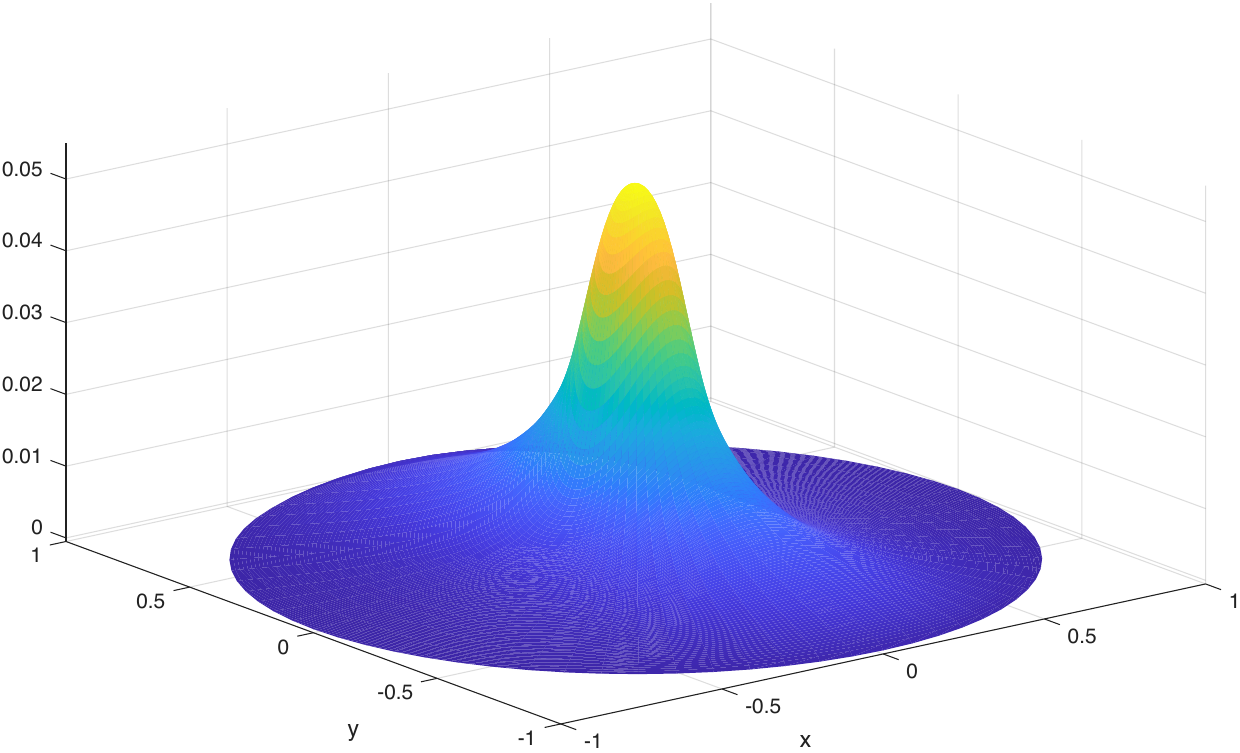}
  \caption{Fractional diffusion with $\alpha = 1.10$}\label{ex1a}
\end{subfigure}
\hspace{0.05\textwidth}
\begin{subfigure}{0.4\textwidth}
  \centering
  \includegraphics[width=\linewidth]{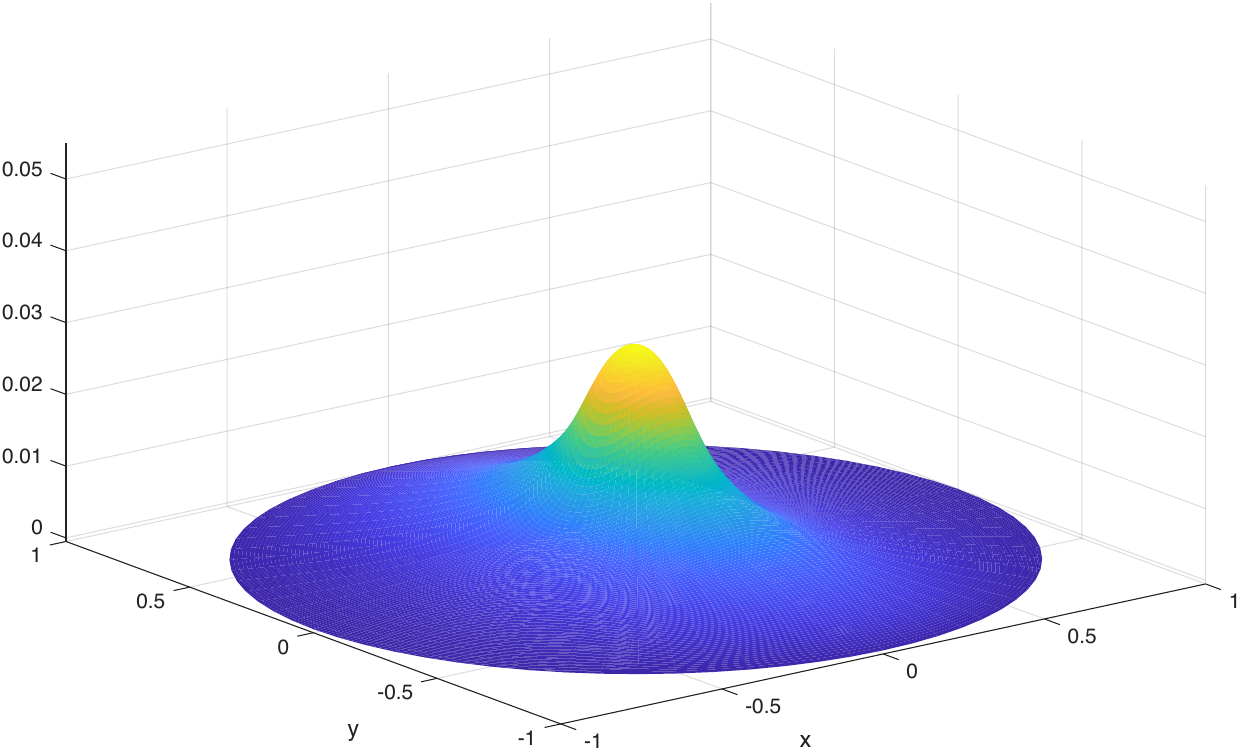}
  \caption{Fractional diffusion with $\alpha = 1.40$}\label{ex1b}
\end{subfigure}

\medskip

\begin{subfigure}{0.4\textwidth}
  \centering
  \includegraphics[width=\linewidth]{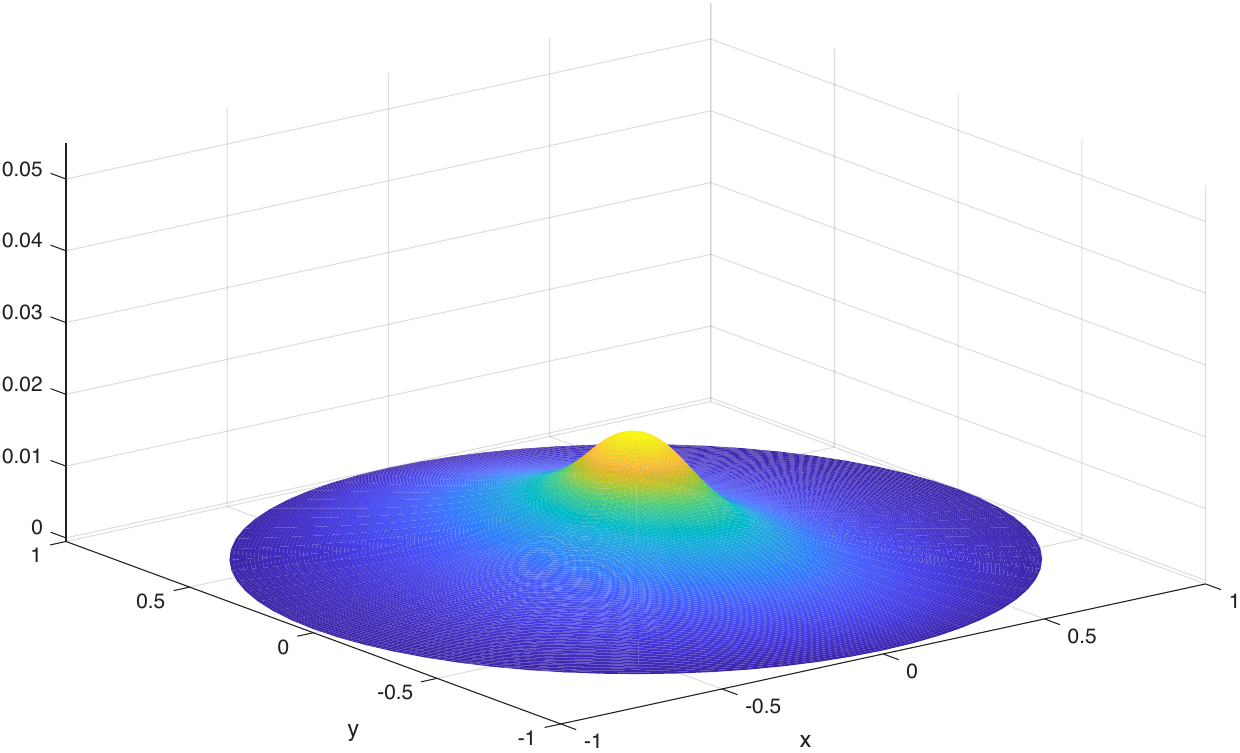}
  \caption{Fractional diffusion with $\alpha = 1.70$}\label{ex1c}
\end{subfigure}
\hspace{0.05\textwidth}
\begin{subfigure}{0.4\textwidth}
  \centering
  \includegraphics[width=\linewidth]{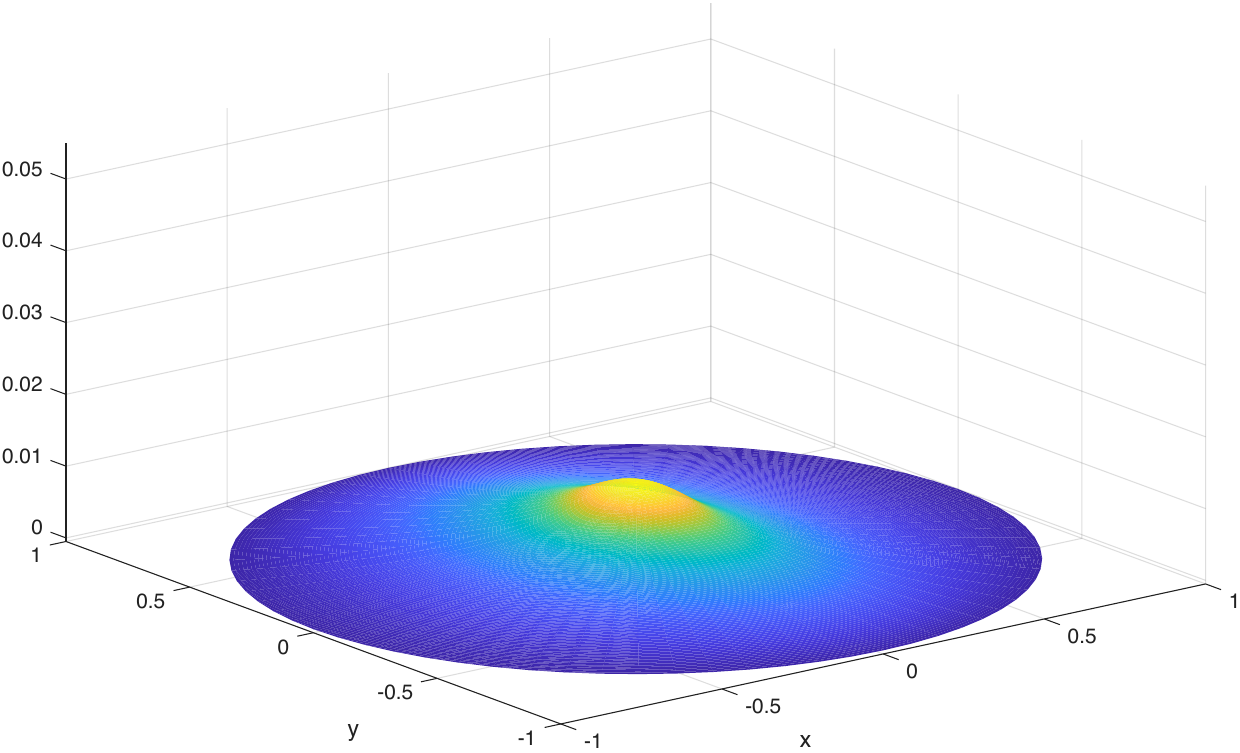}
  \caption{Typical diffusion with $\alpha = 2.00$}\label{ex1d}
\end{subfigure}

\caption{Radial bubble source function with different values of $\alpha$.}
\label{alphaFig}
\end{figure}

To illustrate the influence of $k_{1} \neq k_{2}$ in $K =  \begin{pmatrix} 3 & 0 \\ 0 & 9 \end{pmatrix}$, we also computed cross-sectional views for each of the four figures above. Specifically, in Figure \ref{fig:cross sections}, we computed cross sections along both the $x$-axis and along the $y$-axis.
\begin{figure}[H]
\centering

\begin{subfigure}{0.46\textwidth}
  \centering
  \includegraphics[width=\linewidth]{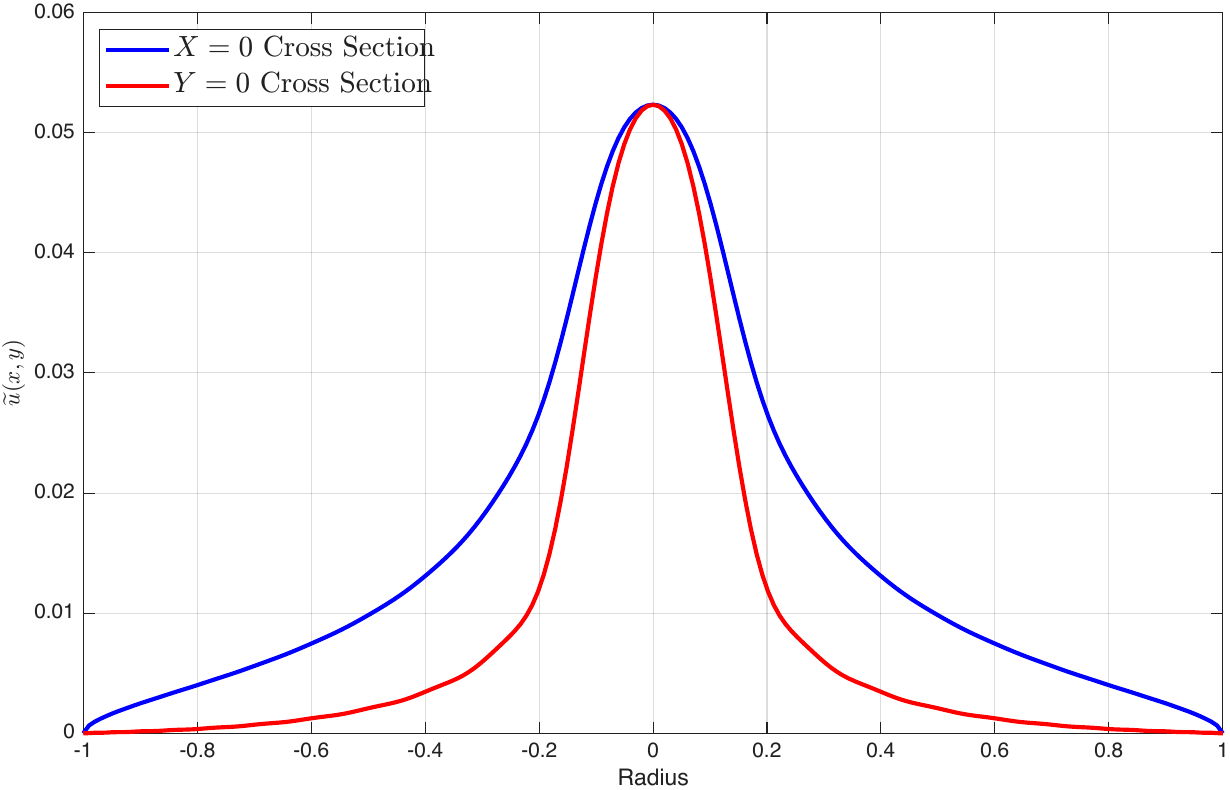}
  \caption{$x$ and $y$ axis cross sections with $\alpha = 1.10$}\label{ex1_cross_a}
\end{subfigure}
\hspace{0.05\textwidth}
\begin{subfigure}{0.46\textwidth}
  \centering
  \includegraphics[width=\linewidth]{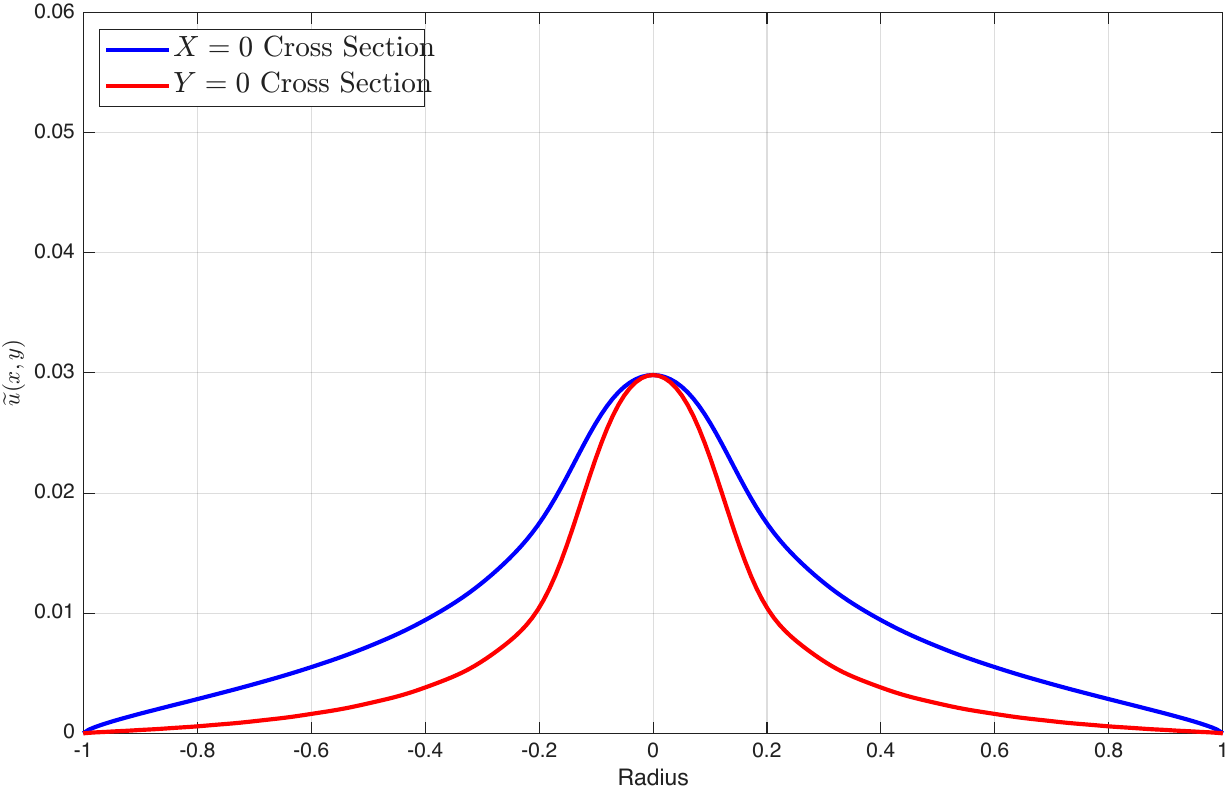}
  \caption{$x$ and $y$ axis cross sections with $\alpha = 1.40$}\label{ex1_cross_b}
\end{subfigure}

\medskip

\begin{subfigure}{0.46\textwidth}
  \centering
  \includegraphics[width=\linewidth]{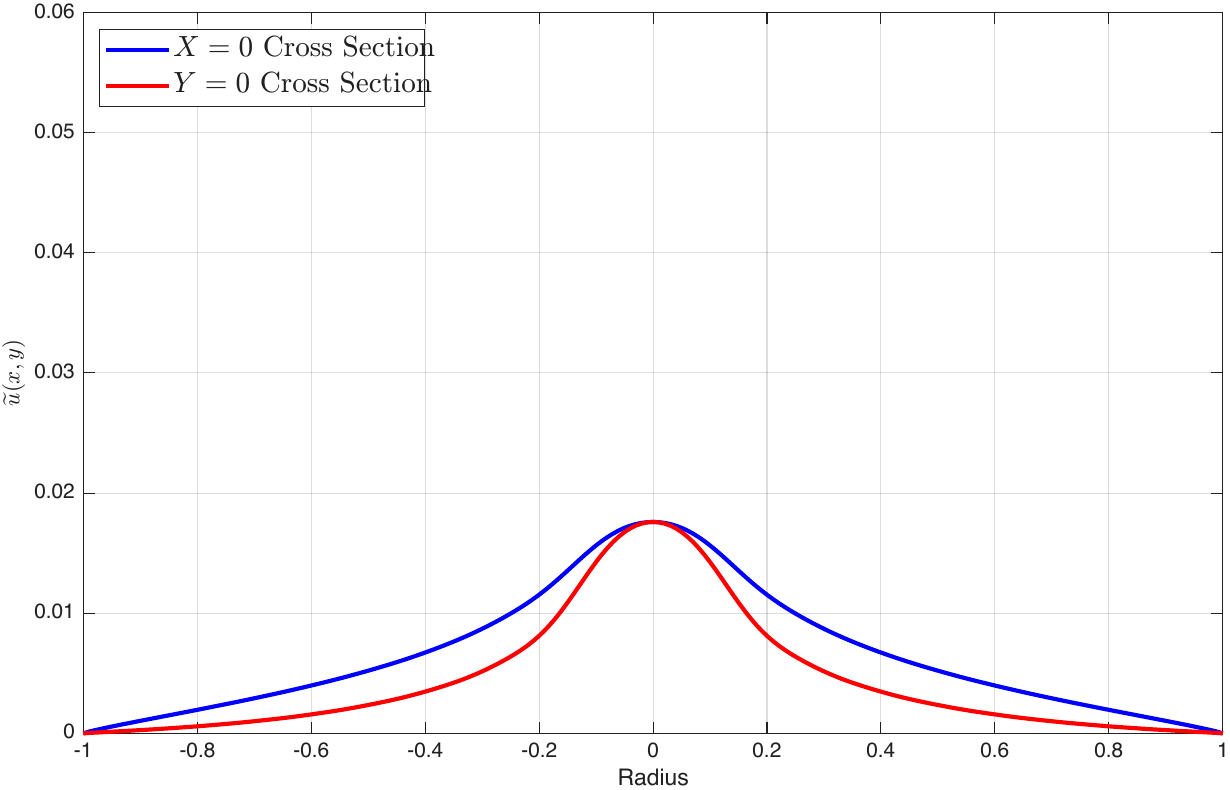}
  \caption{$x$ and $y$ axis cross sections with $\alpha = 1.70$}\label{ex1_cross_c}
\end{subfigure}
\hspace{0.05\textwidth}
\begin{subfigure}{0.46\textwidth}
  \centering
  \includegraphics[width=\linewidth]{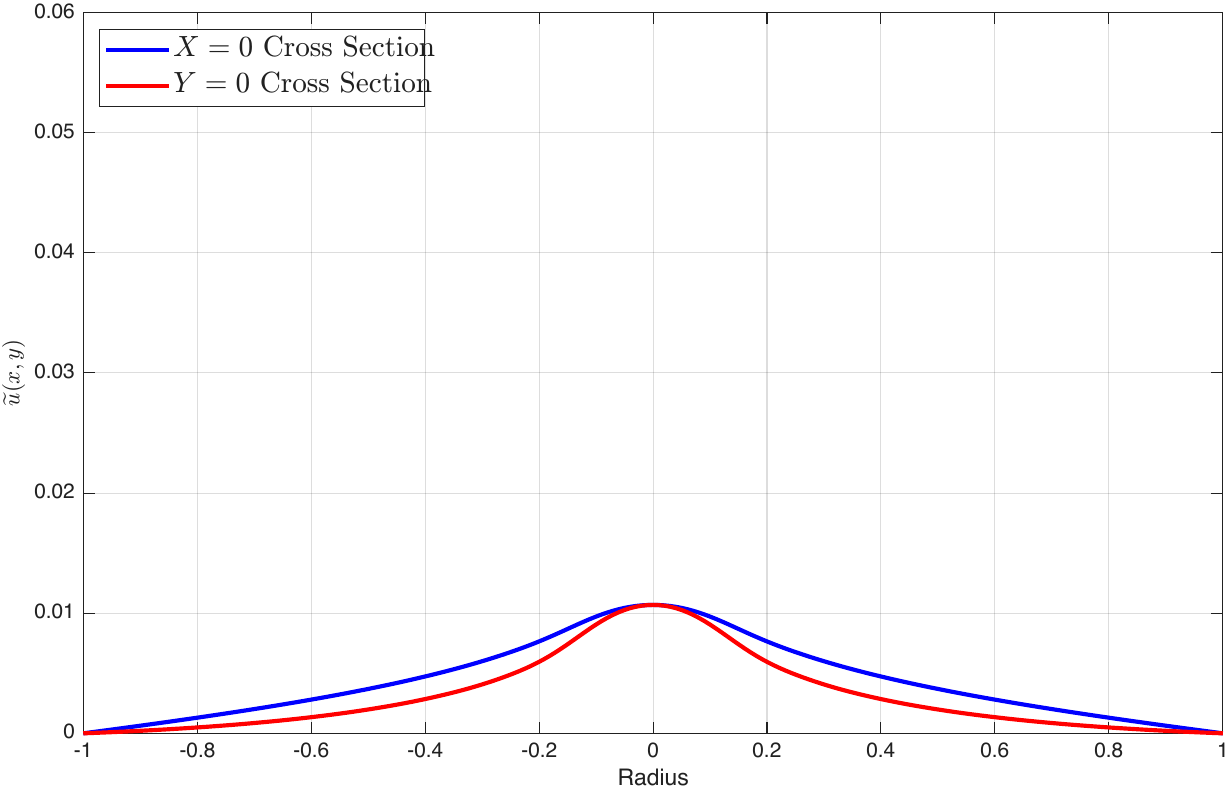}
  \caption{$x$ and $y$ axis cross sections with $\alpha = 2.00$}\label{ex1_cross_d}
\end{subfigure}

\caption{Cross-sectional views of solutions with radial bubble source function for various $\alpha\in [1,2]$.}
\label{fig:cross sections}
\end{figure}
%

\begin{figure}[H]
\centering

\begin{subfigure}{0.46\textwidth}
  \centering
  \includegraphics[width=\linewidth]{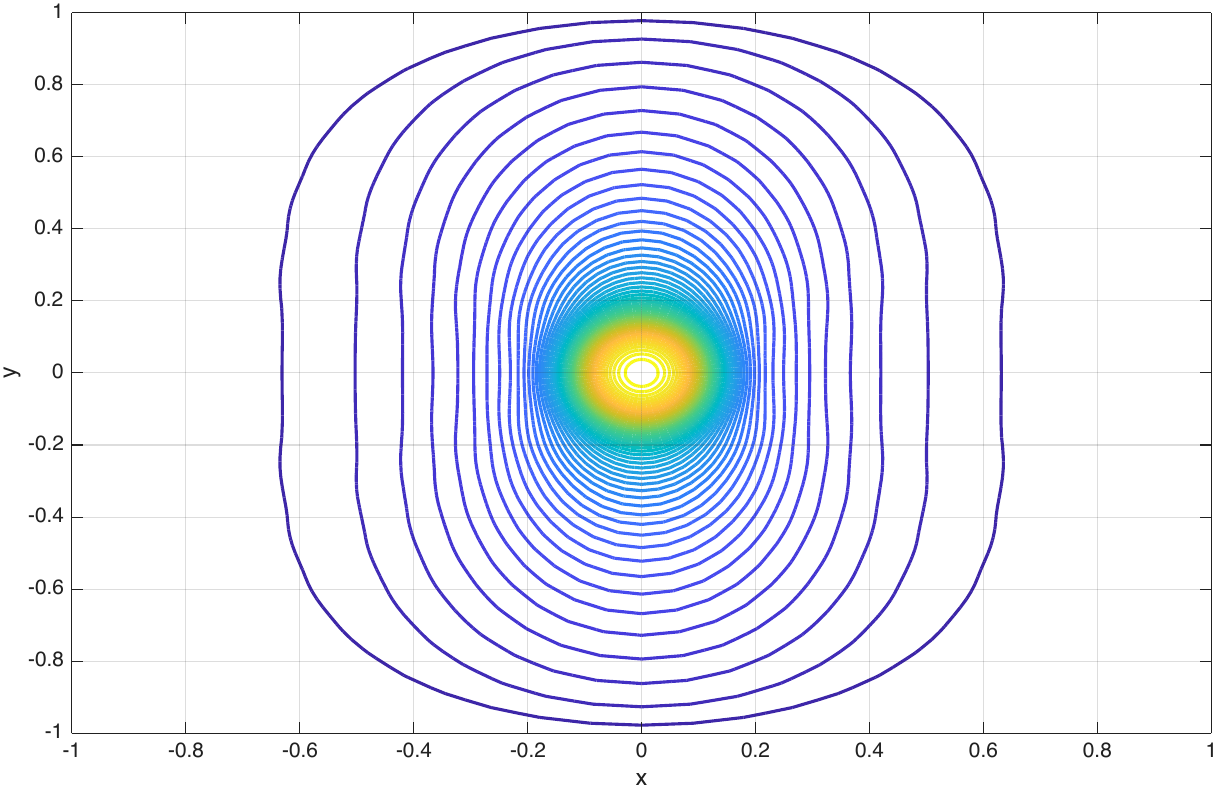}
  \caption{Contour plots with $\alpha = 1.10$}\label{ex1_contour_a}
\end{subfigure}
\hspace{0.05\textwidth}
\begin{subfigure}{0.46\textwidth}
  \centering
  \includegraphics[width=\linewidth]{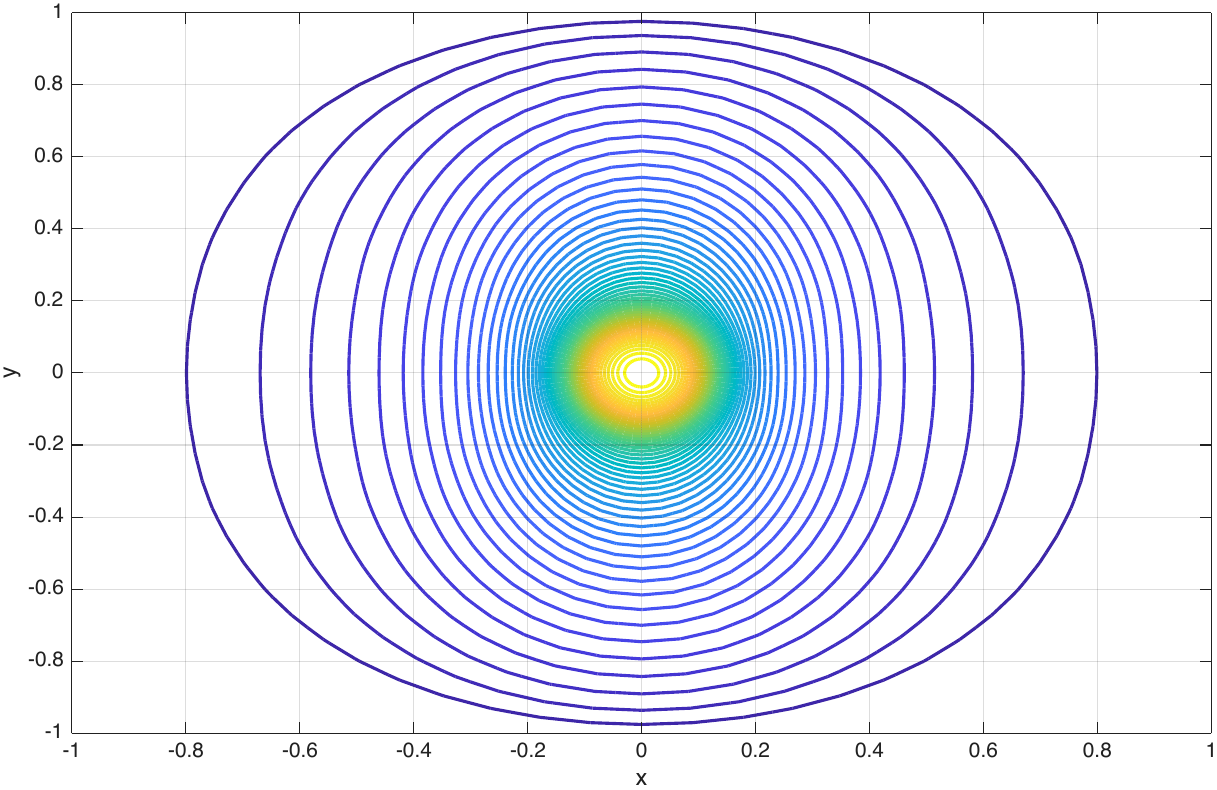}
  \caption{Contour plots with $\alpha = 1.40$}\label{ex1_contour_b}
\end{subfigure}

\medskip

\begin{subfigure}{0.46\textwidth}
  \centering
  \includegraphics[width=\linewidth]{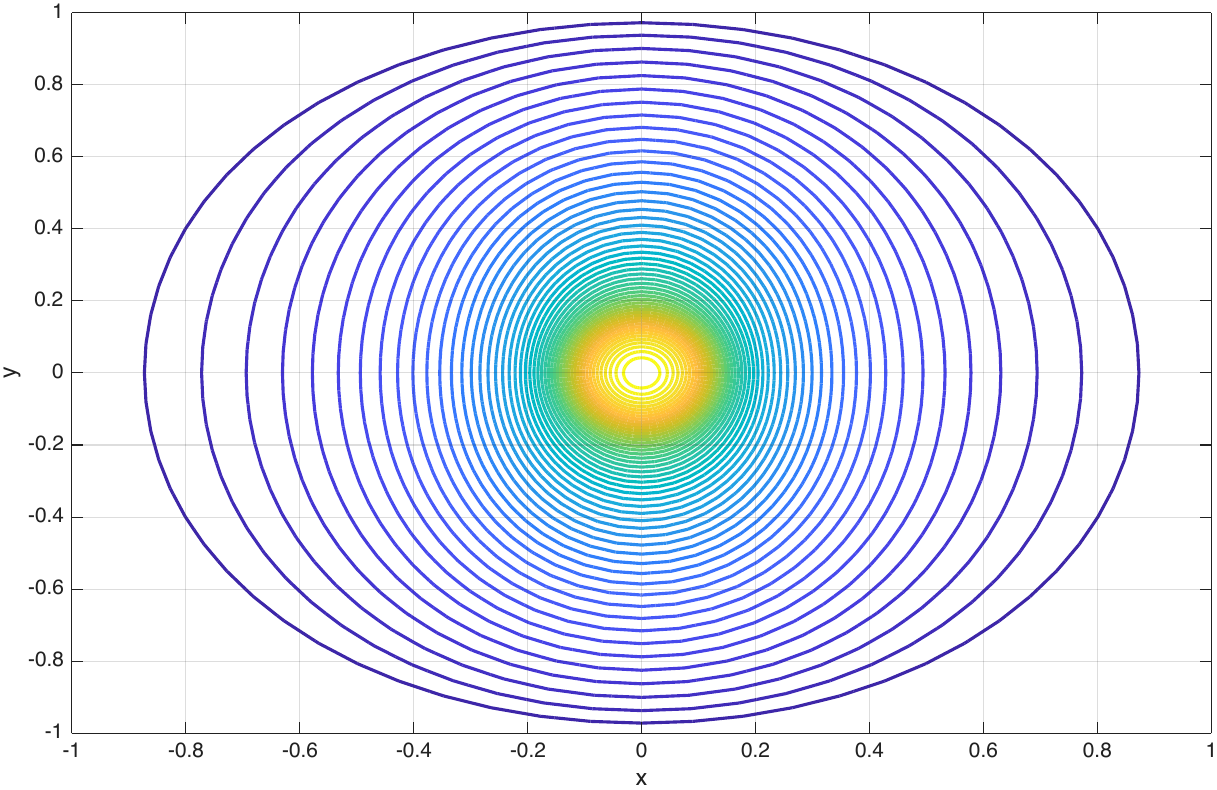}
  \caption{Contour plots  with $\alpha = 1.70$}\label{ex1_contour_c}
\end{subfigure}
\hspace{0.05\textwidth}
\begin{subfigure}{0.46\textwidth}
  \centering
  \includegraphics[width=\linewidth]{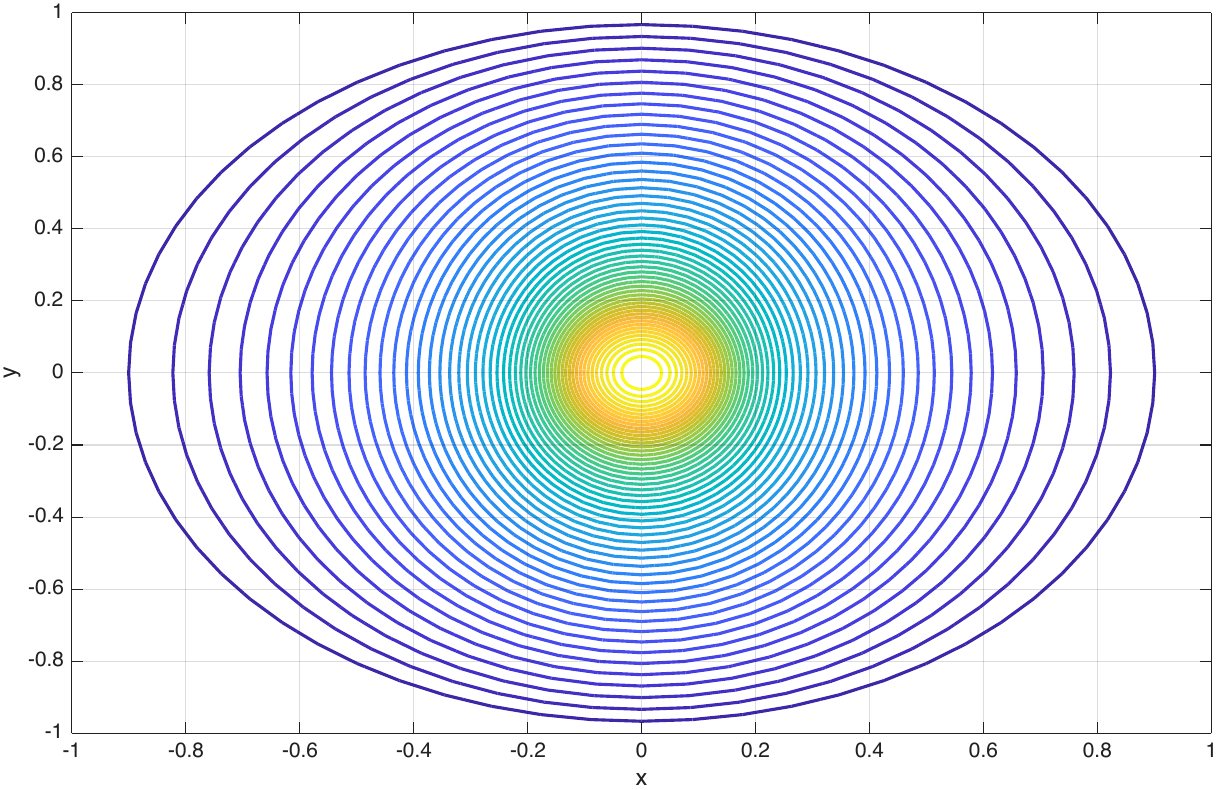}
  \caption{Contour plots with $\alpha = 2.00$}\label{ex1_contour_d}
\end{subfigure}

\caption{Contour plots of solutions with radial bubble source function for various $\alpha\in (1,2]$.}
\label{fig:contours}
\end{figure}

\end{example}

\begin{example}\label{ex2}
Next, we present an example of the steady-state problem \eqref{problem} and compute errors and convergence rates. Consider the source function $f(x) = |x_1|^3 + x_2 \in H^{7/2-\eps}_\alt(\Omega)$ (see Lemma \ref{Regexp} in the appendix) with $\alpha = 1.7$. From Theorem \ref{errorthm1}, we expect convergence rates of $4.35$ for this source. We use a reference solution with $R$ equal to 88 to compute the $L^2_\alt(\Omega)$ error. Table \ref{ex2table} presents the error and experimental convergence rates, $CR$, (computed using $CR_i = \frac{\log(E_i/E_{i-1})}{\log((R_{i-1}+2)/(R_i+2))}$), and Figures 
\ref{fig:ex7.2_contour}, and  \ref{fig:ex7.2_heatmap} display the reference solution.
\begin{table}[H]
\centering
\begin{tabular}{|c|c|c|}
\hline
Radial Degree & Error ($L^2_{\alpot}(\Omega)$) & $CR$ \\ 
\hline 
4 & 7.77e-05 & – \\ 
12 & 1.04e-06 & 5.09 \\ 
20 & 1.10e-07 & 4.97 \\ 
28 & 2.34e-08 & 4.99 \\ 
34 & 9.39e-09 & 5.01 \\ 
42 & 3.44e-09 & 5.00 \\ 
50 & 1.51e-09 & 4.94 \\ 
58 & 7.67e-10 & 4.71 \\ 
\hline
\multicolumn{2}{|c|}{Theoretical CR:} & 4.35 \\
\hline
\end{tabular}
\caption{Error and convergence rates for Example \ref{ex2} with $\alpha = 1.7$.}\label{ex2table}
\end{table}

\begin{figure}[H]
\centering

\begin{minipage}[t]{0.49\textwidth}
    \centering
    \includegraphics[width=.9\linewidth]{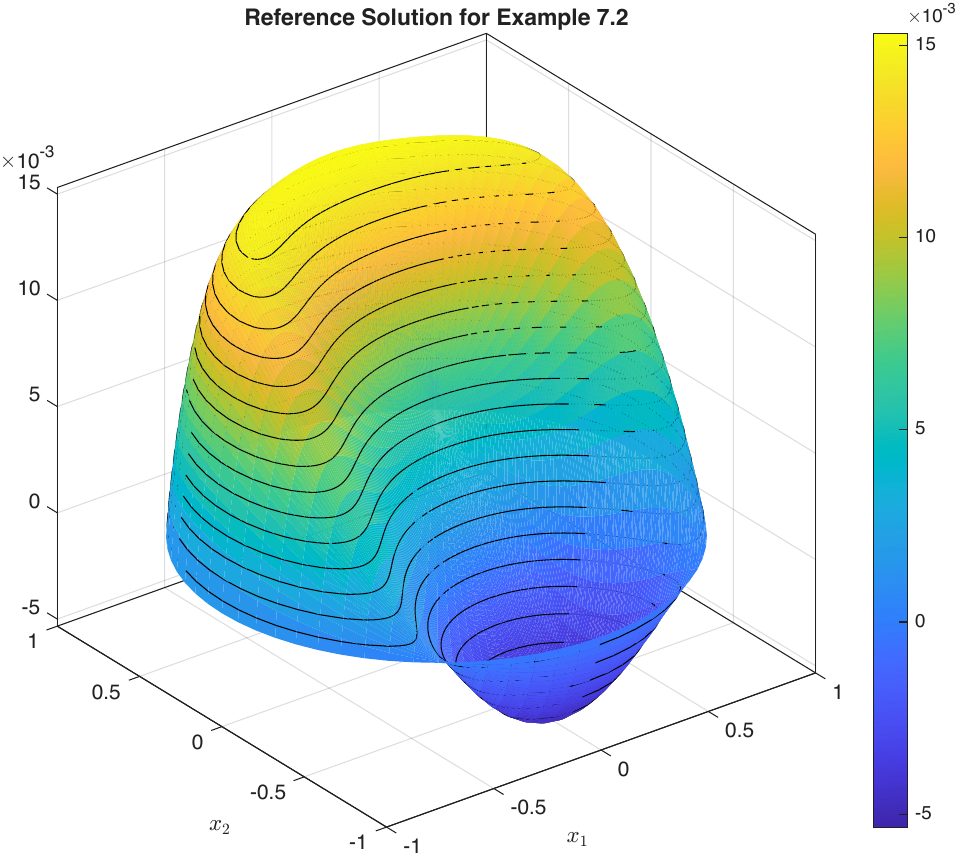}
    \caption{Plot with contour lines}
    \label{fig:ex7.2_contour}
\end{minipage}
\hfill
\begin{minipage}[t]{0.49\textwidth}
    \centering
    \includegraphics[width=.9\linewidth]{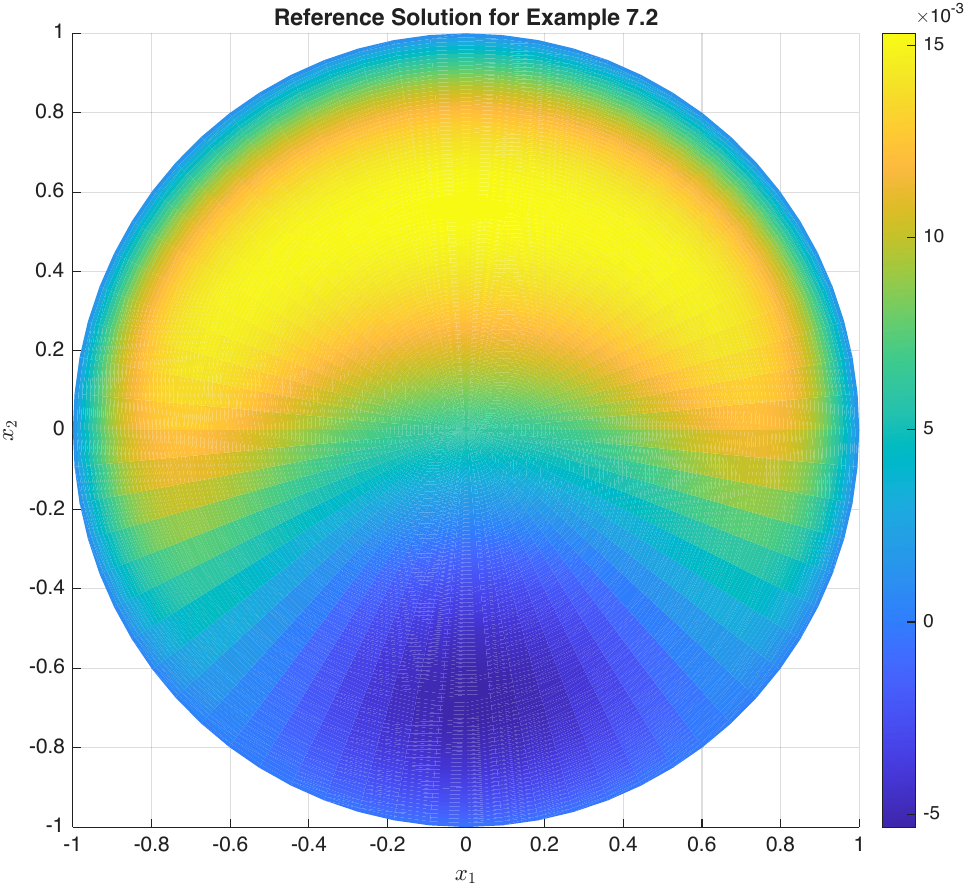}
    \caption{Heat map}
    \label{fig:ex7.2_heatmap}
\end{minipage}
\end{figure}

\end{example}

\begin{example}\label{ex3}
Next, we consider the time-dependent problem \eqref{pde0}. 
The solution $\wt{u} \, = \, \omega^{\alt} u$ is approximated by solving \eqref{approx} for 
$\mc{U}^{n}(x)$, the approximation of $u(x, t_{n})$, $n = 1, 2, \ldots, N$.
We compare the solutions for $\alpha = 1.1$ and $\alpha = 1.7$. We choose the same source function as in Example \ref{ex1} and show four distinct times to see how the solution evolves to the steady state solutions in Figures \ref{ex1a}, \ref{ex1c}. These computations were done with a maximum radial degree of $R = 40$. 

\begin{figure}[H]
\centering

\begin{subfigure}{0.48\textwidth}
\centering
\includegraphics[width=.75\linewidth]{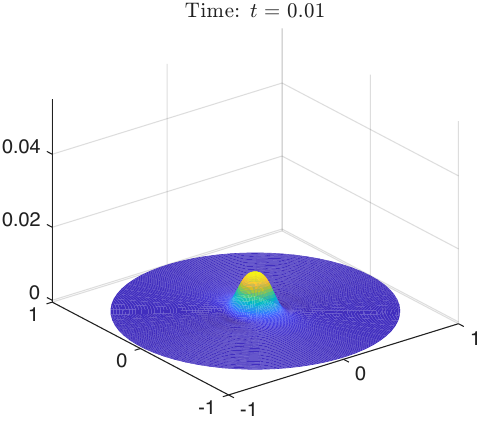}\par\vspace{0.5em}
\includegraphics[width=.75\linewidth]{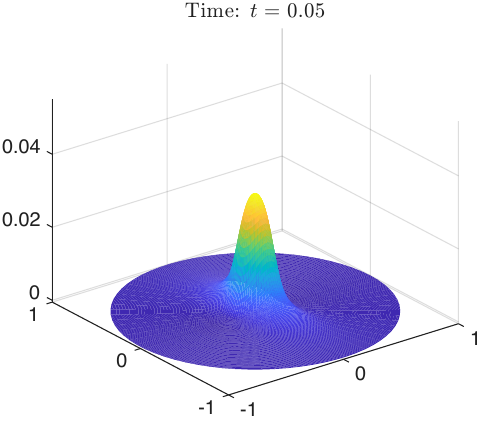}\par\vspace{0.5em}
\includegraphics[width=.75\linewidth]{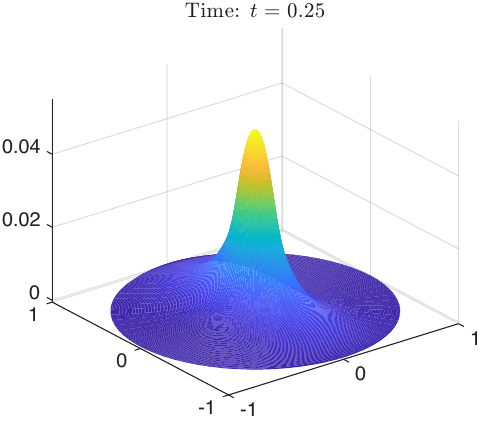}\par\vspace{0.5em}
\includegraphics[width=.75\linewidth]{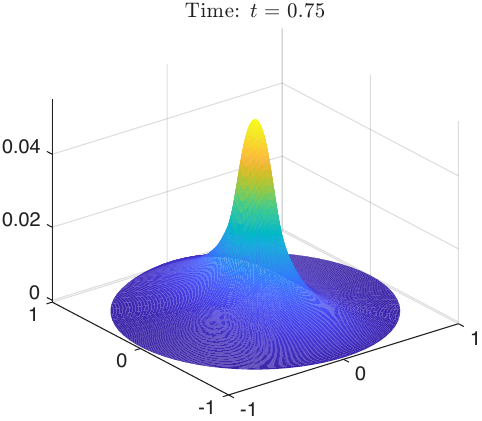}
\caption{$\alpha = 1.1$}
\end{subfigure}
\hfill
\begin{subfigure}{0.48\textwidth}
\centering
\includegraphics[width=.75\linewidth]{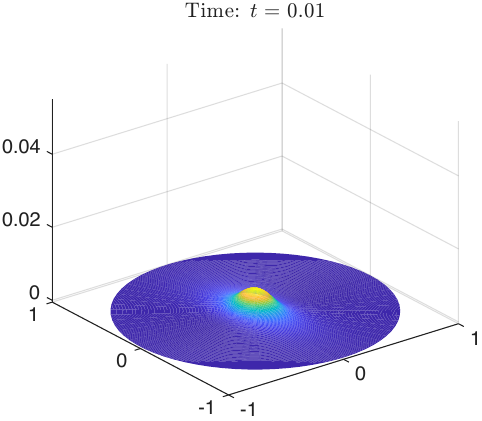}\par\vspace{0.5em}
\includegraphics[width=.75\linewidth]{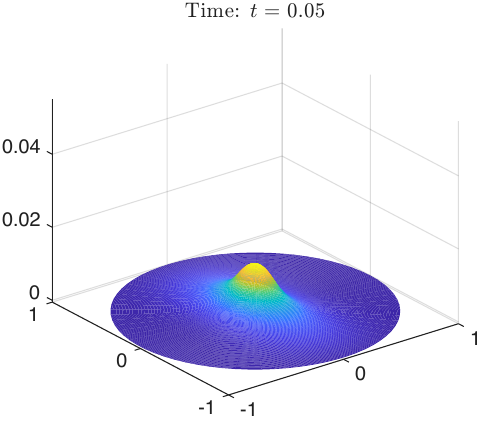}\par\vspace{0.5em}
\includegraphics[width=.75\linewidth]{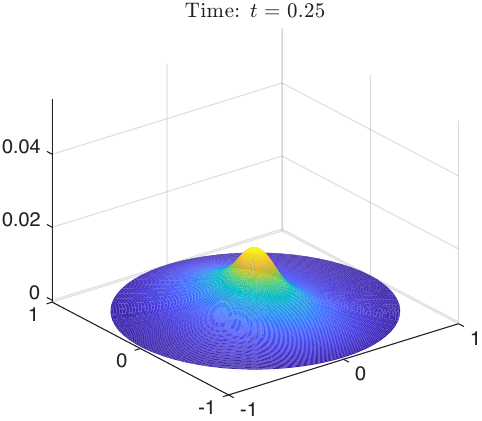}\par\vspace{0.5em}
\includegraphics[width=.75\linewidth]{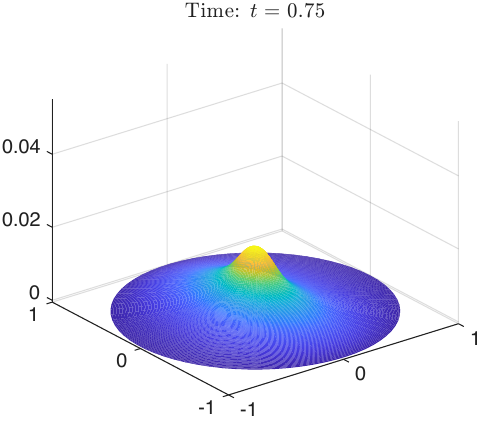}
\caption{$\alpha = 1.7$}
\end{subfigure}

\caption{Solutions at various times from Example \ref{ex3}.}
\label{fig:ex3}
\end{figure}

\end{example}

\begin{example}\label{ex4}
To compare experimental results to the theoretical results in Theorem \ref{ErrThm}, we use a true solution with known regularity, 
\[
u(x,t) = \begin{cases}
t(\frac{3}{2} - t)^3(|x_1|^3 + x_2), & x\in\Omega,\, 0\leq t\leq \frac{3}{2},\\
0, & \text{otherwise}
\end{cases}.
\]
The solution $u, u_t \in L^2(0,T; H^{7/2 - \eps}_\alt(\Omega))$ (see Lemma \ref{Regexp}). 
The value of $\alpha = 1.7$ is used for this example.
For $u \in L^2(0,T; H^r_\alt(\Omega)$ for $r\geq 0$, if we half the value of $\dt$ at each refinement, then using Corollary \ref{cor1}, we find that the optimal scaling of the radial degree is given by $R_{i+1} = \lceil R_i \cdot 2^{1/p}\rceil$, 
where $p = \frac{r + \alt}{4}$ (see \eqref{scal7}). In this case, we have $ p = 1.0875$. The error is computed in the $L^2_\alt(\Omega)$ norm at $t = t_{n}$, $n = 1, 2, \ldots, N$, 
and then the maximum value used to compute convergence rates, which is again done by $CR_i = \frac{\log(E_i/E_{i-1})}{\log((R_{i-1}+2)/(R_i+2))}$. 

\begin{table}[H]
\centering
\begin{tabular}{|c|c|c|c|}
\hline
Radial Degree & Time Step & Error (Max $L^2_\alt(\Omega)$) & CR  \\ 
\hline 
2 & 0.5 		& 6.51e-04 & – 		\\ 
4 & 0.25 		& 1.29e-04 & 1.996 	\\ 
8 & 0.125 		& 3.63e-05 & 1.241  	\\ 
16 & 0.0625 	& 1.04e-05 & 1.063  	\\ 
31 & 0.03125 	& 2.76e-06 & 1.093 	\\ 
59 & 0.015625 	& 7.30e-07 & 1.084 	\\ 
\hline
\multicolumn{3}{|c|}{Theoretical CR:} & 1.0875\\
\hline 
\end{tabular}
\caption{Errors and convergence rates for Example \ref{ex4}.}
\end{table}
We note that the experimental convergence rates are consistent with the theoretical (asymptotic) rate. 
\end{example}

\section{Concluding Remarks}
\label{conc}
In this paper we have investigated the fractional in space diffusion operator with a constant diffusivity matrix $K$, on the unit disk in
$\real^{2}$. Both the time dependent and steady state problems were analyzed. A spectral approximation scheme was used for the
spatial approximation and the backward Euler discretization used for the temporal discretization. An error analysis was presented 
for both the steady state and time dependent problems. Numerical experiments illustrating the influence of the diffusivity matrix $K$,
and comparing the experimental convergence rates with the predicted convergence rates were also given.

\clearpage
\appendix

\section{Function space and mapping properties}
\label{apdx_apx1}
 \begin{lemma}
  \label{lemSpace} For $s\in [0,1]$, $\alpha \geq 1$, $C_{\omega} \, = \, (1 \, + \, 2 \alpha^{2})^{s}$,
 if $f\in H^s_{\alt}(\Omega)$, then $(\omega^{\alt}f) \in H^s_{\alt}(\Omega)$
 and 
 \begin{equation}
 \| \omega^{\alt} f \|_{H^{s}_{\alt}} \, \leq \, (1 \, + \, 2 \alpha^{2})^{s} \| f \|_{H^{s}_{\alt}} \, = \, C_{\omega} \, \| f \|_{H^{s}_{\alt}}  .
 \label{defComega}
 \end{equation}
\end{lemma}
To establish this result we use the property that, for $s \ge 0$,  $H^s_{\alt}(\Omega)$ is a family of interpolation
spaces \cite{erv241}.
 
\textbf{Proof}: 
Define $A_1 = L^2_{\alt}(\Omega) = H^0_{\alt}(\Omega) = B_1$, and $T\colon A_1 \ra B_1$ by $Tf = \omega^{\alt}f$.  Note that
\[
\|\omega^{\alt}f\|^2_{L^2_{\alt}} = \int_\Omega \omega^{\alt}(\omega^{\alt}f)^2d\Omega \leq \int_\Omega\omega^{\alt}(f)^2d\Omega = \|f\|^2_{L^2_{\alt}}<\infty.
\]
Hence,
\begin{align*}
\|T\|_{A_1 \ra B_1} & = \sup_{ 0 \neq f\in L^2_{\alt}(\Omega)} \frac{\|Tf\|_{L^2_{\alt}(\Omega)}}{\|f\|_{L^2_{\alt}(\Omega)}} 
 \leq  \sup_{0 \neq f\in L^2_{\alt}(\Omega)} \frac{\|f\|_{L^2_{\alt}(\Omega)}}{\|f\|_{L^2_{\alt}(\Omega)}} = 1.
\end{align*}

Next, define $A_2 = H^1_{\alt}(\Omega) = B_2$. Consider the operator 
$T\colon A_2 \ra B_2$ given by $Tf = \omega^{\alt}f$.
Note that $f\in H^1_{\alt}(\Omega)$ implies 
 \[
 \|f\|^2_{H^1_{\alt}} = \|f\|^2_{L^2_{\alt}} + \|\grad f\|^2_{L^2_{\alt+1}} < \infty.
 \]
As $\alpha \geq 1$,
\begin{align*}
\|Tf\|^2_{H^1_{\alt}(\Omega)} & =  \|Tf\|^2_{L^2_{\alt}(\Omega)} + \|\grad (Tf)\|^2_{L^2_{\alt+1}(\Omega)}
 = \int_\Omega \omega^{\alt}(\omega^{\alt}f)^2d\Omega + \int_\Omega \omega^{\alt + 1}|\grad(\omega^{\alt}f)|^2d\Omega \\
& \leq \|f\|^2_{L^2_{\alt}(\Omega)}+ \int_\Omega \omega^{\alt + 1}\Big|\frac{\alpha}{2}\omega^{\alt-1}(-2)\begin{pmatrix} x \\ y \end{pmatrix}f + \omega^{\alt}\grad f\Big|^2d\Omega \\
& \leq \|f\|^2_{L^2_{\alt}(\Omega)}+2\alpha^2  \int_\Omega \omega^{3\alt - 1}\Big|\begin{pmatrix} x \\ y \end{pmatrix}f\Big|^2 + \omega^{3\alt+1}|\grad f|^2 d\Omega\\
& \leq \|f\|^2_{L^2_{\alt}(\Omega)}+2\alpha^2  \int_\Omega \omega^{\alt}|f|^2 + \omega^{\alt + 1}|\grad f|^2 d\Omega \\
& \leq (1+2\alpha^2)\|f\|^2_{H^1_{\alt}(\Omega)} < \infty. 
\end{align*}
Thus, 
\begin{align*}
\|T\|_{A_2 \ra B_2} & = \sup_{0 \neq f\in H^1_{\alt}(\Omega)} \frac{\|Tf\|_{H^1_{\alt}(\Omega)}}{\|f\|_{H^1_{\alt}(\Omega)}}
 \leq  \sup_{0 \neq f\in H^1_{\alt}(\Omega)} \frac{(1+2\alpha^2)\|f\|_{H^1_{\alt}(\Omega)}}{\|f\|_{H^1_{\alt}(\Omega)}}
  = 1+ 2\alpha^2. 
\end{align*}
Therefore, by the property of interpolation spaces (see \cite[Proposition 12.1.5]{bre081})
, $T\colon H^s_{\alt}(\Omega) \ra H^s_{\alt}(\Omega)$ is a bounded operator for any $s\in [0,1]$, with
$\| T \|_{H^s_{\alt}  \ra H^s_{\alt} }  \le 1^{1 - s} \, (1+ 2\alpha^2)^{s} = (1+ 2\alpha^2)^{s} $.   \\
\mbox{ } \hfill \qed

\centerline{}

\begin{corollary} \label{corXiBound}
For $s\in [0,1]$, $\alpha \geq 1$, $C_{\omega} \, = \, (1 \, + \, 2 \alpha^{2})^{s}$, 
if $f\in H^{-s}_{\alt}(\Omega)$, then
 \[
 \|\omega^{\alt}f\|_{H^{-s}_{\alt}} \ \leq \ (1 \, + \, 2 \alpha^{2})^{s} 
 \|f\|_{H^{-s}_{\alt}} \, = \, C_{\omega} \,  \|f\|_{H^{-s}_{\alt}}  \, .
 \] 
\end{corollary}
\textbf{Proof}: For $f \in H^{-s}_{\alt}(\Omega)$, $\omega^{\alt} f$ is defined for $v \in H^{s}_{\alt}(\Omega)$ as
\begin{equation}
\langle \omega^{\alt} f \, , \, v \rangle_{H^{-s}_{\alt} , H^{s}_{\alt}} \ 
:= \ \langle f \, , \,  \omega^{\alt}  v \rangle_{H^{-s}_{\alt} , H^{s}_{\alt}} \, .
\label{hguy1}
\end{equation}
That \eqref{hguy1} is well defined follows from Lemma \ref{lemSpace}. Additionally, using Lemma \ref{lemSpace}
\begin{align*}
\| \omega^{\alt} f \|_{H^{-s}_{\alt}} &= \ 
\sup_{0 \neq v \in H^{s}_{\alt}(\Omega)} \, 
\frac{ \langle \omega^{\alt} f \, , \, v \rangle_{H^{-s}_{\alt} , H^{s}_{\alt}}}{ \| v \|_{H^{s}_{\alt}}} 
\ = \  \sup_{0 \neq v \in H^{s}_{\alt}(\Omega)} \, 
\frac{ \langle f \, , \, \omega^{\alt}  v \rangle_{H^{-s}_{\alt} , H^{s}_{\alt}}}{ \| v \|_{H^{s}_{\alt}}}   \\
&\leq \ \sup_{0 \neq v \in H^{s}_{\alt}(\Omega)} \, 
\frac{ \| f \|_{H^{-s}_{\alt}}  \, \| \omega^{\alt} v \|_{H^{s}_{\alt}} }{ \| v \|_{H^{s}_{\alt}}}   
\ \leq \  \sup_{0 \neq v \in H^{s}_{\alt}(\Omega)} \, 
\frac{ \| f \|_{H^{-s}_{\alt}}  \, (1 \, + \, 2 \alpha^{2})^{s} \, \| v \|_{H^{s}_{\alt}} }{ \| v \|_{H^{s}_{\alt}}}   \\
&\leq \ (1 \, + \, 2 \alpha^{2})^{s} \, \| f \|_{H^{-s}_{\alt}}  \, .
\end{align*}
\mbox{ } \hfill \qed

\begin{lemma}  \label{lmafftf} 
For $\xi \in V' = H^{-\alt}_{\alt}(\Omega)$ the functional $\omega^{\alt} \xi$ defined by \eqref{defftf} is well defined with 
$\omega^{\alt} \xi \in V'$.
\end{lemma}
\textbf{Proof} Using Lemma \ref{lemSpace}, for  $v \in V = H^{\alt}_{\alt}(\Omega)$, $ \| \omega^{\alt} v \|_{V} \, \leq \, C_{\omega} \| v \|_{V} $.
Hence,
\begin{align*}
 \sup_{0 \neq v \in V} \frac{ | \langle \omega^{\alt} \xi \, , \, v \rangle_{V' , V} |}{ \| v \|_{V}} 
 &= \  \sup_{0 \neq v \in V} \frac{ | \langle \xi \, , \,  \omega^{\alt} v \rangle_{V' , V} |}{ \| v \|_{V}}   
 \ \leq \ 
 \sup_{0 \neq v \in V} \frac{ \| \xi \|_{V'}  \, \| \omega^{\alt} v \|_{V} }{ \| v \|_{V}}   \\
&\leq \ 
 \sup_{0 \neq v \in V} \frac{ \| \xi \|_{V'}  \, C_{\omega} \| v \|_{V} }{ \| v \|_{V}}   
 \ = \ C_{\omega} \, \| \xi \|_{V'} \, .
\end{align*}
Thus, $\omega^{\alt} \xi$ defines a bounded linear operator on $V$. \\
\mbox{ } \hfill \qed

The following characterization of the weak derivative is used in the proof of Lemma \ref{lmafder1}.
\begin{lemma} \label{Ernlm2} \cite[Proposition 64.33]{ern211}
Let $Z$ be a Banach space. Let $f, g \in L^{1}_{loc}(J ; Z)$. Then $f$ is weakly differentiable with $\partial_{t} f \, = \, g$
if and only if the map $J \ni t \rightarrow \langle z' , f(t) \rangle_{Z' , Z} \in \real$ is weakly differentiable for all
$z' \in Z'$, and $ \partial_{t} \langle z' , f \rangle_{Z' , Z} \, = \, \langle z' , g\rangle_{Z' , Z}$ a.e. in $J$.
\end{lemma}

\textbf{Proof of Lemma \ref{lmafder1}}: In the setting of Lemma \ref{Ernlm2}, $Z \, = \, H^{-\alt}_{\alt}(\Omega) = V'$,
$Z' \, = \, H^{\alt}_{\alt}(\Omega) = V$, and $g = \partial_{t} f \in L^{2}(0 , T ; V')$. For $z' \in V$, from 
Lemma \ref{lemSpace}, $\omega^{\alt} z' \in Z' = V$. Then for all $z' \in V$, from \eqref{defftf},
\begin{equation}
  \langle \omega^{\alt} z' \,  , \, g \rangle_{V , V'} \, = \,  \langle z' \, , \, \omega^{\alt} g \rangle_{V , V'} \ \mbox{ and } \ 
 \partial_{t} \langle \omega^{\alt} z' \, , \, f \rangle_{V , V'} \, = \, \partial_{t} \langle z' \, , \, \omega^{\alt} f \rangle_{V , V'} \, .
\label{ehfgr1}
\end{equation}

Thus, from the weak differentiability of $f$ (Lemma \ref{Ernlm2}), from \eqref{ehfgr1} it follows that
\[
         \partial_{t} \langle z' \, , \, \omega^{\alt} f \rangle_{V , V'} \, = \,  \langle z' \, , \, \omega^{\alt} g \rangle_{V , V'} , \, 
         \mbox{a.e. in } J \, ,
\]
which establishes that $\partial_{t} \, \omega^{\alt} f \, = \, \omega^{\alt} g \, = \,  \omega^{\alt} \partial_{t} f$. \\
\mbox{ } \hfill \qed

\textbf{Proof of Lemma \ref{pAo1}}: \\
The estimates for $A(t)$ in \eqref{propAo1} follow directly from the definition of $A(t)$ and \eqref{propB}. The 
invertibility of $A(t)$, i.e., the existence of $A(t)^{-1}$, follows from the Lax-Milgram theorem. 

To obtain the bounds for $A(t)^{-1}$, note that for $\xi \in V'$, $A(t)^{-1} \xi \, = \, g \in V$ implies
\[
    B(g , h ; t) \ = \ \langle \xi , h \rangle_{V' , V} \, , \ \mbox{ for all } h \in V \, .
\] 
Hence, for $h = g =  A(t)^{-1} \xi $, using  \eqref{propB},
\begin{align}
C_{coe} \| A(t)^{-1} \xi  \|^{2}_{V} 
&\leq \, B(A(t)^{-1} \xi \,  , A(t)^{-1} \xi \,  ; t) \ = \ \langle \xi \, , \, A(t)^{-1} \xi  \rangle_{V' , V} 
\, \leq \, \| \xi \|_{V'} \, \| A(t)^{-1} \xi  \|_{V} \, ,   \label{Aive1}  \\
 \ \Longrightarrow \ \  \| A(t)^{-1} \xi  \|_{V}  &\leq \    \frac{1}{C_{coe}} \| \xi \|_{V'} \, .  \nonumber
\end{align} 
\[
\mbox{Now } \ \| A^{-1}(t) \|_{\mc{L}(V' , V)} \ = \ \sup_{0 \neq \xi \in V'} \frac{ \| A(t)^{-1} \xi  \|_{V} }{ \| \xi \|_{V'}}
\ \leq \ \sup_{0 \neq \xi \in V'} \frac{ \frac{1}{C_{coe}} \| \xi \|_{V'} }{ \| \xi \|_{V'}} \, = \, \frac{1}{C_{coe}} \, .
\]
\[
\mbox{As } \ \| \xi \|_{V'} \, = \, \| A(t) A(t)^{-1} \xi \|_{V'} \, \leq \, \| A(t) \|_{\mc{L}(V , V')} \| A(t)^{-1} \xi \|_{V'} 
\, \leq \, C_{cts} \, \| A(t)^{-1} \xi \|_{V'} \, ,
\]
from \eqref{Aive1},
\[
  \langle \xi \, , \, A(t)^{-1} \xi  \rangle_{V' , V}  \ = \ B(A(t)^{-1} \xi \,  , A(t)^{-1} \xi \,  ; t) \ \geq \ 
  C_{coe} \, \| A(t)^{-1} \xi  \|_{V}^{2} \ \geq \ C_{coe} \, \frac{1}{C_{cts}^{2}} \, \| \xi \|_{V'}^{2} \, .
\]
\mbox{ } \hfill \qed

\section{Approximation properties}
\label{apdx_apx2}
Let $X_{R}$ be given by \eqref{defR} and
$ P_R\colon  L^2_{\alpot}(\Omega) \ra  X_R$ denote the orthogonal projection of $ L^2_{\alpot}(\Omega)$ onto $ X_R$. 
That is,
  \[
  \Big( z(x) -  P_R(z)(x) \, , \,   v(x)\Big)_{L^{2}_{\alt}} = 0, \ \ \forall  \ v(x) \in  X_R.
  \] 
Associated with $P_{R}$ is the following approximation property.
\begin{lemma} \label{lemXi}
Let $z \in H^r_{\alpot}(\Omega)$ for some $r\geq \alpot$.
Then, for any $s \in [-\alpot,\alpot]$,
  \[
  \|z \, - \, P_R(z) \|^2_{H^s_{\alpot}(\Omega)} \ \leq \  \Big(\frac{1}{R + 2}\Big)^{r - s}\| z \|^2_{H^r_{\alpot}(\Omega)}.
  \]
  \end{lemma}
\textbf{Proof}: For $z \in H^r_\alpot(\Omega),$ write 
\[
z(x) \ = \ \sum_{l,n,\mu} d_{l,n,\mu} \, \mathcal{P}_{l,n,\mu}^{(\alpot)}(x) \, , \ \
\mbox{where } \ d_{l,n,\mu}  \, = \, \frac{(z \,  , \, \mathcal{P}_{l,n,\mu}^{(\alpot)} )_{L^{2}_{\alt}}}{h^{2}_{l, n}} \, , 
\ \ \mbox{and } \ h^2_{l,n}  \coloneqq \|\P_{l,n,\mu}^{(\alt,l)}\|^2_{L^2_{\alt}(\Omega)} \, .
\] 
Without loss of generality, assume $R$ is odd.
Similar to the expression for $\| u \, - \, U_{R} \|^{2}_{L^{2}_{\alt}}$ in the proof of  Theorem \ref{errorthm1}, 
 we have \\
\begin{align*}
\|z -  P_R(z)\|^2_{H^s_{\alpot}}
&= \  \sum_{l = R+1}^\infty \sum_{n = 0}^{\lfloor \frac{l}{2} \rfloor}
  (n + 1)^s \, (n + (l - 2n) +1)^s \, (d _{l - 2n , \, n, 1})^2 \, h^2_{l - 2n , \, n}   \nonumber  \\
& \quad  + \
\sum_{l = R+2}^\infty \sum_{n = 0}^{\lfloor \frac{l - 1}{2} \rfloor}
  (n + 1)^s \, (n + (l - 2n) +1)^s \, (d _{l - 2n , \, n, -1})^2 \, h^2_{l - 2n , \, n} 
 \nonumber \\
&= \ 
\sum_{l = R+1}^\infty \sum_{n = 0}^{\lfloor \frac{l}{2} \rfloor}
  (n + 1)^{s - r} \, (l - n +1)^{s - r} \,  (n + 1)^r \, (l - n +1)^r \, (d _{l - 2n , \, n, 1})^2 \, h^2_{l - 2n , \, n}   \nonumber \\
& \quad  + \
\sum_{l = R+2}^\infty \sum_{n = 0}^{\lfloor \frac{l - 1}{2} \rfloor}
  (n + 1)^{s - r} \, (l - n +1)^{s - r} \,  (n + 1)^r \, (l - n +1)^r \, (d _{l - 2n , \, n, -1})^2 \, h^2_{l - 2n , \, n} 
\end{align*}
Now, for $l \geq (R + 1)$ and $0 \leq n \leq \lfloor \frac{l}{2} \rfloor$, we have that
$(n + 1)^{s - r} \, (l - n + 1)^{s - r} \leq (R + 2)^{s - r}$. Hence,
\begin{align*}
\|z -  P_R(z)\|^2_{H^s_{\alpot}}
&\leq \  \big( \frac{1}{R + 2} \big)^{r - s}
\sum_{l = 0}^\infty \sum_{n = 0}^{\infty}
  (n + 1)^r \, (n + l +1)^r \, (d _{l, \, n, 1})^2 \, h^2_{l - 2n , \, n}   \nonumber \\
& \quad  + \  \big( \frac{1}{R + 2} \big)^{r - s}
\sum_{l = 1}^\infty \sum_{n = 0}^{\infty}
 (n + 1)^r \, (n + l +1)^r \, (d _{l, \, n, -1})^2 \, h^2_{l - 2n , \, n}   \nonumber \\
& = \ \| z \|^2_{H^r_{\alpot}} \, .
\end{align*}
\mbox{ } \hfill  \qed

 \begin{lemma} 
 \label{lemInt}  
For $u(x,t) \in C^{2}(t_{n-1} , t_{n} ; H^{\alt}_{\alt}(\Omega))$, and $\Delta t \, = \, t_{n} - t_{n-1}$,
we have 
 \[
 \Big\| \omega^{\alt} \Big(u_t(\cdot , t_n) - \frac{u(\cdot , t_n) - u(\cdot , t_{n-1})} { \Delta t} \Big)\Big\|_{H^{-\alt}_{\alt}}^2 
 \leq \frac{\Delta t}{3} \int_{t_{n-1}}^{t_n} \| \omega^{\alt} u_{tt}(\cdot , \tau)\|_{H^{-\alt}_{\alt}}^2 d\tau.
 \] 
 \end{lemma}
\textbf{Proof}: 
Let $D^{(n)}(x) = u_t (x,t_n) - \frac{u(x,t_n) - u(x,t_{n-1})}{\Delta t}.$ By the Fundamental Theorem of Calculus, we have 
\[
 u(x,t_n) -  u(x,t_{n-1}) = \int _{t_{n-1}}^{t_n}  u_t(x,s)ds.
 \] 
 Thus, 
\begin{align*}
D^{(n)}(x) & =  u_t(x,t_n) - \frac{1}{\Delta t}\int_{t_{n-1}}^{t_n} u_t(x,s) \, ds
 = \frac{1}{\Delta t}\int_{t_{n-1}}^{t_n} ( u_t(x,t_n) -  u_t(x,s)) ds\\
& = \frac{1}{\Delta t}\int_{t_{n-1}}^{t_n} \int_s^{t_n}  u_{tt}(x,\tau) d\tau \, ds
 = \frac{1}{\Delta t} \int_{t_{n-1}}^{t_n} \int_{t_{n-1}}^\tau u_{tt} (x,\tau) ds \, d\tau  \\  
 & = \frac{1}{\Delta t}\int_{t_{n-1}}^{t_n} (\tau - t_{n-1})  u_{tt}(x,\tau) \, d\tau.
\end{align*}
As $u \in C^{2}(t_{n-1} , t_{n} ; H^{\alt}_{\alt}(\Omega))$ then, for $v \in H^{\alt}_{\alt}(\Omega)$,
Then, 
\begin{align*}
\langle \omega^{\alt} D^{(n)} \, , \, v \rangle_{H^{-\alt}_{\alt} , H^{\alt}_{\alt} }
&=  \ \big( \omega^{\alt} D^{(n)}(x)  \, , \, v(x) \big)_{L^{2}_{\alt}}  \\
&= \ \int_{\Omega} \omega^{\alt} \, 
\Big( \omega^{\alt} \frac{1}{\Delta t}\int_{t_{n-1}}^{t_n} (\tau - t_{n-1})  u_{tt}(x,\tau) \, d\tau \Big) v(x) \, dx  \\
&= \ \frac{1}{\Delta t} \int_{t_{n-1}}^{t_n} (\tau - t_{n-1}) \, 
 \int_{\Omega} \omega^{\alt} \, \omega^{\alt} \,  u_{tt}(x,\tau) \, v(x) \, dx \, d\tau \\
&\leq  \ \frac{1}{\Delta t} \int_{t_{n-1}}^{t_n} (\tau - t_{n-1}) \, 
 \| \omega^{\alt} \,  u_{tt}(\cdot , \tau) \|_{H^{-\alt}_{\alt}} \, \| v  \|_{H^{\alt}_{\alt}} \, d\tau  \\
&\leq \  \frac{1}{\Delta t} 
\Big( \int_{t_{n-1}}^{t_n} (\tau - t_{n-1})^{2} \, d\tau \Big)^{1/2} \, 
\Big( \int_{t_{n-1}}^{t_n}  \| \omega^{\alt} \,  u_{tt}(\cdot , \tau) \|_{H^{-\alt}_{\alt}}^{2} \, d\tau \Big)^{1/2} \, 
\| v  \|_{H^{\alt}_{\alt}}  \\
&= \  \Big( \frac{\Delta t}{3} \Big)^{1/2} \, 
\Big( \int_{t_{n-1}}^{t_n}  \| \omega^{\alt} \,  u_{tt}(\cdot , \tau) \|_{H^{-\alt}_{\alt}}^{2} \, d\tau \Big)^{1/2} \, 
\| v  \|_{H^{\alt}_{\alt}} \, .
\end{align*}
Hence,
\[
\| \omega^{\alt} D^{(n)} \|_{H^{-\alt}_{\alt}} \ = \ 
\sup_{0 \neq v \in H^{\alt}_{\alt}(\Omega)} 
\frac{\langle \omega^{\alt} D^{(n)} \, , \, v \rangle_{H^{-\alt}_{\alt} , H^{\alt}_{\alt} } }{ \| v \|_{H^{\alt}_{\alt} }}
\ \leq \
\Big( \frac{\Delta t}{3} \Big)^{1/2} \, 
\Big( \int_{t_{n-1}}^{t_n}  \| \omega^{\alt} \,  u_{tt}(\cdot , \tau) \|_{H^{-\alt}_{\alt}}^{2} \, d\tau \Big)^{1/2} \, .
\]
\mbox{ } \hfill \qed


\begin{lemma} \label{Regexp}
Let $u(x) = |x_1|^3+x_2$. Then, for  any $\epsilon > 0$, $u \in H^{7/2 - \epsilon}_{\alpot} (\Omega)$.
\end{lemma}

\textbf{Proof}:  Writing $u(x)$ in the basis \eqref{basisL2w}, 
$u(x) \, = \, \sum_{l,n,\mu} u_{l,n,\mu} \, \mathcal{V}_{l,\mu}(x)P_n^{(\alt, l)}(2r^2-1)$ where
$u_{l,n,\mu}$ and $h^2_{l,n}$ are given by
\begin{align}
u_{l,n,\mu} &= \frac{\Big(u(x)\, ,\, \mathcal{V}_{l,\mu}(x)P_n^{(\alt, l)}(2r^2-1)\Big)_{L^2_\alt}}{h^2_{l,n}} , \ 
\ \mbox{and }  \nonumber  \\
h^2_{l,n} &= C_{l,\mu} 2^{-1}\frac{1}{2n+\alt + l + 1} \frac{\Gamma(n+\alt + 1)}{\Gamma(n+1)}\frac{\Gamma(n+l+1)}{\Gamma(n+\alt + l + 1)} ,   \label{h2def1}
\end{align}
with $C_{l,\mu}$ defined in Lemma \ref{NormShiftAdd}.

First, consider the second piece of $u$,  $f(x) := x_2$. As $f \in C^\infty (\Omega)$, $f \in H^s_\alt (\Omega)$ for all $s>0$. 
Consider next $g(x) \, := \, |x_1|^3 = r^3|\cos(\phi)|^3$. Of interest is determining 
\begin{align}
g_{l,n,\mu} &= \frac{\Big(g(x)\, ,\, \mathcal{V}_{l,\mu}(x)P_n^{(\alt, l)}(2r^2-1)\Big)_{L^2_\alt}}{h^2_{l,n}}  \nonumber \\
&= \ 
\int_{\phi = 0}^{2\pi} \int_{r = 0}^1 (1-r^2)^\alt r^3|\cos(\phi)|^3\mathcal{V}_{l,\mu}(x)P_n^{(\alt,l)}(2r^2-1) r \, dr \, d\phi 
\, / h^2_{l,n}   \nonumber  \\
\label{integrals} & = \underbrace{\int_{\phi = 0}^{2\pi}\text{trg}(l\phi)|\cos(\phi)|^3 d\phi}_{\coloneqq\mathcal{J}} \ \underbrace{\int_{r=0}^1r^{3+l}P_n^{(\alt,l)}(2r^2-1)(1-r^2)^\alt r dr}_{\mathcal{I}_{l,n,3}}   /  h^2_{l,n} \, ,  \\
& \quad  \quad \mbox{where} \ \text{trg}(l\phi) \, = \, \left\{ \begin{array}{rl}
                                                                            \cos(l\phi) & \mbox{ if } \ \mu = 1 \, , \\
                                                                            \sin(l\phi) & \mbox{ if } \ \mu = -1 \, 
                                                                            \end{array} \right. \, .   \nonumber 
\end{align}
                                                                            
When $\mu = -1$, trg$(l\phi) = \sin(l\phi)$, and so 
\begin{align*}
\mathcal{J} &= \ \int_{\phi = 0}^{2\pi} \sin(l\phi)|\cos(\phi)|^3 d\phi 
\ = \ \int_{\phi= -\pi}^{\pi}\sin(l(\phi+\pi))|\cos(\phi+\pi)|^3 d\phi   \\
& = \int_{\phi = -\pi}^\pi \sin(l\phi)\cos(l\pi)|\cos(\phi)\cos(\pi)|^3 d\phi,\ \ (\mbox{as } \sin(l\pi) = 0) \\
& = \cos(l\pi)\int_{\phi = -\pi}^\pi \sin(l\phi)|cos(\phi)|^3 d\phi 
\ =  \ 0,
\end{align*}
as the integrand in the last integral is the product of an odd and an even function that is
integrated over a symmetric interval about $0$. 

When $\mu = 1$ and trg$(l\phi) = \cos(l\phi)$, 
\begin{align}
\notag\mathcal{J} &= \ \int_{\phi = 0}^{2\pi}\cos(l\phi)|\cos(\phi)|^3 d\phi
\ = \ \int_{\phi = -\pi}^{\pi}\cos(l(\phi+\pi))|\cos(\phi + \pi)|^3 d\phi   \\
\notag& = \int_{\phi = -\pi}^\pi \cos(l\phi)\cos(l\pi)|\cos(\phi)\cos(\pi)|^3 d\phi, \ \ (\mbox{as } \sin(l\pi) = 0)\\
\notag& = \cos(l\pi)\int_{\phi = -\pi}^\pi \cos(l\phi)|\cos(\phi)|^3 d\phi.
\ = \ 2\cos(l\pi)\int_{\phi = 0}^\pi \cos(l\phi)|\cos(\phi)|^3 \, d\phi \\
\label{intestep}&= 2\cos(l\phi)\Bigg(\int_{\phi = 0}^\frac{\pi}{2} \cos(l\phi)\cos^3(\phi)d\phi - \int_{\phi = \frac{\pi}{2}}^\pi \cos(l\phi)\cos^3(\phi) \, d\phi\Bigg).
\end{align}

Manipulating the last integral in (\ref{intestep}),
\begin{align}
\notag & -\int_{\phi = \frac{\pi}{2}}^\pi\cos(l\phi)\cos^3(\phi) \, d\phi \ 
= \ -\int_{\phi = -\frac{\pi}{2}}^0 \cos(l(\phi + \pi))\cos^3(\phi+\pi) \, d\phi  \\
\notag& = -\int_{\phi = -\frac{\pi}{2}}^0 \cos(l\phi)\cos(l\pi)\cos^3(\phi)\cos^3(\pi) \, d\phi 
\ = \ \cos(l\pi)\int_{\phi = -\frac{\pi}{2}}^0\cos(l\phi)\cos^3(\phi) \, d\phi\\
\label{substint}& = \cos(l\pi)\int_{\phi = 0}^\frac{\pi}{2}\cos(l\phi)\cos^3(\phi) \, d\phi, 
\end{align}
again as $\cos(l\phi)$ and $\cos^3(\phi)$ are even functions. Substituting (\ref{substint}) into (\ref{intestep}) gives 
for $\mu = 1$
\begin{equation} \label{exp3} \mathcal{J} = 2\cos(l\pi)(1+\cos(l\pi))\int_{\phi= 0}^\frac{\pi}{2}\cos(l\phi)\cos^3(\phi)d\phi.\end{equation}
In \cite{hao211} pg. 2127, the authors showed that 
\begin{equation}\label{exp4} \int\cos^k(\theta)\cos(m\theta)d\theta = \frac{1}{k+m}\Big(\cos^k(\theta)\sin(m\theta) + k \int\cos^{k-1}(\theta)\cos((m-1)\theta)d\theta\Big) \, .
\end{equation}
Using \eqref{exp4}, we obtain
\begin{equation}
\mathcal{J} = \left\{ \begin{array}{rl}
    \frac{24}{(l+3)(l+1)(l-1)(l-3)}\sin((l-3)\frac{\pi}{2}), & \mbox{if $l$ is even},\\
       0  , & \mbox{if $l$ is odd} 
       \end{array} \right. \, .
\label{v4J}
\end{equation}

For the evaluation of $\mathcal{I}_{l,n,3}$, using Lemma C.1 on pg. 2125 of \cite{hao211}, there are three cases to consider.
\begin{description}
	\item[Case 1:] If $\frac{3-l}{2}\geq n$, then $2n + l \leq 3$, or rather $\begin{cases} n = 0, & l = 0,1,2,3\\n = 1, & l = 0,1 \end{cases}.$ From here, 
\[\mathcal{I}_{l,n,3} = (-1)^n \frac{\Gamma(n+\alt + 1)}{2 \, \Gamma(n+1)}\frac{\Gamma(\frac{l+3}{2} + 1)}{\Gamma(\frac{3-l}{2} - n +1)}\frac{\Gamma(\frac{3-l}{2}+1)}{\Gamma(n+\alt + \frac{l+3}{2} + 2)}.\]
	\item[Case 2:] If $\frac{3-l}{2} < n$ and $\frac{3-l}{2}$ is not a non-negative integer, meaning if $3 < 2n + 1$ and $l \neq 1, 3$, then 
\[\mathcal{I}_{l,n,3} = (-1)^n\frac{\Gamma(n+\alt + 1)}{2 \, \Gamma(n+1)}\frac{\Gamma(\frac{l+3}{2}+1)}{\Gamma(\frac{l-3}{2})}\frac{\Gamma(n+\frac{l-3}{2})}{\Gamma(n+\alt + \frac{l+3}{2} + 2)}.\]
	\item[Case 3:] If $\frac{3-l}{2} < n$ and $\frac{3-l}{2}$ is a non-negative integer, then $3<2n + l$ and $l = 1,3$. In this case, 
\[\mathcal{I}_{l,n,3} = 0.\]
\end{description}

Combining the above we have, $\mathcal{J} \, \mathcal{I}_{l,n,3} = 0$ for $l$ odd or $\mu = -1$, 
and for $l$ even and $\mu = 1$, 
\begin{equation*} 
|\mathcal{J} \, \mathcal{I}_{l,n,3}| = \frac{24}{(l+3) \, (l+1)\, | l-1| \,  | l-3|}\frac{\Gamma(n+\alt+1)}{2 \, \Gamma(n+1)}\frac{\Gamma(\frac{l+3}{2}+1)}{| \Gamma(\frac{l-3}{2})|}\frac{| \Gamma(n+\frac{l-3}{2}) |}{\Gamma(n+\alt + \frac{l+3}{2} + 2)}.
\end{equation*}

Substituting into \eqref{integrals} we obtain that, for $l$ even and $\mu = 1$,
\begin{equation} 
	|g_{l,n,1} | = 
	\frac{24}{(l+3) \, (l+1)\, | l-1| \,  | l-3|}\frac{\Gamma(n+\alt+1)}{2\Gamma(n+1)}\frac{\Gamma(\frac{l+3}{2}+1)}{| \Gamma(\frac{l-3}{2})|}\frac{| \Gamma(n+\frac{l-3}{2}) |}{\Gamma(n+\alt + \frac{l+3}{2} + 2)} \, / h^{2}_{l, n} \, 
	. \label{ggh1}
\end{equation}

Of interest is determining $s\geq 0$ such that 
\[
\| g \|^2_{H^s_\alt}  \ = \ \||x_1|^3\|^2_{H^s_\alt} \ = \
\sum_{l,n,\mu} (n+1)^s (n+l+1)^s |g_{l,n,\mu}|^2 h^2_{l,n} < \infty.
\]

To that end, as the question is the values of $s$ that guarantee the convergence of the above sum, we are only interested in the behavior of the coefficients for large values of $l$ and $n$. 

Note that Stirling's formula (\ref{stirling}) gives that for $l$ or $n$ is large, 
\[
\frac{\Gamma(n+\alt+1)}{\Gamma(n+1)} \ \approx \ (n+1)^{\alt} \, , \ \ 
 \frac{\Gamma(\frac{l + 3}{2}+1)}{\Gamma(\frac{l-3}{2})} = \frac{\Gamma(\frac{l-3}{2} + 4)}{\Gamma(\frac{l-3}{2})} \approx \Big(\frac{l-3}{2}\Big)^4,
\]
\[
\frac{\Gamma(n+\frac{l-3}{2})}{\Gamma(n+\alt + \frac{l+3}{2} + 2)} = \frac{\Gamma(n+\frac{l-3}{2})}{\Gamma(n+\frac{l-3}{2}+\alt + 5)}\approx (n+\frac{l-3}{2})^{-(5+\alt)} = (2n+l-3)^{-(5+\alt)} 2^{5+\alt} \, .
\] 
and
\[
\frac{\Gamma(n+\alt + l + 1)}{\Gamma(n+l+1)} \approx (n+l+1)^{\alt}.
\]
Substituting each of these into the expression for $| g_{l,n,1} |$ in \eqref{ggh1}, 
and $h^{2}_{l, n}$ in \eqref{h2def1}, 
for $l$ even and 
$l$ or $n$ large,
\begin{align*}
| g_{l,n,1} |^{2} \, h^{2}_{l, n} &\approx (n+1)^{\alt} \, (n+1)^{\alt} \, (2n+l-3)^{-9 - \alpha}  \, .
\end{align*}

Of interest is determining $s\geq 0$ such that 
\[
\| g \|^2_{H^s_\alt}  \ = \ \||x_1|^3\|^2_{H^s_\alt} \ = \
\sum_{l, n} (n+1)^s (n+l+1)^s |g_{l,n,1}|^2 h^2_{l,n} < \infty.
\]

Note that, 
\begin{align}
	\notag(2n+l-3) & \approx (n+1) + (n + l + 1), \text{ and}\\
	\notag2(n+1)(n+l+1) &\leq (n+1)^2 + (n+l+1)^2 \\
	& \notag \leq ((n+1) + (n+l+1))^2. 
\end{align}
Hence, 
\begin{equation*} 
	\Big(2(n+1)(n+l+1)\Big)^{\frac{9 + \alpha}{2}}  \leq \Big((n+1) + (n + l + 1)\Big)^{9 + \alpha} 
\end{equation*}
and so 
\begin{equation}\label{sub2}
	 \Big(2(n+1)(n+l+1)\Big)^{-(\frac{9+\alpha}{2})} \geq \Big((n+1) + (n+l+1)\Big)^{-(9 + \alpha)}
	 \approx (2n \, + \, l \, - \, 3)^{-9 - \alpha} .
 \end{equation}
 
Using (\ref{sub2}), 
\begin{align}
	\notag 
	\| g \|^2_{H^s_\alt}  \ = \ \||x_1|^3 \|^2_{H^s_\alt}& \lesssim  \sum_{l, n} (n+1)^{s+\alt}(n+l+1)^{s + \alt}(n+1)^{-(\frac{9+\alpha}{2})}(n+l+1)^{-(\frac{9+\alpha}{2})}\\
	\label{finsum} & = \sum_{l, n} (n+1)^{s-\frac{9}{2}} (n+l+1)^{s-\frac{9}{2}}.
\end{align} 

Finally, to test the convergence of the sum in (\ref{finsum}), we apply the two-dimensional version of the integral test. 
For $n_0, l_0$ sufficiently large, 
\begin{align*}
\||x_1|^3\|^2_{H^s_{\alt}} &\lesssim  \sum_{l, n} (n+1)^{s-\frac{9}{2}} (n+l+1)^{s-\frac{9}{2}}\\
& \lesssim \int_{n = n_0 }^\infty \int_{l = l_0}^\infty (n+1)^{s-\frac{9}{2}}(n+l+1)^{s-\frac{9}{2}} \, dl \, dn \\
& = \int_{n = n_0}^\infty (n+1)^{s-\frac{9}{2}}\begin{cases} \frac{1}{s-\frac{7}{2}}(n+l+1)^{s-\frac{7}{2}}\Big|_{l = l_0}^\infty, & \text{if } s\neq \frac{7}{2},\\ \log(n+l+1), & \text{if } s = \frac{7}{2},\end{cases}\ dn\\
(\text{assuming } s< \frac{7}{2})\ & \approx \int_{n = n_0 }^\infty (n+1)^{s-\frac{9}{2}}(n+l_0+1)^{s-\frac{7}{2}} \, dn\\
& \approx \int_{n = n_0}^\infty (n+1)^{2s - 8} \, dn \\
& = \frac{1}{2s-7}(n+1)^{2s-7}\Big|_{n=n_0}^\infty < \infty, \ \mbox{ as } s < \frac{7}{2} \, .
\end{align*}

\par So, $u(x) = |x_1|^3 + x_2 \in H^s_\alt(\Omega)$ for any $s < \frac{7}{2}$. 
\mbox{ } \hfill \qed

\subsection{Optimal scaling for $\Delta t$ and $R$}
\label{ssec_optS}
Of interest is to determine $p$ such that the scaling 
\begin{equation}
    \Delta t \, \sim \,  \left( \frac{1}{R} \right)^{p} 
\label{scal1}
\end{equation}
results in the fastest convergence rate for the error in the approximation of the time dependent problem.

Using \eqref{scal1}, the terms on the RHS of Corollary \ref{cor1} have the orders of convergence
\begin{equation}
 \sim \   \left( \frac{1}{R} \right)^{2 p} , \,  \left( \frac{1}{R} \right)^{r} , \,  \left( \frac{1}{R} \right)^{r + \alt - 2p} , \,  \left( \frac{1}{R} \right)^{r - \alt} , \
 \mbox{ respectively}.
\label{scal2}
\end{equation}

To maximize the convergence with respect to $R$ we consider the two cases: (i) $2p \, = \, r + \alt - 2p$, and (ii) $2p \, = \, r - \alt$.

\underline{Case(i)}: For $2p \, = \, r + \alt - 2p \ \Rightarrow \ p \, = \, (r + \alt)/4 $, \eqref{scal2} becomes
\[
 \sim \   \left( \frac{1}{R} \right)^{(r + \alt)/2} , \,  \left( \frac{1}{R} \right)^{r} , \,  \left( \frac{1}{R} \right)^{(r + \alt)/2} , \,  
 \left( \frac{1}{R} \right)^{r - \alt} \, .
\]
Note that $(r + \alt)/2 \, = \, r - \alt \ \Rightarrow \ r = \frac{3 \alpha}{2} $.
\begin{align}
\mbox{Hence, if }  \ \  & r \geq \frac{3 \alpha}{2} \ \min \Big\{ (r + \alt)/2 \, , \, r - \alt \Big\} \ = \  (r + \alt)/2  \, ,   \label{scal3}  \\
\mbox{and if }    \ \  & r < \frac{3 \alpha}{2} \ \min \Big\{ (r + \alt)/2 \, , \, r - \alt \Big\} \ = \  r - \alt  \, .   \label{scal4} 
\end{align}

\underline{Case(ii)}: For $2p \, = \, r - \alt \ \Rightarrow \ p \, = \, (r - \alt)/2 $, \eqref{scal2} becomes
\[
 \sim \   \left( \frac{1}{R} \right)^{r - \alt} , \,  \left( \frac{1}{R} \right)^{r} , \,  \left( \frac{1}{R} \right)^{\alpha} , \,  
 \left( \frac{1}{R} \right)^{r - \alt} \, .
\]
Note that $r - \alt \, = \, \alpha \ \Rightarrow \ r = \frac{3 \alpha}{2} $.
\begin{align}
\mbox{Hence, if }  \ \  & r \geq \frac{3 \alpha}{2} \ \min \Big\{ r - \alt \, , \, \alpha \Big\} \ = \  \alpha  \, ,   \label{scal5}  \\
\mbox{and if }    \ \  & r < \frac{3 \alpha}{2} \ \min \Big\{ r - \alt \, , \, \alpha \Big\} \ = \  r - \alt  \, .   \label{scal6} 
\end{align}

Note that for $ r \, > \, \frac{3 \alpha}{2}$ then $(r + \alt)/2 \, > \, \alpha$.

Comparing \eqref{scal3} - \eqref{scal6} we see that the optimal value for $p$ is given by Case (i), 
\begin{equation}
p \, = \, (r + \alt)/4 \, .
\label{scal7}
\end{equation}


\begin{thebibliography}{10}

\bibitem{aco171}
G.~Acosta and J.P. Borthagaray.
\newblock A fractional {L}aplace equation: {R}egularity of solutions and finite
  element approximations.
\newblock {\em SIAM J. Numer. Anal.}, 55(2):472--495, 2017.

\bibitem{ain171}
M.~Ainsworth and C.~Glusa.
\newblock Aspects of an adaptive finite element method for the fractional
  {L}aplacian: a priori and a posteriori error estimates, efficient
  implementation and multigrid solver.
\newblock {\em Comput. Methods Appl. Mech. Engrg.}, 327:4--35, 2017.

\bibitem{ant223}
H.~Antil, T.~Brown, R.~Khatri, A.~Onwunta, D.~Verma, and M.~Warma.
\newblock Optimal control, numerics, and applications of fractional {PDE}s.
\newblock In {\em Numerical control. {P}art {A}}, volume~23 of {\em Handb.
  Numer. Anal.}, pages 87--114. North-Holland, Amsterdam, 2022.

\bibitem{ant202}
H.~Antil, Z.W. Di, and R.~Khatri.
\newblock Bilevel optimization, deep learning and fractional {L}aplacian
  regularization with applications in tomography.
\newblock {\em Inverse Problems}, 36(6):064001, 22, 2020.

\bibitem{bab011}
I.~Babu\v{s}ka and B.~Guo.
\newblock Direct and inverse approximation theorems for the {$p$}-version of
  the finite element method in the framework of weighted {B}esov spaces. {I}.
  {A}pproximability of functions in the weighted {B}esov spaces.
\newblock {\em SIAM J. Numer. Anal.}, 39(5):1512--1538, 2001/02.

\bibitem{bae102}
B.~Baeumer and M.M. Meerschaert.
\newblock Tempered stable {L}évy motion and transient super-diffusion.
\newblock {\em J. Comput. Appl. Math.}, 233(10):2438--2448, 2010.

\bibitem{bog031}
K.~Bogdan, K.~Burdzy, and Z.-Q. Chen.
\newblock Censored stable processes.
\newblock {\em Probab. Theory Related Fields}, 127(1):89--152, 2003.

\bibitem{bon191}
A.~Bonito, W.~Lei, and J.E. Pasciak.
\newblock Numerical approximation of the integral fractional {L}aplacian.
\newblock {\em Numer. Math.}, 142(2):235--278, 2019.

\bibitem{bre081}
S.C. Brenner and L.R. Scott.
\newblock {\em The mathematical theory of finite element methods}, volume~15 of
  {\em Texts in Applied Mathematics}.
\newblock Springer, New York, third edition, 2008.

\bibitem{bua101}
A.~Buades, B.~Coll, and J.M. Morel.
\newblock Image denoising methods. {A} new nonlocal principle.
\newblock {\em SIAM Rev.}, 52(1):113--147, 2010.
\newblock Reprint of ``A review of image denoising algorithms, with a new one''
  [MR2162865].

\bibitem{bue141}
A.~Bueno-Orovio, D.~Kay, V.~Grau, B.~Rodriguez, and K.~Burrage.
\newblock Fractional diffusion models of cardiac electrical propagation reveal
  structural heterogeneity effects on dispersion of repolarization.
\newblock {\em J. Roy. Soc. Interface}, 11(97):20140352, 2014.

\bibitem{bur211}
J.~Burkardt, Y.~Wu, and Y.~Zhang.
\newblock A unified meshfree pseudospectral method for solving both classical
  and fractional {PDE}s.
\newblock {\em SIAM J. Sci. Comput.}, 43(2):A1389--A1411, 2021.

\bibitem{del041}
D.~del Castillo-Negrete, B.A. Carreras, and V.E. Lynch.
\newblock Fractional diffusion in plasma turbulence.
\newblock {\em Phys. Plasmas}, 11(8):3854--3864, 2004.

\bibitem{del131}
M.~D'Elia and M.~Gunzburger.
\newblock The fractional {L}aplacian operator on bounded domains as a special
  case of the nonlocal diffusion operator.
\newblock {\em Comput. Math. Appl.}, 66(7):1245--1260, 2013.

\bibitem{duo182}
S.~Duo, H.W. van Wyk, and Y.~Zhang.
\newblock A novel and accurate finite difference method for the fractional
  {L}aplacian and the fractional {P}oisson problem.
\newblock {\em J. Comput. Phys.}, 355:233--252, 2018.

\bibitem{duo181}
S.~Duo and Y.~Zhang.
\newblock Computing the ground and first excited states of the fractional
  {S}chr\"odinger equation in an infinite potential well.
\newblock {\em Commun. Comput. Phys.}, 18(2):321--350, 2015.

\bibitem{duo191}
S.~Duo and Y.~Zhang.
\newblock Accurate numerical methods for two and three dimensional integral
  fractional {L}aplacian with applications.
\newblock {\em Comput. Methods Appl. Mech. Engrg.}, 355:639--662, 2019.

\bibitem{dyd171}
B.~Dyda, A.~Kuznetsov, and M.~Kwa\'{s}nicki.
\newblock Fractional {L}aplace operator and {M}eijer {G}-function.
\newblock {\em Constr. Approx.}, 45(3):427--448, 2017.

\bibitem{epp181}
B.P. Epps and B.~Cushman-Roisin.
\newblock Turbulence {M}odeling via the {F}ractional {L}aplacian.
\newblock arxiv preprint arxiv: 1803.05286, 2018.

\bibitem{ern041}
A.~Ern and J.-L. Guermond.
\newblock {\em Theory and practice of finite elements}, volume 159 of {\em
  Applied Mathematical Sciences}.
\newblock Springer-Verlag, New York, 2004.

\bibitem{ern211}
A.~Ern and J.-L. Guermond.
\newblock {\em Finite elements {III} -- {F}irst-order and time-dependent
  {PDE}s}, volume~74 of {\em Texts in Applied Mathematics}.
\newblock Springer, Cham, 2021.

\bibitem{erv241}
V.J. Ervin.
\newblock Equivalence of the weighted fractional {S}obolev space on a disk with
  characterization by the decay rate of {F}ourier-{J}acobi coefficients and
  {$K$}-interpolation.
\newblock {\em J. Fourier Anal. Appl.}, 30(6):Paper No. 71, 27, 2024.

\bibitem{erv251}
V.J. Ervin.
\newblock A variable diffusivity fractional {L}aplacian.
\newblock {\em J. Math. Anal. Appl.}, 547(1):Paper No. 129283, 23, 2025.

\bibitem{fai201}
G.~Failla and M.~Zingales.
\newblock Advanced materials modelling via fractional calculus: challenges and
  perspectives.
\newblock {\em Philosophical Transactions of the Royal Society A: Mathematical,
  Physical and Engineering Sciences}, 378(2172):20200050, 05 2020.

\bibitem{gil081}
G.~Gilboa and S.~Osher.
\newblock Nonlocal operators with applications to image processing.
\newblock {\em Multiscale Model. Simul.}, 7(3):1005--1028, 2008.

\bibitem{hao251}
Z.~Hao, Z.~Cai, and Z.~Zhang.
\newblock Fractional-order dependent radial basis functions meshless methods
  for the integral fractional {L}aplacian.
\newblock {\em Comput. Math. Appl.}, 178:197--213, 2025.

\bibitem{hao211}
Z.~Hao, H.~Li, Z.~Zhang, and Z.~Zhang.
\newblock Sharp error estimates of a spectral {G}alerkin method for a
  diffusion-reaction equation with integral fractional {L}aplacian on a disk.
\newblock {\em Math. Comp.}, 90(331):2107--2135, 2021.

\bibitem{hao212}
Z.~Hao, Z.~Zhang, and R.~Du.
\newblock Fractional centered difference scheme for high-dimensional integral
  fractional {L}aplacian.
\newblock {\em J. Comput. Phys.}, 424:Paper No. 109851, 17, 2021.

\bibitem{hua141}
Y.~Huang and A.~Oberman.
\newblock Numerical methods for the fractional {L}aplacian: a finite
  difference--quadrature approach.
\newblock {\em SIAM J. Numer. Anal.}, 52(6):3056--3084, 2014.

\bibitem{jav131}
M.~Javanainen, H.~Hammaren, L.~Monticelli, J.-H. Jeon, M.S. Miettinen,
  H.~Martinez-Seara, R.~Metzler, and I.~Vattulainen.
\newblock Anomalous and normal diffusion of proteins and lipids in crowded
  lipid membranes.
\newblock {\em Faraday Discuss.}, 161:397--417, 2013.

\bibitem{kla051}
J.~Klafter and I.M Sokolov.
\newblock Anomalous diffusion spreads its wings.
\newblock {\em Phys. World}, 18(8):29--32, 2005.

\bibitem{kwa171}
M.~Kwa\'{s}nicki.
\newblock Ten equivalent definitions of the fractional {L}aplace operator.
\newblock {\em Fract. Calc. Appl. Anal.}, 20(1):7--51, 2017.

\bibitem{las001}
N.~Laskin.
\newblock {Fractional quantum mechanics and Levy paths integrals}.
\newblock {\em Phys. Lett. A}, 268:298--305, 2000.

\bibitem{leh161}
R.B. Lehoucq and S.T. Rowe.
\newblock A radial basis function {G}alerkin method for inhomogeneous nonlocal
  diffusion.
\newblock {\em Comput. Methods Appl. Mech. Engrg.}, 299:366--380, 2016.

\bibitem{li141}
H.~Li and Y.~Xu.
\newblock Spectral approximation on the unit ball.
\newblock {\em SIAM J. Numer. Anal.}, 52(6):2647--2675, 2014.

\bibitem{lin021}
N.~Lindemulder and E.~Lorist.
\newblock Stein interpolation for the real interpolation method.
\newblock {\em Banach J. Math. Anal.}, 16(1):Paper No. 7, 18, 2022.

\bibitem{min201}
V.~Minden and L.~Ying.
\newblock A simple solver for the fractional {L}aplacian in multiple
  dimensions.
\newblock {\em SIAM J. Sci. Comput.}, 42(2):A878--A900, 2020.

\bibitem{qi121}
F.~Qi and Q.-M. Luo.
\newblock Bounds for the ratio of two gamma functions---from {W}endel's and
  related inequalities to logarithmically completely monotonic functions.
\newblock {\em Banach J. Math. Anal.}, 6(2):132--158, 2012.

\bibitem{shl871}
M.F. Shlesinger, B.J. West, and J.~Klafter.
\newblock L\'evy dynamics of enhanced diffusion: {A}pplication to turbulence.
\newblock {\em Phys. Rev. Lett.}, 58(11):1100--1103, 1987.

\bibitem{tia161}
X.~Tian, Q.~Du, and M.~Gunzburger.
\newblock Asymptotically compatible schemes for the approximation of fractional
  {L}aplacian and related nonlocal diffusion problems on bounded domains.
\newblock {\em Adv. Comput. Math.}, 42(6):1363--1380, 2016.

\bibitem{wan111}
H.~Wang and K.~Wang.
\newblock An {$O(N\log^2N)$} alternating-direction finite difference method for
  two-dimensional fractional diffusion equations.
\newblock {\em J. Comput. Phys.}, 230(21):7830--7839, 2011.

\bibitem{xu181}
K.~Xu and E.~Darve.
\newblock Spectral method for the fractional {L}aplacian in 2d and 3d.
\newblock Preprint: \verb+http://arxiv.org/abs/1812.08325+, 2018.

\bibitem{zen141}
F.~Zeng, F.~Liu, C.~Li, K.~Burrage, I.~Turner, and V.~Anh.
\newblock A {C}rank-{N}icolson {ADI} spectral method for a two-dimensional
  {R}iesz space fractional nonlinear reaction-diffusion equation.
\newblock {\em SIAM J. Numer. Anal.}, 52(6):2599--2622, 2014.

\bibitem{zha121}
Y.~Zhang, M.M. Meerschaert, and A.~Packman.
\newblock Linking fluvial bed sediment transport across scales.
\newblock {\em Geophys. Res. Lett.}, 39(20), 2012.

\bibitem{zhe251}
X.~Zheng, V.J. Ervin, and H.~Wang.
\newblock An anomalous fractional diffusion operator.
\newblock {\em Fract. Calc. Appl. Anal.}, 28(3):1198--1228, 2025.

\bibitem{zho241}
S.~Zhou and Y.~Zhang.
\newblock A novel and simple spectral method for nonlocal {PDE}s with the
  fractional {L}aplacian.
\newblock {\em Comput. Math. Appl.}, 168:133--147, 2024.

\end{thebibliography}

\end{document}